\documentclass[12pt,chapterheads]{ucsd}
\usepackage{math-nate}
\usepackage{txfonts}
\usepackage{url}
\usepackage{graphicx}
\usepackage{ifpdf}
\ifpdf
\fi

\numberwithin{equation}{section}
\usepackage{paralist}
\usepackage{enumitem}
\usepackage{cancel}
\usepackage{verbatim} 
\usepackage{hyphenat}

\usepackage{rcs}
\RCS $Id: dissertation.tex,v 1.35 2009/05/20 00:29:38 nate Exp $
\usepackage{datetime}

\usepackage{makeidx}
\makeindex

\usepackage[colorlinks=true, pdfstartview=FitV, linkcolor=black, citecolor=black, urlcolor=black,plainpages=false,pdfpagelabels]{hyperref}
\hypersetup{ pdfauthor = {Nathaniel Eldredge}, pdftitle =
  {Hypoelliptic heat kernel inequalities on H-type groups},
  pdfkeywords = { }, pdfcreator = {pdfLaTeX with hyperref package}, pdfproducer = {pdfLaTeX}}

\newcounter{continuehere}

\newcommand{\pp}[1]{\frac{\partial}{\partial #1}}

\newcommand{\grad}{\nabla}
\newcommand{\evalat}[2]{\bigl. #1 \bigr|_{#2}\rule{0pt}{3ex}}
\renewcommand{\Re}{\operatorname{Re}}
\renewcommand{\Im}{\operatorname{Im}}
\newcommand{\Res}{\operatorname{Res}}

\newcommand{\ad}{\operatorname{ad}}
\newcommand{\spanop}{\operatorname{span}}
\newcommand{\inner}[2]{\left\langle #1 , #2 \right\rangle}
\newcommand{\innerp}[2]{\left( #1, #2 \right)}
\newcommand{\hinner}[2]{\innerp{#1}{#2}}
\newcommand{\ncinner}[2]{\innerp{#1}{#2}} 
\newcommand{\definenoindex}[1]{\textbf{#1}}
\newcommand{\define}[2]{\definenoindex{#1}\index{#2|textbf}}
\newcommand{\Ric}{\operatorname{Ric}}
\newcommand{\Lie}{\operatorname{Lie}}
\newcommand{\Exp}{\operatorname{Exp}}
\newcommand{\End}{\operatorname{End}}
\newcommand{\supp}{\operatorname{supp}}
\newcommand{\csch}{\operatorname{csch}}
\newcommand{\funcclass}{\mathcal{C}}
\newcommand{\cylinder}{D}
\newcommand{\haar}{m}
\newcommand{\groupop}{\star}
\newcommand{\diffat}[2]{\left.\frac{d}{d #1}\right\vert_{#1 = #2}}
\newcommand{\downto}{\downarrow}
\newcommand{\asympon}[1]{\overset{#1}\asymp}

\newcommand{\heatcite}{Eldredge, Nathaniel, ``Precise Estimates for the Subelliptic
    Heat Kernel on H-type Groups,'' to appear, \emph{Journal de
      Math\'ematiques Pures et Appliqu\'ees}, 2009}
\newcommand{\gradcite}{Eldredge, Nathaniel, ``Gradient Estimates for the Subelliptic
    Heat Kernel on H-type Groups,'' submitted, \emph{Journal of
      Functional Analysis}, 2009}

\theoremstyle{remark}
\newtheorem*{remark*}{Remark}

\newcommand{\ignorethis}[1]{}

\newcommand{\polar}{\gamma}
\newcommand{\fixnotme}[1]{} 

\begin{document}

\title{Hypoelliptic heat kernel inequalities on H-type groups}

\author{Nathaniel Gilbert Bartsch Eldredge}
\degreeyear{2009}
\degreetitle{Doctor of Philosophy} 

\field{Mathematics}
\chair{Professor Bruce K. Driver}
\othermembers{
Professor George M. Fuller\\ 
Professor Benjamin Grinstein\\
Professor J. William Helton\\
Professor Linda P. Rothschild\\
}
\numberofmembers{5} 

\begin{frontmatter}
\makefrontmatter 

%
%

%
%


\begin{epigraph} 



   \begin{quotation}
     \textit{Pereant qui ante nos nostra dixerunt.}
   \end{quotation}
   ---Aelius Donatus

   (Quoted by St. Jerome, his pupil)

   \bigskip
   \begin{quotation}
     I wish to God these calculations had been executed by steam.
   \end{quotation}
   ---Charles Babbage

\end{epigraph}


\tableofcontents

\listoffigures  


\begin{acknowledgements} 
\fixnotme{Consider a Creative Commons or other free license for the
  dissertation.  Should be fine with publishers; Elsevier says ``may use in
  thesis'', which would be copyright by me in any case.}

Through the long and occasionally rocky process of writing this
dissertation, the support of many of my colleagues, friends and family
has been invaluable.  First and foremost among these is my advisor,
Bruce Driver, who has been throughout my graduate education
an outstanding teacher, colleague, and mentor; who has generously
shared his time and experience; and whose interest, patience, and
motivation throughout my graduate education has been above and beyond
the call of duty.  I would also like to thank the members of my
dissertation committee, George Fuller, Benjamin Grinstein, Bill Helton,
and Linda Rothschild for their kind and patient service.

Much credit is due to my parents, Gilbert and Lisa Eldredge, for their
lifelong love, encouragement and support in uncountable ways as I have
found my way into mathematics as an interest, a discipline, and now as
a profession; likewise to my brother, Jordan Eldredge, whose parallel
path in other fields has been an inspiration over the years.  And to
all the friends whose companionship through some of these challenging
years has kept me going, especially Andy Niedermaier, Chelsea Hollow,
Eric Tressler, Kristin Jehring, Kelly Jacobson, and Ross Richardson.

My thanks to the dedicated faculty and staff and my fellow students at
UCSD, at Harvey Mudd College, and throughout my schooling, with whom I
learned what mathematics is and what it is to me.  In particular, to
Michael Orrison of Harvey Mudd College, my undergraduate thesis
advisor, who introduced me to the world of research.  To the members
and coordinators of the UCSD Graduate Student Growth Group for mutual
assistance with non-mathematical obstacles.  And finally, to everyone
else who has had a hand in my development as a mathematician thus far.
Even if I could have done it without you, I wouldn't have wanted to.

In financial terms, I would like to acknowledge the support of the
UCSD Department of Mathematics, and of the National Science
Foundation, by whom this research was funded in part through NSF
Grants DMS-0504608 and DMS-0804472, as well as an NSF Graduate
Research Fellowship.

This dissertation was typeset with \LaTeX{} using the \texttt{TX}
fonts package; the primary text font is Nimbus Roman No9.  My thanks
to the authors and maintainers of the UCSD \LaTeX{} dissertation
class and template for making the formatting process as smooth as
possible.  The figures were prepared using \texttt{gnuplot} and MetaPost.


Portions of Chapters \ref{h-type-chapter},
\ref{subriemannian-chapter}, and \ref{heat-chapter} are adapted from
material awaiting publication as \heatcite.  The dissertation author
was the sole author of this paper.

Portions of Chapters \ref{h-type-chapter} and \ref{gradient-chapter}
are adapted from material awaiting publication as \gradcite.  The
dissertation author was the sole author of this paper.

\end{acknowledgements}

\begin{vitapage}
\begin{vita}
  \item[2003] B.~S. in Mathematics with High Distinction, Harvey Mudd College
  \item[2003--2008] Teaching Assistant, University of California, San Diego
  \item[2005] M.~A. in Mathematics, University of California, San
    Diego
  \item[2007--2009] Research Assistant, University of California, San Diego
  \item[2008, 2009] Associate Instructor in Mathematics, University of California,
    San Diego
  \item[2009] Ph.~D. in Mathematics, University of California, San Diego 
\end{vita}
\begin{publications}
  \item \heatcite.
  \item \gradcite.
\end{publications}
\end{vitapage}


\begin{abstract}
  We study inequalities related to the heat kernel for the
  hypoelliptic sublaplacian on an H-type Lie group.  Specifically, we
  obtain precise pointwise upper and lower bounds on the heat kernel
  function itself.  We then apply these bounds to derive an estimate
  on the gradient of solutions of the heat equation, which is known to
  have various significant consequences including logarithmic Sobolev
  inequalities.  We also present a computation of the heat kernel, and
  a discussion of the geometry of H-type groups including their
  geodesics and Carnot-Carath\'eodory distance functions.
\end{abstract}
\end{frontmatter}



\chapter{Introduction}\label{intro-chapter}

\section{Two trivial examples}

In recent years, there has been considerable interest in the study of
hypoelliptic operators and associated problems.  The purpose of this
dissertation is to address two specific questions, regarding estimates
for the heat kernel, in the context of H-type Lie groups.

To introduce and motivate these problems, we begin with a simpler
example.

\begin{example}\label{R3-example}\index{R3@$\R^3$}
  Let
\begin{equation*}
  \Delta^{(3)} := \left(\pp{x}\right)^2 + \left(\pp{y}\right)^2 + \left(\pp{z}\right)^2
\end{equation*}
be the usual Laplacian \index{Laplacian} on $\R^3$.  (We decorate various objects in
this example with the superscript ${}^{(3)}$ to contrast with examples to come.)
Consider the Cauchy problem for the associated heat equation:
\begin{equation}\label{R3-cauchy}
  \begin{split}
    \left(\Delta^{(3)} - \pp{t}\right) u (t,\vec{x}) &= 0 \quad \text{for all
      $t > 0$, $\vec{x} \in \R^3$}
    \\
    u(0, \vec{x}) &= f(\vec{x}) \quad \text{for all $\vec{x} \in \R^3$}
  \end{split}
\end{equation}
where, for instance, $f \in C_c(\R^3)$ (i.e. $f$ is an continuous
real-valued function on $\R^3$ with compact support).  Then
(\ref{R3-cauchy}) has a unique bounded solution $u(t,\vec{x})$ (see
\cite{evans-pde-book}).  To emphasize the dependence on the initial
condition $f$, we write $u(t, \vec{x}) = P_t^{(3)} f(\vec{x})$.
($P_t^{(3)}$ is really the heat semigroup \index{heat semigroup} $P_t^{(3)} =
e^{t\Delta^{(3)}}$.)  It is well-known that $P_t^{(3)}$ has a
convolution kernel, which is the heat kernel \index{heat kernel!Euclidean} $p_t^{(3)}$, i.e.
\begin{equation}\label{R3-convolution}
  P_t^{(3)} f(\vec{x}) = f(\vec{x}) * p_t^{(3)} = \int_{\R^3}
  f(\vec{x} - \vec{u})
    p_t^{(3)} (\vec{u})\,d\vec{u}
\end{equation}
where
\begin{equation}\label{R3-pt-formula}
  p_t^{(3)} (\vec{x}) = \frac{1}{(4\pi t)^{3/2} }
  e^{-\frac{1}{4t}\abs{\vec{x}}^2}.
\end{equation}
$p_t^{(3)}$ can also be viewed as the fundamental solution to
(\ref{R3-cauchy}), with initial condition a delta distribution
supported at the origin, i.e. $p_t^{(3)} = P_t^{(3)} \delta_0$.

A useful interpretation of the heat equation (\ref{R3-cauchy}) is as a
model for diffusion.\index{diffusion}  Imagine that $\R^3$ is filled
with air, which is contaminated unevenly with perfume.\index{perfume|see{diffusion}}  If $f(\vec{x})$ represents the
concentration of perfume at the point $\vec{x}$ at time $t=0$, then
the solution $u(t,\vec{x})$ gives the concentration at later times $t$
as the perfume diffuses throughout $\R^3$.  The heat kernel
$p_t^{(3)}(\vec{x})$ corresponds to the concentration following an
initial configuration with a unit mass of perfume concentrated at the
origin.  This model can also be viewed probabilistically, if the
perfume is considered to consist of a large number of tiny particles
moving randomly through $\R^3$ (specifically, moving according to
Brownian motion; see Figure \ref{bm3-fig}).  Then $p_t^{(3)}(\vec{x})$ gives the probability
density that a particle which begin at the origin at time $t=0$ winds
up ``near'' the point $\vec{x}$ at time $t$.  It is then plausible
that $p_t^{(3)}$ should be a smooth function, and $P_t^{(3)}$ a
``smoothing'' operator, since particles will immediately spread out
from any area where they may be highly concentrated.
\begin{figure}
  \begin{center}
    \includegraphics[height=3in]{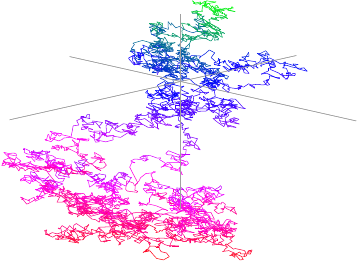}
    \end{center}
  \caption[Brownian motion in $\R^3$]{Brownian motion in $\R^3$.  The colors vary with the $z$ coordinate.}\label{bm3-fig}
\end{figure}

For comparison with following examples, we note the trivial fact that
the vector fields $\left\{\pp{x}, \pp{y}, \pp{z}\right\}$, the sum of
whose squares gives $\Delta^{(3)}$, are translation invariant with
respect to the usual vector addition $+$ on $\R^3$ (by the chain
rule), and the same is true for $\Delta^{(3)}$.  That is, if
$L_\vec{u}(\vec{x}) = \vec{u} + \vec{x}$ denotes translation by
$\vec{u} \in \R^3$, we have by the chain rule that $\Delta^{(3)}(f
\circ L_\vec{u}) = (\Delta^{(3)}f) \circ L_\vec{u}$.  Moreover, these
vector fields have the property that they span the tangent space to
$\R^3$ at every point; every differentiable curve is tangent at each
point to some linear combination of $\left\{\pp{x}, \pp{y},
\pp{z}\right\}$.  For each fixed $t > 0$, the heat kernel behaves at
infinity like $e^{-\frac{1}{4t} \abs{\vec{x}}^2}$, where the distance
$\abs{\vec{x}}$ can be interpreted as the length of the shortest path
from $\vec{0}$ to $\vec{x}$ which is tangent to the span of
$\left\{\pp{x}, \pp{y}, \pp{z}\right\}$ (namely, a straight line,
since the latter condition holds trivially).

\index{gradient bounds!Laplacian}
Another property of interest is that the gradient in $\vec{x}$ of a solution
$u = P_t^{(3)} f$ can be controlled in terms of the usual gradient
$\grad^{(3)}$ of the
initial condition $f$, and we have the inequality
\begin{equation}\label{grad-R3}
  \abs{\left(\grad^{(3)} P_t^{(3)} f\right)(\vec{x})} \le K(t) P_t^{(3)}
  \left(\abs{\grad^{(3)} f}\right)(\vec{x})
\end{equation}
for some constant $K(t)$ depending on $t$.  Indeed, by differentiating
under the integral sign we have
\begin{align*}
  \abs{\left(\grad^{(3)} P_t^{(3)} f\right)(\vec{x})} &=
  \abs{\grad^{(3)} \int_{\R^3}  f(\vec{x} - \vec{u})
    p_t^{(3)} (\vec{u})\,d\vec{u}} \\
  &= \abs{\int_{\R^3} \grad^{(3)} \left[ f(\vec{x} - \vec{u})
    p_t^{(3)} (\vec{u})\right]\,d\vec{u}} \\
   &= \abs{\int_{\R^3} \grad^{(3)}f(\vec{x} - \vec{u})
    p_t^{(3)} (\vec{u})\,d\vec{u}} \\
  &\le \int_{\R^3} \abs{\grad^{(3)}f(\vec{x} - \vec{u})}
    p_t^{(3)} (\vec{u})\,d\vec{u} \\
    &= P_t^{(3)}
  \left(\abs{\grad^{(3)} f}\right)(\vec{x})   
\end{align*}
so that (\ref{grad-R3}) holds with $K(t) \equiv 1$.

\end{example}

Much of the nice behavior of the operator $\Delta^{(3)}$ is related to
the fact that it is an elliptic \index{elliptic} operator (see Definition
\ref{elliptic-def}).  A particular consequence of this is ``elliptic
regularity,'' which guarantees that even weak solutions of
(\ref{R3-cauchy}) (in the sense of distributions) must actually be
smooth functions.  The following example shows what can go wrong if
this condition is relaxed too far.

\begin{example}\label{D2-example}\index{R3@$\R^3$}
  We continue to work in $\R^3$, and let
  \begin{equation*}
    \Delta^{(2)} := \left(\pp{x}\right)^2 + \left(\pp{y}\right)^2
  \end{equation*}
  be the Laplacian \index{Laplacian!degenerate} in the $x$ and $y$ variables only.  For $\Delta^{(2)}$,
  the Cauchy problem
  \begin{equation}\label{D2-cauchy}
    \begin{split}
      \left(\Delta^{(2)} - \pp{t}\right) u (t,\vec{x}) &= 0 \quad \text{for all
        $t > 0$, $\vec{x} \in \R^3$}
      \\
      u(0, \vec{x}) &= f(\vec{x}) \quad \text{for all $\vec{x} \in \R^3$}
    \end{split}
  \end{equation}
  reduces to a one-parameter family of Cauchy problems in $\R^2$,
  indexed by $z$.  In some sense, the degenerate operator
  $\Delta^{(2)}$ is not ``using'' all three dimensions.
  (\ref{D2-cauchy}) still has a unique bounded solution $P_t^{(2)} f$ given
  the initial values $f$, namely
  \begin{equation*}
    P_t^{(2)} f(x,y,z) = \iint_{\R^2} f(x-x', y-y', z) p_t^{(2)}(x', y')\,dx'\,dy'
  \end{equation*}
  where
  \begin{equation*}
    p_t^{(2)} (x,y) = \frac{1}{4\pi t}  e^{-\frac{1}{4t}(x^2+y^2)}.
  \end{equation*}
  However, it is obvious that no smoothness is imposed on the $z$
  dependence of $P_t^{(2)} f$.  Indeed, the fundamental solution of
  (\ref{D2-cauchy}) must be interpreted as the distribution $p_t^{(2)}
  (x,y) \delta_0(z)$.

  The vector fields $\left\{\pp{x}, \pp{y}\right\}$ obviously do not
  span the tangent space of $\R^3$ at any point, and paths which are
  everywhere tangent to this subspace must be horizontal, i.e. have
  constant $z$ coordinate.   $p_t^{(2)} (x,y)$ behaves at infinity
  like $e^{-\frac{1}{4t}\abs{(x,y)}^2}$, where the ``horizontal distance'' $\abs{(x,y)} =
  \sqrt{x^2+y^2}$ could be interpreted as the length of the shortest
  horizontal path from the origin to $(x,y)$.  Off the $x$-$y$ plane,
  this distance should be considered infinite.

  The inequality
  \index{gradient bounds!degenerate Laplacian}
  \begin{equation}\label{D2-grad}
    \abs{\left(\grad^{(2)} P_t^{(2)} f\right)(\vec{x})} \le K(t) P_t^{(2)}
    \left(\abs{\grad^{(2)} f}\right)(\vec{x})
  \end{equation}
  again holds with $K(t) \equiv 1$, if we take $\grad^{(2)} =
  \left(\pp{x}, \pp{y}\right)$ to be the gradient in the $x$ and $y$
  variables only.  However, this is really a one-parameter family of
  inequalities indexed by $z$, and provides very limited control over
  the behavior of the solution.  For instance, a function $f$
  satisfying $\grad^{(2)}f \equiv 0$ need not be constant on $\R^3$.

  In terms of diffusion, (\ref{D2-cauchy}) describes a system in which
  perfume diffuses horizontally, but not vertically.
  Probabilistically, particles move according to a two-dimensional
  ``horizontal'' Brownian motion, keeping their $z$ coordinate fixed.
  This can be viewed as a three-dimensional Brownian motion which has
  been ``constrained'' to only follow horizontal paths.  Of course, it
  is natural that concentration smoothing occurs within planes of
  constant $z$ coordinate, but not between them.
\end{example}

\section{One nontrivial example: the Heisenberg group}\label{H1-sec}

The foregoing examples represent two extremes of behavior.  This
dissertation focuses on a class of operators which occupy a middle
ground.

\begin{example}\label{H1-example}
  We work again on $\R^3$.  Let $X, Y$ be the vector fields
  \begin{equation}
    X = \pp{x} - \frac{1}{2} y \pp{z}, \quad Y = \pp{y} + \frac{1}{2}
    x \pp{z}
  \end{equation}
  and take $L$ to be the operator
  \begin{equation}
    L = X^2 + Y^2 = \left(\pp{x}\right)^2 + \left(\pp{y}\right)^2 +
    \left(x \pp{y} - y \pp{x}\right) \pp{z} + \frac{1}{4} \left(x^2 +
    y^2\right) \left(\pp{z}\right)^2.
  \end{equation}
  We immediately note that $X, Y, L$ are not translation invariant
  with respect to vector addition on $\R^3$.  However, if we equip
  $\R^3$ with the binary operation
  \begin{equation}
    (x,y,z) \groupop (x', y', z') = \left(x+x', y+y', z+z'+\frac{1}{2}(x y'
    - x' y)\right),
  \end{equation}
  then $(\R^3, \groupop)$ becomes a Lie group, known as the
  \define{Heisenberg group}{Heisenberg group} $\mathbb{H}_1$.  This is the
  prototype for the \emph{H-type groups} which are the subject of this
  dissertation.  Then $X, Y, L$ are invariant with respect to
  \emph{left} translation under $\groupop$ (or simply
  \definenoindex{left-invariant}), i.e. if $L_h(g) = h \groupop g$ is
  left translation by $h$, we have $L(f \circ L_h) = (Lf) \circ L_h$.
  (We shall begin using the letters $g, h, k$ instead of $\vec{x},
  \vec{u}$ to represent elements of $\mathbb{H}_1$, to emphasize its group
  structure, but shall not forget that $\mathbb{H}_1 = \R^3$ as a set and as a
  smooth manifold.)

  We again consider the Cauchy problem
  \begin{equation}\label{H1-cauchy}
    \begin{split}
      \left(L - \pp{t}\right) u (t,g) &= 0 \quad \text{for all
        $t > 0$, $g \in \mathbb{H}_1$}
      \\
      u(0, g) &= f(g) \quad \text{for all $g \in \mathbb{H}_1$}.
    \end{split}
  \end{equation}
  As before, a unique bounded solution exists, given by
  \begin{equation*}
    P_t f(g) = \int_{\mathbb{H}_1} f(g \groupop k^{-1}) p_t(k)\,d\haar(k)
  \end{equation*}
  where the heat kernel $p_t$ is given by the more complicated formula
  \begin{equation}\label{H1-pt-formula}\index{heat kernel!formula}
    p_t(x,y,z) =  \frac{1}{2\pi} \int_{\R} e^{i \lambda z - \frac{1}{4}
    \lambda \coth(t \lambda) (x^2+y^2} \frac{\lambda}{4
  \pi \sinh(t
  \lambda)}\,d\lambda.
  \end{equation}
  Here $\haar$ is Lebesgue measure, which is also the Haar measure
  \index{Haar measure} for
  $\mathbb{H}_1$ as it is invariant under left and right translation.
  A derivation of a generalization of (\ref{H1-pt-formula}) appears in
  Section \ref{pt-sec}.

  The operator $L$ is not elliptic; the matrix of coefficients of
  second-order partials is
  \begin{equation*}
    Q(x,y,z) =
    \begin{pmatrix}
      1 & 0 & -\frac{1}{2} y \\
      0 & 1 & \frac{1}{2} x \\
      -\frac{1}{2}y & \frac{1}{2}x & \frac{1}{4}(x^2+y^2)
    \end{pmatrix}
  \end{equation*}
  which is easily shown to be positive semidefinite but degenerate for
  all $(x,y,z)$.  However, the heat kernel $p_t$ is actually smooth,
  and hence so are bounded solutions to the Cauchy problem
  (\ref{H1-cauchy}).  Thus the operator $L$ retains some regularity;
  specifically, it is hypoelliptic \index{hypoelliptic} (see
  Definition \ref{hypoelliptic-def}).

  The reason for this regularity is related to the following
  observation.  Although the vector fields $\{X,Y\}$ clearly do not
  span the tangent space of $\mathbb{H}_1 = \R^3$ at any point (since
  there are only two of them), yet their Lie bracket $[X,Y] := XY - YX
  = \pp{z} =: Z$ is linearly independent of $\{X,Y\}$, so that $\{X, Y, Z\}$
  \emph{does} span the tangent space at each point.  (The relations
  $[X,Y]=Z$ and $[X,Z]=[Y,Z]=0$, which define the
  \define{Heisenberg Lie algebra}{Heisenberg Lie algebra}, arise in quantum mechanics
  and are the reason for the use of the name of Heisenberg.)  By
  analogy with Example \ref{D2-example}, we call a path $\gamma(t) =
  (x(t), y(t), z(t))$ \definenoindex{horizontal} 
  \index{horizontal path}
  if it is tangent to some linear combination of $X,Y$ (but
  not $Z$) at each point.  See Figure \ref{planes-heis-fig}.  The
  relationship $[X,Y] = Z$ then suggests that a horizontal curve which
  travels a distance $\epsilon$ in the directions of $+X, +Y, -X, -Y$
  successively will make $\epsilon^2$ progress in the ``forbidden''
  $Z$ direction.  Thus it is plausible that, unlike in Example
  \ref{D2-example}, horizontal paths may be able to join arbitrary
  pairs of points of $\mathbb{H}_1$.
\begin{figure}
  \begin{center}
    \includegraphics[height=3in]{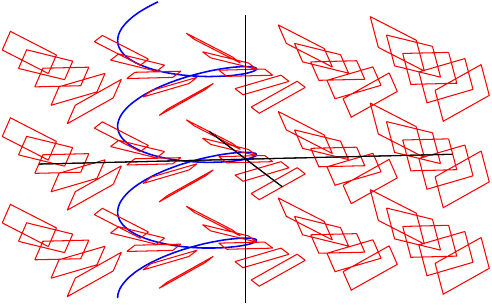}
    \end{center}
  \caption[Horizontal planes in $\mathbb{H}_1$]{Horizontal planes in $\mathbb{H}_1$.
    Each of the planes in red represents the two-dimensional subspace
    of the tangent space $T_g \mathbb{H}_1$ at a point $g$ which is spanned by
    $X(g), Y(g)$.  A sample horizontal curve, which is tangent at each
    point to the corresponding subspace, is shown in
    blue.}\label{planes-heis-fig}
\end{figure}

\index{diffusion}
  Our probabilistic diffusion model suggests how this relates to the
  regularity property of $L$.  Perfume particles should move according
  to a ``horizontal Brownian motion'' 
  \index{horizontal Brownian motion}
  that has been constrained to
  follow horizontal paths.  See Figure \ref{bm-heis-fig}.  (Our definition of ``horizontal'' as
  ``tangent to some linear combination of $X,Y$'' requires
  reinterpretation, since Brownian motion paths are nowhere
  differentiable.  This can be done in the language of stochastic
  calculus.\fixnotme{See Section \ref{prob-sec}.})  Since, unlike in
  Example \ref{D2-example}, horizontal paths do not remain stuck in a
  submanifold, but are able to reach arbitrary points (see below), it
  seems reasonable that Brownian particles should be able to diffuse
  throughout space.  Locally, their motion is horizontal to first
  order, but also vertical to second order.  Thus high concentrations
  of perfume spread out in all directions, giving rise to a smoothing
  effect.
\begin{figure}
  \begin{center}
    \includegraphics[height=3in]{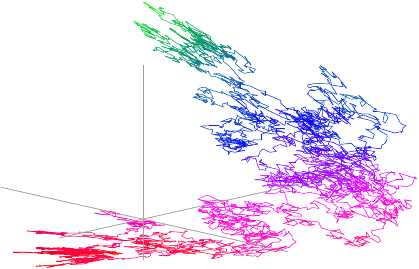}
    \end{center}
  \caption[Horizontal Brownian motion in $\mathbb{H}_1$]{Horizontal Brownian motion in $\mathbb{H}_1$.  The colors vary with the $z$ coordinate.}\label{bm-heis-fig}
\end{figure}

  Let us see what such paths look like.  It is clear that a horizontal
  path must satisfy $\dot{\gamma}(t) = \dot{x}(t) X(\gamma(t)) +
  \dot{y}(t) Y(\gamma(t))$, so that solving for $\dot{z}(t)$ and
  integrating we find
  \begin{equation}\label{H1-green}
    z(t) = z(0) + \frac{1}{2} \int_0^t x(t) (\dot{y}(t) - \dot{x}(t)
    y(t))\,dt.
  \end{equation}
  \index{Green's theorem}
  By Green's theorem, this says that $z(t) - z(0)$ is equal to the
  (signed) area enclosed by the two-dimensional curve $(x(t),y(t))$.
  (If the curve is not closed, one may close it by adjoining straight
  lines from the origin to $(x(0), y(0))$ and $(x(t), y(t))$, since
  such lines do not contribute to the integral in (\ref{H1-green}).
  See Figure \ref{signed-area-fig}.)
  \index{signed area}
  It is intuitively clear that one may connect any pair of points in
  $\R^2$ by a curve which encloses any prescribed signed area; thus
  any pair of points in $\mathbb{H}_1$ can indeed be joined by a horizontal
  path.
\begin{figure}
  \begin{center}
    \includegraphics[height=3in]{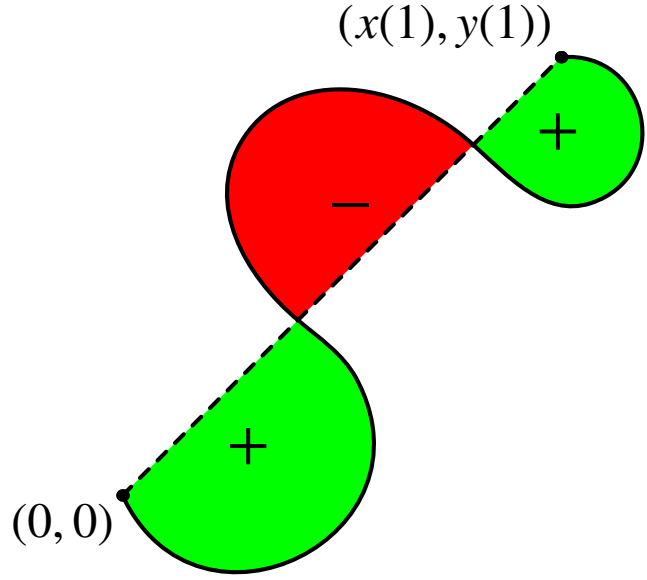}
    \end{center}
  \caption[Signed area]{Signed area, as used to describe horizontal
    curves in $\mathbb{H}_1$.  The plane curve from $(0,0)$ to
    $(x(1), y(1))$ (solid black) may be closed by the straight line
    back to the origin (dashed).  The areas thus enclosed contribute
    positively or negatively according to a ``right-hand rule'' based
    on the orientation of the curve.}\label{signed-area-fig}
\end{figure}

  The regularity of $L$ and the connectedness of $\mathbb{H}_1$ by horizontal
  paths are both strongly related to the fact that the vector fields
  $X,Y$, together with their bracket $[X,Y]=Z$, span the tangent space
  at each point.  This \emph{bracket generating condition} is what
  separates $\mathbb{H}_1$ from degenerate situations like Example
  \ref{D2-example}.  We shall say more about this condition in Section \ref{hypo-sec}.
  \index{bracket generating}

  If
  we consider the \definenoindex{length} of a horizontal path to be
  the length of its horizontal projection $(x(t), y(t))$ (because
  $(\dot{x}(t), \dot{y}(t))$ are the coefficients of the horizontal
  vector fields $X,Y$, which we may consider to be ``orthonormal''),
  the problem of finding the shortest horizontal curve joining two
  points is just the problem of finding the shortest plane curve
  enclosing a given area in the previous sense.  This is a classic
  problem in the calculus of variations\footnote{The problem is
    commonly called Dido's Problem after the legendary \cite{aeneid} queen of
    Carthage.  It seems she and her followers found themselves
    shipwrecked in North Africa after fleeing the murderous King
    Pygmalion of Tyre.  She
    pled with the local authorities for some land on which to settle,
    but was offered only as much land as she could cover
    with an ox-hide.  Interpreting the word ``cover'' creatively, she
    cut the hide into thin strips and used them to bound a large
    region of land, with the ocean as the other boundary.  On this
    prime waterfront property she founded the city of Carthage.  There
    she lived happily with her companions until the arrival of Aeneas
    from Troy, who caused for Dido an entirely different sort of
    problem \cite{dido}, less easily solved by mathematics.
  } whose solution is
  the arc of a circle.  Thus the ``horizontal distance'' (or
  \define{Carnot-Carath\'eodory distance}{Carnot-Carath\'eodory distance}) $d(g,h)$ between any
  two points $g,h \in \mathbb{H}_1$ is finite.
 \begin{figure}
  \begin{center}
    \includegraphics[height=3in]{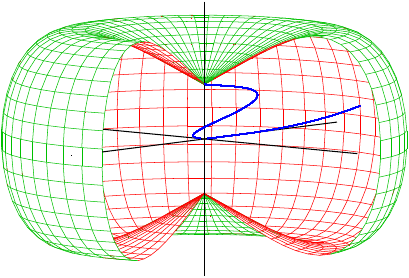}
    \end{center}
  \caption[The unit ball of $\mathbb{H}_1$]{The unit ball of $\mathbb{H}_1$, with
    respect to the Carnot-Carath\'eodory distance $d$.  A section has
    been removed to show geodesics 
    \index{geodesic}
    (shortest horizontal paths, blue)
    emanating from the origin.}\label{ball-fig}
\end{figure}
  It is this distance that, as in
  the previous examples, we might expect to describe the behavior of
  the heat kernel $p_t$ at infinity.  Indeed, it was shown in
  \cite{li-heatkernel} that
  \index{heat kernel bounds!Heisenberg group}
  \index{Heisenberg group!heat kernel bounds}
  \begin{equation}\label{H1-heat}
    \frac{C_1}{1+\sqrt{\abs{(x,y)}d(0,g)}} e^{-\frac{1}{4}d(0, g)^2}
    \le p_1(g) \le \frac{C_2}{1+\sqrt{\abs{(x,y)}d(0,g)}} e^{-\frac{1}{4}d(0, g)^2}
  \end{equation}
  for constants $C_1, C_2$. 
  Chapter \ref{heat-chapter} of this
  dissertation is concerned with extending estimates like
  (\ref{H1-heat}) to general H-type groups.

  \index{gradient bounds!Heisenberg group}
  Regarding the problem of gradient estimates, it was first shown in
  \cite{li-jfa} that there exists a constant $K$ (independent of $t$!)
  such that
  \begin{equation}\label{H1-grad}
    \abs{\grad P_t f} \le K P_t(\abs{\grad f})
  \end{equation}
  where $\grad = (X,Y)$ is the ``horizontal gradient.''  The proof
  makes extensive use of the heat kernel bounds (\ref{H1-heat}).
  (\ref{H1-grad}) is a considerably more informative statement than its
  analogue (\ref{D2-grad}) in Example \ref{D2-example}.  For instance,
  because $[X,Y]=Z$, we have that \mbox{$\grad f \equiv 0$} implies that $f$
  is constant.  Chapter \ref{gradient-chapter} of this dissertation
  follows a proof appearing in \cite{bbbc-jfa} to show that
  (\ref{H1-heat}) also holds for general H-type groups.
\end{example}

\section{Hypoelliptic operators and Lie groups}\label{hypo-sec}

As suggested in the previous section, some sort of regularity
condition is needed on an operator $L$ to avoid the degeneracy of
Example \ref{D2-example} without requiring ellipticity as in Example
\ref{R3-example}.  We therefore confine our attention to hypoelliptic
operators.

\begin{definition}\label{hypoelliptic-def}
  A partial differential operator $L$ on a manifold $M$ is said to be
  \define{hypoelliptic}{hypoelliptic} if, for every distribution $u$
  on $M$, $L u \in C^\infty(M)$ if and only if $u \in C^\infty(M)$.
\end{definition}

By standard elliptic regularity results, every elliptic operator is
hypoelliptic;
\index{elliptic regularity}
 its corresponding parabolic
heat operator is hypoelliptic as well.  See, for instance, Section 3.4 of
\cite{stroock-pde-probabilists}.  The Heisenberg sublaplacian is an example of a
hypoelliptic operator that is not elliptic.  More examples are
supplied by the following theorem, of which a simplified proof can be
found in Chapter 7 of \cite{stroock-pde-probabilists}.

\begin{theorem}[H\"ormander \cite{hormander67}]\label{hormander-thm}
\index{Hormander's theorem@H\"ormander's theorem}
  Let $X_1, \dots, X_n$ be smooth vector fields on a manifold $M$
  satisfying the following \define{bracket generating
    condition}{bracket generating}: for each $m \in M$ there
  exists an integer $r$ (the \define{rank}{rank} of $\{X_1, \dots,
  X_n\}$ at $m$) such that
  \begin{equation}\label{bracket-gen-condition}
    \begin{split}
    T_m M = \spanop\{X_{i_1}(m), [X_{i_1}, X_{i_2}](m), [X_{i_1}, [X_{i_2},
      X_{i_3}]](m), \dots : 1 \le i_1, \dots, i_r \le n \}
    \end{split}
  \end{equation}
  where at most $r-1$ brackets are taken.
  Let $Y$ be another smooth vector field on $M$.  Then the
  second-order operator $L := X_1^2 + \dots + X_n^2 + Y$ is
  hypoelliptic.
\end{theorem}

In particular, this implies that ``harmonic'' functions on $M$ (satisfying
$L f = 0$) are automatically smooth.  Also, if we replace $M$ by $M
\times (0, \infty) = \{ (m, t) : m \in M, t > 0\}$ and set $Y =
\pp{t}$, we see that solutions $u$ to the heat equation $\left(L -
\pp{t}\right)u = 0$ are also smooth functions of $m$ and $t$. 

The bracket generating condition (\ref{bracket-gen-condition})
requires that $L$ be built out of enough vector fields to fill out the
tangent space to $M$ at each point, when their brackets are included.
This serves to rule out situations like Example \ref{D2-example}.  In
particular, if $X_1, \dots, X_n$ satisfy
(\ref{bracket-gen-condition}), it is easy to see that if $X_i f \equiv 0$
for all $i$ (i.e. its ``gradient'' is identically zero) then $f$ must
be constant.

As in the case of the Heisenberg group, the bracket generating
condition is also related to a geometric fact about horizontal paths.

\begin{theorem}[Chow]\label{chow-thm}
\index{Chow's theorem}
  If $M$ is a connected manifold with vector fields $X_1, \dots, X_n$
  satisfying (\ref{bracket-gen-condition}), then any pair of points in
  $M$ can be joined by a path which is tangent at each point to some
  linear combination of $X_1, \dots, X_n$.
\end{theorem}

This also allows a reasonable Carnot-Carath\'eodory distance
\index{Carnot-Carath\'eodory distance} to be
defined, as in $\mathbb{H}_1$.  More will be said about this idea in
Section \ref{subriemannian-sec}.

H\"ormander's theorem supplies a very large class of hypoelliptic
operators; indeed, too large for present purposes.  It is difficult to
say much about an operator on a general smooth manifold without having
some structure on the manifold.  Lie groups provide such structure
while at the same time not giving up too much generality, as we shall
see.

\begin{example}
  \index{Lie group}
  \index{group!Lie|see{Lie group}}
  Let $G$ be a Lie group with group operation $\groupop$, and let
  $X_1, \dots, X_n$ be left-invariant vector fields on $G$.  The
  bracket-generating condition (\ref{bracket-gen-condition}) is then
  equivalent to the condition that the vector fields $\{X_1, \dots,
  X_n\}$ generate the Lie algebra $\mathfrak{g} = \Lie G$ of all left-invariant
  vector fields on $G$.  (Note that in this case, the rank of $\{X_1,
  \dots, X_n\}$ is the same at every point of $G$.)  The
  left-invariant operator $L =
  X_1^2 + \dots + X_n^2$, called a \define{sublaplacian}{sublaplacian}
  is thus hypoelliptic. \index{hypoelliptic}
\end{example}

\begin{definition}
  A Lie algebra $\mathfrak{g}$ is \define{nilpotent}{Lie
    algebra!nilpotent} of step $r$ if all $r$-fold Lie brackets
  vanish, i.e. 
  $[X_1, [X_2, \dots [X_{r}, X_{r+1}]\dots]] =0$ for all $X_1, \dots,
  X_r, X_{r+1} \in \mathfrak{g}$.  A nilpotent Lie algebra
  $\mathfrak{g}$ is \define{stratified}{Lie algebra!stratified} if
  there is a decomposition $\mathfrak{g} = V_1 \oplus \dots \oplus
  V_r$ such that $[V_1, V_i] = V_{i+1}$ for $1 \le i < r$ and $[V_1,
    V_r] = 0$.  A Lie group is nilpotent (respectively, stratified) if its Lie
  algebra is.
  \define{}{Lie group!nilpotent}
  \define{}{Lie group!stratified}
\end{definition}

A theorem of Rothschild and Stein \cite{rothschild-stein} states,
informally speaking, that a bracket-generating set of vector fields
$\{X_i\}$ on a manifold $M$ can be locally approximated in a
neighborhood of a point $m \in M$ by a bracket-generating set of
left-invariant vector fields $\{Y_i\}$ on some nilpotent Lie group
$G$.  This approach involves first lifting the vector fields $\{X_i\}$
to vector fields $\{\tilde{X}_i\}$ on $M \times \R^k$ for some $k$
(effectively adding additional variables, to handle the possibility
that the $\{X_i\}$ are not linearly independent at $m$).  Then, $M
\times \R^k$ is identified with the free nilpotent Lie group $G$ with
$n$ generators, and under this identification the lifted vector fields
$\{\tilde{X}_i\}$ differ from left-invariant vector fields only to
small order.  Thus, it makes sense to study H\"ormander-type
hypoelliptic operators by studying nilpotent Lie groups.
\index{Lie group!approximation by}

With regard to the heat kernel inequalities we study in this
dissertation, much less is known for general nilpotent Lie groups than
for the Heisenberg group.  The known pointwise bounds on the heat
kernel corresponding to the left-invariant hypoelliptic operator $L$
are in general much less sharp than those in (\ref{H1-heat}); see
Section \ref{heat-previous-sec}.  Gradient bounds like
(\ref{H1-grad}) are also not known to hold with much generality,
although a weaker $L^p$-type estimate has been shown in
\cite{tai-thesis}; see Section \ref{grad-previous-sec}.

It is not clear at this stage what sort of tools are appropriate to
attack these problems in a general Lie group setting.  Therefore, for
this dissertation, we restrict our attention to a smaller class of
nilpotent Lie groups, the so-called H-type or Heisenberg-type groups,
which generalize in a more limited way the Heisenberg group $\mathbb{H}_1$ of Section
\ref{H1-sec}.  In this setting it is possible to carry out more
explicit computations involving heat kernels and thereby obtain
stronger results, analogous to (\ref{H1-heat}) and (\ref{H1-grad})
which are known for $\mathbb{H}_1$.

\chapter{H-type Groups}\label{h-type-chapter}

\section{Definition and elementary properties}

The objects of central study in this dissertation are the so-called
H-type or Heisenberg-type groups.  H-type groups were first introduced
by Kaplan in \cite{kaplan80}.  Chapter 18 of \cite{blu-book} is an
excellent reference for basic facts about these groups.

We begin with the definition.

\begin{definition}\label{h-type-def}
  Let $\mathfrak{g}$ be a finite dimensional real Lie algebra.  We say $\mathfrak{g}$ is an
  \define{H-type Lie algebra}{H-type Lie algebra} \index{Lie
    algebra!H-type|see{H-type Lie algebra}} if there exists an inner
  product $\inner{\cdot}{\cdot}$ on $\mathfrak{g}$ such that:
  \begin{enumerate}
  \item $[\mathfrak{z}^\perp, \mathfrak{z}^\perp] = \mathfrak{z}$,
    where $\mathfrak{z}$ is the center of $\mathfrak{g}$;
    and
\item \label{h-type-property} For each $Z \in \mathfrak{z}$, the map $J_Z :
  \mathfrak{z}^\perp \to \mathfrak{z}^\perp$ defined by
  \begin{equation}\label{Jz-relation}
    \inner{J_Z X}{Y} = \inner{Z}{[X,Y]} \quad \text{for $X,Y \in \mathfrak{z}^\perp$}
  \end{equation}
  is an orthogonal map when $\norm{Z}^2 := \inner{Z}{Z}=1$.
  \end{enumerate}
  We will say that such an inner product is
  \definenoindex{admissible}, and that $\mathfrak{g}$ is an H-type
  Lie algebra \definenoindex{under} $\inner{\cdot}{\cdot}$.  (This should
  not be interpreted as a restrictive statement; if one particular inner
  product will do, certainly others will do as well.)
\end{definition}

\begin{notation}
    If $G$ is a connected finite-dimensional Lie group, we write $\Lie
  G$ for the Lie algebra of left-invariant smooth vector fields on
  $G$, under the bracket operation $[X,Y] = X Y - Y X$. 
\end{notation}

\begin{definition}
 An \define{H-type group}{H-type group} 
 \index{group!H-type|see{H-type group}}
 \index{Lie group!H-type|see{H-type group}}
 is a connected, simply connected Lie group $G$ such that
 $\Lie G$ is an H-type Lie algebra.
\end{definition}

\begin{example}
  As in Example \ref{H1-example}, the classical Heisenberg group
  $\mathbb{H}_1$
  \index{Heisenberg group!as H-type group}
  is the Lie group consisting of
  $\R^3$ with the following group operation:
  \begin{equation}
    (x,y,z) \groupop (x', y', z') = \left(x+x', y+y', z+z'+\frac{1}{2}(x y'
    - x' y)\right).
  \end{equation}
  The Heisenberg Lie algebra
  $\mathfrak{h}_1 = \Lie \mathbb{H}_1$
  \index{Heisenberg Lie algebra!as H-type Lie algebra}
 is spanned by the vector fields
  \begin{equation}
    X = \pp{x} - \frac{1}{2} y \pp{z}, \quad Y = \pp{y} + \frac{1}{2}
    x \pp{z}, \quad Z = \pp{z}.
  \end{equation}
  We note that $[X,Y] = Z$, $[X,Z] = [Y,Z] = 0$.  Then
  $\mathfrak{h}_1$ is an H-type Lie algebra under an inner product
  such that $\{X,Y,Z\}$ are orthonormal.  (In particular, the center
  of $\mathfrak{h}_1$ is one-dimensional and spanned by $Z$, and $J_Z
  X = Y$, $J_Z Y = -X$.)  Thus $\mathbb{H}_1$ is an H-type group.
\end{example}

\begin{example}
  For $n \ge 1$, the (isotropic) \define{Heisenberg-Weyl
    group}{Heisenberg-Weyl group!isotropic} $\mathbb{H}_n$ is the Lie group consisting of
  $\R^{2n+1}$ with the following group operation: 
  \begin{equation}
    \begin{split}
    &(x_1, \dots, x_{2n}, z) \groupop (x_1', \dots, x_{2n}', z') = (  x_1 + x_1', \dots, x_{2n} + x_{2n}', \\
      &\quad\quad
      z + z' +  \frac{1}{2}((x_1 x_2' - x_1' x_2) + (x_3 x_4' - x_3' x_4) + \dots +
    (x_{2n-1} x_{2n}' - x_{2n-1}' x_{2n}))).
    \end{split}
  \end{equation}
  The Lie algebra $\mathfrak{h}_n = \Lie \mathbb{H}_n$ is spanned by
  \begin{equation}
    X_{2i-1} = \pp{x_{2i-1}} - \frac{1}{2} x_{2i} \pp{z}, \quad X_{2i}
    = \pp{x_{2i}} + \frac{1}{2} x_{2i-1} \pp{z}, \quad Z = \pp{z}
  \end{equation}
  where $i$ ranges from $1$ to $n$.  Note that $[X_{2i-1}, X_{2i}] =
  Z$ and all other independent brackets are zero.  Then
  $\mathfrak{h}_n$ is an H-type Lie algebra under an inner product
  such that $\{X_1, \dots, X_{2n}, Z\}$ are orthonormal.  The center
  of $\mathfrak{h}_n$ is one-dimensional and spanned by $Z$, and $J_Z
  X_{2i-1} = X_{2i}$, $J_Z X_{2i} = -X_{2i-1}$. Thus $\mathbb{H}_n$ is an
  H-type group.
\end{example}

\begin{example}
  The \define{complex Heisenberg group}{Heisenberg group!complex} is
  the Lie group $G$ consisting of $\C^3$ with the group operation
  \begin{equation}
    (x,y,z) \groupop (x', y', z') = \left(x+x', y+y', z+z'+\frac{1}{2}(x y'
    - x' y)\right).
  \end{equation}
  If we write $x = x_1 + i x_2$, $y = y_1 + i y_2$, $z = z_1 + i z_2$,
  we find that $\Lie G$ is spanned by the vector fields
  \begin{align*}
    X_1 &= \pp{x_1} - \frac{1}{2} y_1 \pp{z_1} - \frac{1}{2} y_2
    \pp{z_2} & 
    X_2 &= \pp{x_2} + \frac{1}{2} y_2 \pp{z_1} - \frac{1}{2}
    y_1 \pp{z_2} \\
    Y_1 &= \pp{y_1} + \frac{1}{2} x_1 \pp{z_1} + \frac{1}{2} x_2
    \pp{z_2} & 
    Y_2 &= \pp{y_2} - \frac{1}{2} x_2 \pp{z_1} + \frac{1}{2}
    x_1 \pp{z_2} \\
    Z_1 &= \pp{z_1} & Z_2 &= \pp{z_2}
  \end{align*}
  We have $[X_1, Y_1] = -[X_2, Y_2] = Z_1$, $[X_2, Y_1] = [X_1, Y_2] =
  Z_2$, and all other independent brackets vanish.  $\Lie G$ is an
  H-type Lie algebra under an inner product so that $\{X_1, X_2, Y_1,
  Y_2, Z_1, Z_2\}$ are orthonormal.  The center is two-dimensional and
  spanned by $Z_1, Z_2$.  We have
  \begin{align*}
    J_{Z_1} X_1 &= Y_1 & J_{Z_1} X_2 &= -Y_2 
    & J_{Z_1} Y_1 &= -X_1 & J_{Z_1} Y_2 &= X_2 \\
    J_{Z_2} X_1 &= Y_2 & J_{Z_2} X_2 &= Y_1  &
    J_{Z_2} Y_1 &= -X_2 & J_{Z_2} Y_2 &= -X_1 
  \end{align*}
  Thus $G$ is an H-type group.  This example shows explicitly that the
  H-type groups consist of more than the Heisenberg-Weyl groups, and
  may have centers with dimension larger than $1$.
\end{example}

We now list a number of elementary algebraic properties of H-type Lie
algebras, which are useful in computations.

\index{H-type Lie algebra!elementary properties of}
\begin{proposition}\label{Jz-props}
  Let $\mathfrak{g}$ be an H-type Lie algebra under the inner product
  $\inner{\cdot}{\cdot}$, with center
  $\mathfrak{z}$ and $J_Z$ defined by (\ref{Jz-relation}).  If $Z,W
  \in \mathfrak{z}$, $X,Y \in \mathfrak{z}^\perp$, we have:
  \begin{enumerate}
  \item $J_Z$ is linear in $Z$, i.e. $J_{aZ+W} = a J_Z +
    J_W$. \label{Jz-linear}
  \item $J_Z$ is skew-adjoint, i.e. $\inner{J_Z X}{Y} = -\inner{X}{J_Z
    Y}$.  In particular $\inner{J_Z X}{X} = 0$. \label{Jz-skew}
  \item $J_Z^2 = -\norm{Z}^2 I$. \label{Jz-square}
  \item If $Z \ne 0$, $J_Z$ is invertible and $J_Z^{-1} =
    -\norm{Z}^{-2} J_Z$.   \label{Jz-inverse}
  \item $J_Z J_W + J_W J_Z = -2\inner{Z}{W} I$.  (This is the
    fundamental relation that defines Clifford algebras, and suggests
    a connection between H-type Lie algebras and Clifford algebras.
    This connection is more fully explored in Section
    \ref{clifford-sec}.)  \label{Jz-clifford}
  \item $\inner{J_Z X}{J_W X} = \inner{Z}{W}
    \norm{X}^2$. \label{Jz-inner-a}
  \item $\inner{J_Z X}{J_Z Y} = \inner{X}{Y}
    \norm{Z}^2$. \label{Jz-inner-b}
  \item $[X, J_Z X] = \norm{X}^2 Z$. \label{Jz-bracket}
  \item Define $\ad_X : \mathfrak{g} \to \mathfrak{g}$ as usual by
    $\ad_X Y = [X,Y]$.  If $\norm{X}=1$ then $\ad_X$ maps $(\ker
    \ad_X)^\perp$ isometrically onto $\mathfrak{z}$.  (This is
    sometimes taken as part of the definition of an H-type Lie
    algebra, in place of item \ref{h-type-property} of Definition \ref{h-type-def}.) \label{adX}
  \item $\dim \mathfrak{z}^\perp$ is even. \label{horiz-even}
  \item $\dim \mathfrak{z}^\perp \ge \dim \mathfrak{z} + 1$.  (This
    bound is far from sharp; see Theorem
    \ref{dimension-classification} below.) \label{horiz-dim}
  \end{enumerate}
\end{proposition}

\begin{proof}
  \begin{enumerate}
  \item We have
    \begin{align*}
    \inner{J_{aZ+W} X}{Y} &= \inner{aZ+W}{[X,Y]} = a
    \inner{Z}{[X,Y]} + \inner{W}{[X,Y]} \\
    &= a \inner{J_Z X}{Y} +
    \inner{J_W X}{Y} = \inner{(a J_Z + J_W)X}{Y}.
    \end{align*}
  \item $\inner{J_Z X}{Y} = \inner{Z}{[X,Y]} = -\inner{Z}{[Y,X]} =
    -\inner{J_Z Y}{X}$.
  \item If $\norm{Z} = 1$, then $J_Z$ is orthogonal by definition and
    skew-adjoint by item \ref{Jz-skew}, so that
    \begin{equation*}
      \inner{J_Z^2 X}{Y} = -\inner{J_Z X}{J_Z Y} = -\inner{X}{Y}.
    \end{equation*}
    The general case follows by linearity (item \ref{Jz-linear}).
  \item An obvious consequence of item \ref{Jz-square}.
  \item Using polarization, we have
    \begin{align*}
      J_Z J_W + J_W J_Z &= (J_Z + J_W)^2 - J_Z^2 - J_W^2 \\
      &= J_{Z+W}^2 - J_Z^2 - J_W^2 \\
      &= -(\norm{Z+W}^2 - \norm{Z}^2 - \norm{W}^2) I \\
      &= -2\inner{Z}{W} I.
    \end{align*}
  \item By skew-adjointness, we have $\inner{J_Z X}{J_W X} =
    -\inner{J_Z J_W X}{X} = -\inner{J_W J_Z X}{X}$, so that
    \begin{align*}
      \inner{J_Z X}{J_W X} = -\frac{1}{2} \inner{(J_Z J_W + J_W
        J_Z)X}{X} = -\frac{1}{2} \inner{-2\inner{Z}{W} X}{X} =
      \inner{Z}{W}\norm{X}^2
    \end{align*}
    using item \ref{Jz-clifford}.
  \item An obvious consequence of items \ref{Jz-skew} and
    \ref{Jz-square}.
  \item Given $W \in \mathfrak{z}$, we have $\inner{W}{[X,J_Z X]} =
    \inner{J_W X}{J_Z X} = \inner{W}{Z}\norm{X}^2$ by definition of
    $J_Z$ and item \ref{Jz-inner-a}.  Given $Y \in
    \mathfrak{z}^\perp$, we have $\inner{Y}{[X,J_Z X]} = 0 =
    \inner{Y}{Z}$.  Thus $\inner{U}{[X,J_Z X]} =
    \inner{U}{Z}\norm{X}^2$ for all $U \in \mathfrak{g}$, so that
    $[X,J_Z X] = \norm{X}^2 Z$.
  \item Suppose $\norm{X}=1$.  It is obvious that the restriction of
    $\ad_X$ to $(\ker \ad_X)^\perp$ is injective.  We next show the
    restriction maps onto $\mathfrak{z}$.  Given
    $Z \in \mathfrak{z}$, let $Y = J_Z X$.  Then $\ad_X Y = [X,Y] =
    [X, J_Z X] = Z$ by item \ref{Jz-bracket}.  Moreover, suppose $U
    \in \ker \ad_X$; then $\inner{Y}{U} = \inner{J_Z X}{U} =
    \inner{Z}{[X,U]} = \inner{Z}{\ad_X U} = 0$, so $Y \in (\ker
    \ad_X)^\perp$.

    To show isometry, suppose $\ad_X Y = Z$.  By injectivity $Y = J_Z
    X$.  Then $\norm{Y}^2 = \norm{J_Z X}^2 = \norm{Z}^2 \norm{X}^2 =
    \norm{Z}^2$.
  \item For any nonzero $Z \in \mathfrak{z}$, $J_Z :
    \mathfrak{z}^\perp \to \mathfrak{z}^\perp$ is a nonsingular
    skew-adjoint linear transformation.  It follows that $\dim
    \mathfrak{z}^\perp$ must be even.  (In particular, all the
    eigenvalues of $J_Z$ are imaginary and must come in conjugate
    pairs.)
  \item Let $\{Z_1, \dots, Z_m\}$ be an orthonormal basis for
    $\mathfrak{z}$, and $X \in \mathfrak{z}^\perp$ be a unit vector.  The
    vectors $\{J_{Z_1}X, \dots, J_{Z_m}X\} \subset \mathfrak{z}^\perp$
    are unit vectors, which are mutually orthogonal by
    item \ref{Jz-inner-a}, and all are orthogonal to $X$ by item
    \ref{Jz-skew}.  Thus $\{X, J_{Z_1}X, \dots, J_{Z_m}X\}$ is a set
    of $m+1$ orthonormal vectors in $\mathfrak{z}^\perp$.
  \end{enumerate}
\end{proof}

\begin{proposition}
  If $G$ is an H-type group, then $G$ is nilpotent of step $2$ and
  stratified.
  \index{H-type group!nilpotent}
  \index{H-type group!stratified}
\end{proposition}

\begin{proof}
  This is obvious from the condition that $[\mathfrak{z}^\perp,
    \mathfrak{z}^\perp] = \mathfrak{z}$, where $\mathfrak{z}$ is the
  center of $\mathfrak{g} = \Lie G$.  An appropriate stratification is
  $V_1 = \mathfrak{z}^\perp$, $V_2 = \mathfrak{z}$.
\end{proof}

Not all step $2$ stratified Lie groups are H-type.

\begin{example}
\index{Lie group!example of stratified, not H-type}
  The abelian Lie group $G = \R^n$ is step $1$ stratified, but not
  H-type.  (The Lie algebra $\mathfrak{g} = \Lie G$ has center
  $\mathfrak{z} = \mathfrak{g}$, so $[\mathfrak{z}^\perp,
    \mathfrak{z}^\perp] = [0,0] \ne \mathfrak{z}$).
\end{example}

\begin{example}
\index{Lie group!examples of stratified, not H-type}
  For $n \ge 2$, let $a_1, \dots, a_n \in \R$ be constants, and let
  $G$ be the 
  \define{anisotropic Heisenberg-Weyl group}{Heisenberg-Weyl group!anisotropic}
  $G$ consisting
  of $\R^{2n+1}$ with the following group operation:
  \begin{equation}
    \begin{split}
    &(x_1, \dots, x_{2n}, z) \groupop (x_1', \dots, x_{2n}', z') = (  x_1 + x_1', \dots, x_{2n} + x_{2n}', \\
      &\quad\quad
      z + z' +  \frac{1}{2}(a_1(x_1 x_2' - x_1' x_2) + a_2(x_3 x_4' - x_3' x_4) + \dots +
    a_n(x_{2n-1} x_{2n}' - x_{2n-1}' x_{2n}))).
    \end{split}
  \end{equation}
  The Lie algebra $\mathfrak{h}_n = \Lie \mathbb{H}_n$ is spanned by
  \begin{equation}
    X_{2i-1} = \pp{x_{2i-1}} - \frac{a_i}{2} x_{2i} \pp{z}, \quad X_{2i}
    = \pp{x_{2i}} + \frac{a_i}{2} x_{2i-1} \pp{z}, \quad Z = \pp{z}
  \end{equation}
  where $i$ ranges from $1$ to $n$.  Note that $[X_{2i-1}, X_{2i}] =
  a_i Z$ and all other independent brackets are zero, so that $G$ is
  step $2$ stratified (with $V_1 = \spanop\{X_j : 1 \le j \le 2n\}$,
  $V_2 = \spanop Z$).  If the $a_i$ are not all equal, then there is
  no inner product on $\mathfrak{g} = \Lie G$ under which it is an
  H-type Lie algebra.  If there were, then we would have $J_Z X_{2i-1}
  = a_i X_{2i}$, $J_Z X_{2i} = - a_i X_{2i-1}$, so that $J_Z^2
  X_{2i-1} = -a_i^2 X_{2i-1}$.  Since the $a_i$ are not all equal,
  this contradicts item \ref{Jz-square} of Proposition \ref{Jz-props}.
  Thus $G$ is not an H-type group.
\end{example}

Any H-type group can be realized in terms of Euclidean space, as we
now show.  \fixnotme{The proposition and proof are kind of confusing.  Is
  there a better way?}

\begin{proposition}
  \index{H-type group!realization as $\R^{2n+m}$}
  Let $G$ be an H-type group, with $\Lie G = (\mathfrak{g},
  [\cdot,\cdot])$ an H-type Lie algebra under some inner product
  $\inner{\cdot}{\cdot})$, and $\mathfrak{z}$ the center of
  $\mathfrak{g}$.  There exists $n,m \ge 0$, and a bijective linear
  isometry $\phi : (\mathfrak{g}, \inner{\cdot}{\cdot}) \to
  (\R^{2n+m}, \inner{\cdot}{\cdot}_e)$, where $\inner{\cdot}{\cdot}_e$
  is the usual Euclidean inner product on $\R^{2n+m}$, such that
  $\phi(\mathfrak{z}) = 0 \oplus \R^m$.

  Define a bracket $[\cdot,\cdot]'$ on $\R^{2n+m}$ via $[\phi(X),
    \phi(Y)]' = [X,Y]$.  Then $(\R^{2n+m}, [\cdot,\cdot]')$ is an
  H-type Lie algebra under the Euclidean inner product, and $\phi :
  (\mathfrak{g}, [\cdot,\cdot]) \to (\R^{2n+m}, [\cdot,\cdot]')$ is an
  isomorphism of Lie algebras.

  Define a group operation $\groupop'$ on $\R^{2n+m}$ via the
  Baker-Campbell-Hausdorff formula:
  \index{Baker-Campbell-Hausdorff formula}
  \begin{equation}
    v \groupop' w := v + w + \frac{1}{2} [v,w]'.
  \end{equation}
  Then $(\R^{2n+m}, \groupop')$ is an H-type Lie group which is
  isomorphic to $G$, and whose Lie algebra is canonically isomorphic
  to $(\R^{2n+m}, [\cdot,\cdot]')$.
\end{proposition}

\begin{proof}
  Let $m = \dim \mathfrak{z}$ and $2n = \dim \mathfrak{z}^\perp$
  (recall from Proposition \ref{Jz-props} item \ref{horiz-even} that
  $\dim \mathfrak{z}^\perp$ is even).  By selecting an orthonormal
  basis for $(\mathfrak{g}, \inner{\cdot}{\cdot})$ which is adapted to
  the decomposition $\mathfrak{g} = \mathfrak{z}^\perp \oplus
  \mathfrak{z}$, we can construct a (non-canonical) bijective linear
  isometry $\phi : (\mathfrak{g}, \inner{\cdot}{\cdot}) \to
  (\R^{2n+m}, \inner{\cdot}{\cdot}_e)$ such that $\phi(\mathfrak{z}) =
  0 \oplus \R^m$.  If $[\cdot,\cdot]'$ is constructed as specified, it
  is obvious that $(\R^{2n+m}, [\cdot,\cdot]')$ is an
  H-type Lie algebra under the Euclidean inner product, and $\phi :
  (\mathfrak{g}, [\cdot,\cdot]) \to (\R^{2n+m}, [\cdot,\cdot]')$ is an
  isomorphism of Lie algebras.

  Since $G$ is a step 2 nilpotent Lie group which is connected and simply
  connected, if we define a group operation $\groupop$ on
  $\mathfrak{g}$ via
  \begin{equation}\label{BCH}
    X \groupop Y := X + Y + \frac{1}{2}[X,Y]
  \end{equation}
  then the exponential map $\Exp : (\mathfrak{g}, \groupop) \to G$ is
  a Lie group isomorphism.  Since $\phi : (\mathfrak{g},
  [\cdot,\cdot]) \to (\R^{2n+m}, [\cdot,\cdot]')$ is a Lie algebra
  isomorphism, and $\groupop'$ is defined appropriately in terms of
  $[\cdot,\cdot]'$, we have that $\phi : (\mathfrak{g}, \groupop) \to
  (\R^{2n+m}, \groupop')$ is a Lie group isomorphism.  Thus $G$ is
  isomorphic to $(\R^{2n+m}, \groupop')$.
\end{proof}

Thus, henceforth we can, and will, assume that our H-type group $G$ is
$\R^{2n+m}$ with a group operation $\groupop$ defined by (\ref{BCH})
for some Lie bracket $[\cdot,\cdot]$, and that $\Lie G \cong
(\R^{2n+m}, [\cdot,\cdot]$ is an H-type Lie algebra under the
Euclidean inner product.

\begin{notation}
  We will write elements of $G = \R^{2n+m}$ as $g = (x,z)$, where $x \in
  \R^{2n}$, $z \in \R^m$.  $\inner{\cdot}{\cdot}$ denotes the
  Euclidean inner product on $\R^{2n+m}$.  $\{e_1, \dots, e_{2n}\}$ is
  the standard orthonormal basis for $\R^{2n} \oplus 0$, and $\{u_1,
  \dots, u_m\}$ is the standard orthonormal basis for $0 \oplus
  \R^{m}$.  We will use the coordinates $x_i(g) = \inner{g}{e_i}$,
  $z_j(g) = \inner{g}{u_j}$.
\end{notation}

We now have several distinct structures on the single set $\R^{2n+m}$:
\begin{itemize}
\item A vector space, under the usual addition and scalar multiplication;
\item An inner product space, under the usual Euclidean inner product;
\item A Lie algebra, under the bracket $[\cdot,\cdot]$;
\item A Lie group, under the group operation $\groupop$.
\end{itemize}

These structures do not necessarily interact nicely.  In particular,
addition, scalar multiplication, and the Euclidean inner product are
not left- or right-invariant with respect to $\groupop$; the scalings $v
\mapsto c v$ are not Lie group homomorphisms for $c \ne 1$; and $\groupop$ is not
bilinear.  Nevertheless, we will sometimes view these structures,
especially the inner product and the operators $J_z \in \End(\R^{2n})$
for $z \in \R^m$, as functions on $\R^{2n+m}$ viewed as a group.  The
reader who prefers a more intrinsic view may insert the exponential
map (which is the identity on $\R^{2n+m}$) where appropriate.

\begin{notation}
Note that $\R^m = 0 \oplus \R^m$ is the center of the Lie
algebra and also of the Lie group.  Thus we will sometimes think of
the Lie bracket as a map $[\cdot,\cdot] : \R^{2n} \times \R^{2n} \to
\R^m$.  We can thus write the group operation as
\begin{equation}\label{BCH-realized}
  (x,z) \groupop (x', z') = \left(x + x', z + z' + \frac{1}{2} [x, x']\right).
\end{equation}
Note that $0$ is the group identity, and the inverse map is $(x,z)^{-1} = -(x,z)$.
\end{notation}

\begin{notation}
  For each $g \in G$, let $L_g, R_g : G \to G$ be the maps of left-
  and right-translation by $g$, and $j : G \to G$ be the inverse map.
  That is, $L_g h = g \groupop h$, $R_g h = h \groupop g$, $j h =
  h^{-1}$.  Of course, $L_g, R_g, j$ are diffeomorphisms of the
  smooth manifold $G$.
\end{notation}

\begin{proposition}
  \index{H-type group!Haar measure}
  \index{Haar measure}
  Lebesgue measure $\haar$ on $G = \R^{2n+m}$ is invariant under left
  and right translation and inverses with respect to $\groupop$.
  Thus, $G$ is unimodular, and we may take $\haar$ to be a
  bi-invariant Haar measure on $G$.
\end{proposition}

\begin{proof}
  By inspection of (\ref{BCH-realized}), it is clear that for any $g$
  the differentials of $L_g$ (and likewise $R_g$) is lower triangular, and thus
  the Jacobian determinant is $1$.  The Jacobian determinant of $j$ is
  obviously $1$ in absolute value.
\end{proof}

Given this measure, we can define convolution on $G$.

\begin{definition}\label{convolution-def}
  \index{H-type group!convolution $*$}
  \index{convolution}
  If $f_1, f_2 : G \to \R$, their convolution $f_1 * f_2$ is the function
  \begin{equation*}
    (f_1 * f_2)(g) := \int_G f_1(g \groupop k^{-1}) f_2(k)\,d\haar(k)
    = \int_G f_1(k) f_2(k^{-1} \groupop g)\,d\haar(k).
  \end{equation*}
  for all $g$ such that the integral makes sense.  We may also take
  $f_1, f_2$ to be appropriate distributions.
\end{definition}

The following properties are typical and we omit the proofs.

\begin{proposition}\label{convolution-props}
  \begin{enumerate}
  \item $(f_1 \circ j) * (f_2 \circ j) = (f_2 * f_1) \circ j$. \label{convolution-commute}
  \item If $f$ is a distribution on $G$ and $\psi \in C^\infty_c(G)$, then $f * \psi, \psi * f \in
    C^\infty(G)$.  The same holds if $f$ is a tempered distribution
    and $\psi$ is a Schwartz function on $G = \R^{2n+m}$.  \label{convolution-smoothing}
  \item (Young's inequality) If $f_1 \in L^1(G)$ and $f_2 \in L^p(G)$,
    then $f_1 * f_2 \in L^p(G)$ and $\norm{f_1 * f_2}_{L^p} \le
    \norm{f_1}_{L^1} \norm{f_2}_{L^2}$. \label{convolution-young}
    \index{Young's inequality}
  \item If $\hinner{\cdot}{\cdot}$ denotes the inner product on
    $L^2(G)$, $f_1, f_2 \in L^2(G)$ and $\psi \in L^1(G)$, then
    $\hinner{\psi * f_1}{f_2} = \hinner{f_1}{\tilde{\psi} * f_2}$ and
    $\hinner{f_1 * \psi}{f_2} = \hinner{f_1}{f_2 * \tilde{\psi}}$,
    where $\tilde{\psi}(g) =
    \conj{\psi(g^{-1})}$. \label{convolution-adjoint}
  \end{enumerate}
\end{proposition}

An important operation on H-type groups is the following dilation.

\begin{definition}\label{dilation-def}
\index{H-type group!dilation $\varphi_\alpha$}
\index{dilation}
  For $\alpha \in (0,\infty)$, define $\varphi_\alpha : G \to G$ by
    $\varphi_\alpha(x,z) = (\alpha x, \alpha^2 z)$.
\end{definition}

Observe that
$\varphi_\alpha$ is a group automorphism of $G$, $\varphi_\alpha
\circ \varphi_\beta = \varphi_{\alpha \beta}$, and
$\varphi_\alpha^{-1} = \varphi_{\alpha^{-1}}$.  We also note the change of
variables $d\haar(\varphi_\alpha(g)) = \alpha^{2(n+m)}\,d\haar(g)$.

\section{Algebraic properties}\label{clifford-sec}

Algebraically, H-type Lie algebras (and hence also H-type groups)
correspond to representations of Clifford algebras.  In this section,
we describe this correspondence, and classify the possible dimensions
of H-type groups.

\begin{definition}
  Let $V$ be a real vector space and $B : V \times V \to \R$ be a
  bilinear form.  The \define{Clifford algebra}{Clifford algebra}
  $C\ell(V, B)$ is the algebra (with identity $1$) freely generated by $V$
  subject to the relation
  \begin{equation}\label{clifford-relation}
    u v + v u = - 2 B(u,v).
  \end{equation}
  (Some authors omit the negative sign.)
\end{definition}

\begin{definition}
  A \define{representation}{representation} of an
  algebra $\mathcal{A}$ is an algebra homomorphism $\pi : \mathcal{A}
  \to \End(W)$ for some finite-dimensional vector space $W$; the
  \definenoindex{dimension} of $\pi$ is the dimension of $W$.  (We
  require that $\pi(1) = I$.)
  If $\pi : \mathcal{A} \to \End(W)$, $\pi' : \mathcal{A}'
  \to \End(W')$ are representations of two algebras, we say they are
  \define{equivalent}{representation!equivalent} if there is an
  algebra isomorphism $\psi : \mathcal{A} \to \mathcal{A}'$ and a
  vector space isomorphism $T : W \to W'$ such that
  \begin{equation}\label{equivalent-representation}
    \pi(a) = T^{-1} \pi'(\psi(a)) T
  \end{equation}
  for all $a \in \mathcal{A}$.  If the vector spaces $V,W$ are
  equipped with inner products and $T$ is unitary, we say that $\pi,
  \pi'$ are \definenoindex{unitarily equivalent}.
\end{definition}

We note that if $\pi : \mathcal{A} \to \End(W)$ is a representation of
$\mathcal{A}$, then the left action $v \cdot w = \pi(v) w$ turns $W$
into a left $\mathcal{A}$-module.  This is an alternate, and indeed
more common, way to study representations of algebras.

One direction of the correspondence between H-type Lie algebras and
representations of Clifford algebras is rather easy.

\index{Clifford algebra!correspondence with H-type Lie algebra|(}
\index{H-type Lie algebra!correspondence with Clifford algebra|(}
\index{H-type group!correspondence with Clifford algebra|(}
\begin{proposition}
  \begin{enumerate}
    \item 
  If $\mathfrak{g}$ is an H-type Lie algebra under the inner product
  $\inner{\cdot}{\cdot}$, then the map 
  \begin{equation*}
    \mathfrak{z} \ni z \mapsto \pi(z) := J_z
    \in \End(\mathfrak{z}^\perp)
  \end{equation*}
  extends uniquely to a representation of $C\ell(\mathfrak{z},
  \inner{\cdot}{\cdot})$.  

  \item If $\mathfrak{g}'$ is another
  H-type Lie algebra, isomorphic to $\mathfrak{g}$, then there is an
  admissible inner product $\inner{\cdot}{\cdot}'$ on $\mathfrak{g}'$
  such that the map
  \begin{equation*}
    \mathfrak{z}' \ni z' \mapsto \pi'(z') := J_{z'}
    \in \End(\mathfrak{z}'^\perp)
  \end{equation*}
  defines a representation of
  $C\ell(\mathfrak{z}', \inner{\cdot}{\cdot}')$ which is unitarily equivalent to
  $\pi$.

  \item
  If $(\mathfrak{g}, \inner{\cdot}{\cdot}), (\mathfrak{g}',
  \inner{\cdot}{\cdot}')$ are two H-type Lie algebras with admissible
  inner products, and the corresponding representations $\phi, \phi'$
  are unitarily equivalent, then $\mathfrak{g}$ and $\mathfrak{g}'$
  are isomorphic.
  \end{enumerate}
\end{proposition}

\begin{proof}
  \begin{enumerate}
  \item This follows directly from item \ref{Jz-clifford} of Proposition \ref{Jz-props}.

  \item If $\phi : \mathfrak{g} \to \mathfrak{g}'$ is a Lie algebra
    isomorphism, then define the inner product $\inner{\cdot}{\cdot}'$
    on $\mathfrak{g}'$ via $\inner{\phi(v)}{\phi(w)}' :=
    \inner{v}{w}$.  It is clear that $\mathfrak{z}' =
    \phi(\mathfrak{z})$, $\mathfrak{z}'^\perp =
    \phi(\mathfrak{z}^\perp)$.  We also note that for $z \in
    \mathfrak{z}$ and $x,y \in \mathfrak{z}^\perp$, we have
    \begin{align*}
      \inner{J_{\phi(z)} \phi(x)}{\phi(y)}' &=
      \inner{\phi(z)}{[\phi(x), \phi(y)]'}' \\
      &= \inner{\phi(z)}{\phi([x,y])}' \\
      &= \inner{z}{[x,y]} \\
      &= \inner{J_z x}{y} \\
      &= \inner{\phi(J_z x)}{\phi(y)}'
    \end{align*}
    so that $J_{\phi(z)} \phi(x) = \phi(J_z x)$.  If $\norm{\phi(z)}'
    = \norm{z} = 1$, we have
    \begin{equation*}
      \inner{J_{\phi(z)} \phi(x)}{J_{\phi(z)} \phi(y)}' =
      \inner{\phi(J_z x)}{\phi(J_z y)}' = \inner{J_z x}{J_z y} =
      \inner{x}{y} = \inner{\phi(x)}{\phi(y)}'
    \end{equation*}
    so that $J_{\phi(z)}$ is indeed orthogonal.  Thus
    $\inner{\cdot}{\cdot}'$ is admissible for $\mathfrak{g}$.
    
    Moreover, the restriction of $\phi$ to $\mathfrak{z}$ extends
    uniquely to an isomorphism of algebras $\psi : C\ell(\mathfrak{z},
    \inner{\cdot}{\cdot}) \to C\ell(\mathfrak{z}',
    \inner{\cdot}{\cdot}')$, and the restriction to
    $\mathfrak{z}^\perp$ is a unitary map $T : \mathfrak{z}^\perp \to
    \mathfrak{z}'^\perp$.

    Let $\pi, \pi'$ be the representations corresponding to
    $(\mathfrak{g}, \inner{\cdot}{\cdot}), (\mathfrak{g}',
    \inner{\cdot}{\cdot}')$.  To show they are equivalent, it suffices
    to verify (\ref{equivalent-representation}) for $a = z \in
    \mathfrak{z}$.  But if $z \in \mathfrak{z}$, $x \in
    \mathfrak{z}^\perp$, we have
    \begin{align*}
      T^{-1} \pi'(\psi(z)) T x = \phi^{-1}(J_{\phi(z)} \phi(x)) = J_z
      x = \pi(z) x
    \end{align*}
    since $J_{\phi(z)} \phi(x) = \phi(J_z x)$ as shown above.

  \item Suppose there exists an algebra isomorphism $\psi :
    C\ell(\mathfrak{z}, \inner{\cdot}{\cdot}) \to C\ell(\mathfrak{z}',
    \inner{\cdot}{\cdot}')$ and a unitary $T : \mathfrak{z}^\perp \to
    \mathfrak{z}'^\perp$ such that (\ref{equivalent-representation})
    holds.  Given $x \in \mathfrak{z}^\perp$, $z \in \mathfrak{z}$,
    and set $\phi(x+z) = Tx + \psi(z)$.  Clearly $\phi : \mathfrak{g}
    \to \mathfrak{g}'$ is a well defined linear bijection; we verify
    it is a Lie algebra homomorphism.  First, we note that
    $[\phi(x+z), \phi(y+w)]' = [\phi(x), \phi(y)]'$ by linearity and
    the fact that $\phi(\mathfrak{z}) = \mathfrak{z}'$.  Next, if $x,
    y \in \mathfrak{z}^\perp$ and $z \in \mathfrak{z}$, we have
    \begin{align*}
      \inner{\phi(z)}{[\phi(x), \phi(y)]'}' &= \inner{J_{\phi(z)}
        \phi(x)}{\phi(y)}' \\
      &= \inner{\pi'(\phi(z)) Tx}{Ty}' \\
      &= \inner{T \pi(z) x}{Ty}' \\
      &= \inner{\pi(z) x}{y} \\
      &= \inner{J_z x}{y} \\
      &= \inner{z}{[x,y]} \\
      &= \inner{\phi(z)}{\phi([x,y])}'.
    \end{align*}
    Thus $[\phi(x),\phi(y)]' = \phi([x,y])$, so that $\phi$ is indeed
    a Lie algebra homomorphism.
  \end{enumerate}
\end{proof}

The converse is only a little more involved.

\begin{theorem}\label{clifford-gives-h-type}
  If $V$ is a finite-dimensional vector space equipped with an inner
  product $\inner{\cdot}{\cdot}_V$ and $\pi : C\ell(V,
  \inner{\cdot}{\cdot}_V) \to \End(W)$ is a representation of the corresponding
  Clifford algebra, then there exist a bracket $[\cdot,\cdot]$ and an
  inner product $\inner{\cdot}{\cdot}$ on $W \oplus V$ under which it
  is an H-type Lie algebra with center $V$, and $J_z = \pi(z)$ for $z
  \in V$.
\end{theorem}

\begin{proof}
  We begin by constructing an inner product on $W$ by a standard
  averaging technique.  Let $u_1, \dots, u_m$ be an orthonormal basis
  for $V$, so that $\inner{u_i}{u_j}_V = \delta_{ij}$.  Notice that
  $u_i^2 = -1$, $u_i u_j = -u_j u_i$ for $i \ne j$.  Then the finite
  subset $H$ of $C\ell(V, \inner{\cdot}{\cdot}_V)$ defined by
  \begin{equation}\label{H-clifford-def}
    H = \{ \pm 1, \pm u_{i_1}u_{i_2}\dots u_{i_n} : 1 \le i_1, \dots, i_n \le
      m, n \ge 1\}
  \end{equation}
  is a group under the algebra multiplication.  Let
  $\inner{\cdot}{\cdot}_{W,1}$ be any inner product on $W$.  For $w, w' \in W$, let
  \begin{equation}
    \inner{w}{w'}_{W,2} := \frac{1}{\abs{H}} \sum_{g \in H} \inner{\pi(g) w}{\pi(g) w'}_{W,1}.
  \end{equation}
  Clearly $\inner{\cdot}{\cdot}_{W,2}$ is again an inner product on $W$.

  Define an inner product $\inner{\cdot}{\cdot}$ on $W \oplus V$ by
  \begin{equation*}
    \inner{w + v}{w' + v'} := \inner{w}{w'}_{W,2} + \inner{v}{v'}_V.
  \end{equation*}
  Observe that with respect to this inner product, we have for each
  basis vector $u_j$ that
  \begin{align*}
    \inner{\pi(u_j) w}{w'} &=  \frac{1}{\abs{H}} \sum_{g \in
      H} \inner{\pi(g u_j) w}{\pi(g) w'}_{W,1} \\
    &= \frac{1}{\abs{H}} \sum_{g' \in
  H} \inner{\pi(g') w}{\pi(g' u_j^{-1}) w'}_{W,1} \\
    &= \inner{w}{\pi(u_j^{-1}) w'} \\
    &= -\inner{w}{\pi(u_j) w'}
  \end{align*}
  by making the change of dummy variables $g' = g u_j$, and noticing
  that $u_j^{-1} = -u_j$ (the inverse taken in $H$).  Thus $\pi(u_j)$
  is skew-adjoint on $W$ with respect to $\inner{\cdot}{\cdot}$, and
  by linearity this is also true of $\pi(v)$ for any $v \in V$.  From
  the Clifford algebra relation, we also have $\pi(u)^2 = -I$, for any
  unit vector $u \in V$, from which it follows that $\pi(u)$ is an
  orthogonal linear transformation with respect to
  $\inner{\cdot}{\cdot}$.

  Define a bracket $[\cdot,\cdot]$ on $W \oplus V$ by
  \begin{equation*}
    [w + v, w' + v'] := \sum_{j = 1}^m \inner{\pi(u_j) w}{w'} u_j \in V.
  \end{equation*}
  It is obvious that this bracket is bilinear, and its skew-symmetry
  follows from the skew-symmetry of $\pi(u_j)$ established above.  The
  Jacobi identity is trivial, since $[W \oplus V, V] = 0$.  Thus
  $\mathfrak{g} = (W
  \oplus V, [\cdot,\cdot])$ is a Lie algebra.

  Let $\mathfrak{z}$ be the center of $\mathfrak{g}$.  We already have
  that $V \subset \mathfrak{z}$.  To see the reverse inclusion,
  suppose there exists $w,v$ such that $[w, w'] = 0$ for all $w' \in
  W$.  In particular,
  \begin{align*}
    0 = \inner{u_1}{[w, \pi(u_1) w]} = \sum_{j = 1}^m \inner{\pi(u_j) w}{\pi(u_1)
      w} \inner{u_1}{u_j} = \inner{\pi(u_1) w}{\pi(u_1) w} = \inner{w}{w}
  \end{align*}
  so that $w = 0$.  Thus $W \cap \mathfrak{z} = 0$, so that
  $\mathfrak{z} = V$.

  To see that $\mathfrak{g}$ is an H-type Lie algebra under
  $\inner{\cdot}{\cdot}$, we first note that $[\mathfrak{z}^\perp,
    \mathfrak{z}^\perp] = [W, W] \subset V = \mathfrak{z}$.  Next, $v
  \in V$, $w, w' \in W$ we have
  \begin{equation*}
    \inner{v}{[w, w']} = \sum_{j = 1}^m \inner{\pi(u_j) w}{w'}
    \inner{v}{u_j} = \inner{\pi(v) w}{w'} 
  \end{equation*}
  so that $J_v = \pi(v)$.  We have already shown that $\pi(v)$ is
  orthogonal when $\norm{v} = 1$.  Finally, we must show $[W,W] = V$.
  For nonzero $w \in W$, we have $\inner{u_k}{[w, \pi(u_l) w]} =
  \inner{\pi(u_k) w}{\pi(u_l) w}$.  If $k=l$ then the isometry of
  $\pi(u_k)$ gives $\inner{u_l}{[w, \pi(u_l) w]} = \inner{w}{w}$.  If
  $k \ne l$ then $u_k, u_l$ anticommute, so by skew-symmetry 
  \begin{equation*}
    \inner{\pi(u_k) w}{\pi(u_l) w} = -\inner{\pi(u_l) \pi(u_k) w}{w} =
    \inner{\pi(u_k) \pi(u_l) w}{w} = -\inner{\pi(u_l) w}{\pi(u_k) w}
  \end{equation*}
  so that $\inner{u_k}{[w, \pi(u_l) w]} = 0$.  Thus $[w, \pi(u_l) w] =
  u_l \inner{w}{w}$.  It follows by linearity that $[w, \pi(v) w] = v$
  for any $v \in V$.
\end{proof}
\index{Clifford algebra!correspondence with H-type Lie algebra|)}
\index{H-type Lie algebra!correspondence with Clifford algebra|)}
\index{H-type group!correspondence with Clifford algebra|)}

From this theorem we can immediately derive some consequences
regarding the possible dimensions of H-type Lie algebras and their
centers.

\begin{corollary}
  For any $m \ge 1$, there exists an H-type Lie algebra
  $\mathfrak{g}$ with center $\mathfrak{z}$ such that $\dim
  \mathfrak{z} = m$ and $\dim \mathfrak{g}$ is arbitrarily large.  
\end{corollary}

\begin{proof}
  Take $V = \R^m$ with the Euclidean inner product.
  Let $\pi : C\ell(V, \inner{\cdot}{\cdot}) \to \End(W)$ be any
  nontrivial representation of the corresponding Clifford algebra.
  (Note that the group $H$ defined in (\ref{H-clifford-def}) forms a
  basis for $C\ell(V, \inner{\cdot}{\cdot}_V)$, so any group
  representation of $H$ extends by linearity to an algebra
  representation of $C\ell(V, \inner{\cdot}{\cdot}_V)$.)  Then Theorem
  \ref{clifford-gives-h-type} gives an H-type Lie algebra
  $\mathfrak{g}$ with $\dim \mathfrak{z} = \dim V = m$ and $\dim
  \mathfrak{g} = m + \dim W$.  To make $\dim \mathfrak{g}$ larger,
  replace $\pi$ with $\pi \oplus \pi$, et cetera.
\end{proof}


Necessary and sufficient conditions on the dimension of a Clifford
algebra and its representations are given by the Hurwitz-Radon-Eckmann
theorem \index{Hurwitz-Radon-Eckmann theorem} \cite{eckmann}.  The
corresponding statement for H-type Lie algebras was given in
\cite{kaplan80}; we restate it here.

\begin{theorem}\label{dimension-classification}
  For any nonnegative integer $k$, we can uniquely write $k=a
  2^{4p+q}$ where $a$ is odd and $0 \le q \le 3$; let $\rho(k) :=
  8p+2^q$.  ($\rho$ is sometimes called the \define{Hurwitz-Radon
    function}{Hurwitz-Radon function $\rho$}.)
  There exists an H-type Lie algebra 
  of dimension $2n+m$ with center of dimension $m$ if and only if $m <
  \rho(2n)$.
\end{theorem}

In particular, suppose $m < \rho(2n)$ where $2n = a 2^{4p+q}$ as
above.  Since $2^q \le 2q+2$ for $0 \le q \le 3$, we have
\begin{equation*}
  2^m < 2^{8p+2^q} \le 2^{8p+2q+2} = 4(2^{4p+q})^2 \le 4(2n)^2
\end{equation*}
so that $n > \frac{1}{4}2^{m/2}$.  Thus, in order for an H-type Lie
algebra to have a large-dimensional center, the complement of the
center must be of very large dimension.

Much more information about Clifford algebras, including a
classification of their representations, can be found in
\cite{atiyah-bott-shapiro}.  This in particular could be useful in
constructing examples of H-type groups for computations.  

\section{The sublaplacian $L$}

In this section, we construct the sublaplacian operator which will be
the focus of this dissertation.  $G$ denotes an H-type Lie group
identified with $\R^{2n+m}$.

\begin{notation}
  For $i=1, \dots, 2n$, let $X_i, \hat{X}_i$ be respectively the
  unique left- and right-invariant vector fields on $G$ with $X_i(0) =
  \hat{X}_i(0) = \pp{x_i}$.  For $j = 1, \dots, m$, let $Z_j$ be the
  bi-invariant vector field $Z_j = \pp{u_j}$.
  \define{}{H-type group!vector fields $X_i$}
  \define{}{H-type group!vector fields $Z_j$}
\end{notation}

\begin{proposition}
  The vector fields $\{X_1, \dots, X_{2n}\}$ are bracket generating in
  the sense of \ref{bracket-gen-condition}.
\end{proposition}

\begin{proof}
  $\spanop \{X_1, \dots, X_{2n}\} = \mathfrak{z}^\perp$, and we have
  $[\mathfrak{z}^\perp, \mathfrak{z}^\perp] = \mathfrak{z}$.  Thus any
  element of $\mathfrak{z}$ can be written as a linear combination of
  brackets of pairs of the $X_i$.
\end{proof}

We can write
\begin{equation}\label{Xi-dds}
  X_i f(g) = \diffat{s}{0} f(g \groupop (s e_i, 0)),
  \quad   \hat{X}_i f(g) = \diffat{s}{0} f((s e_i,
  0) \groupop g).
\end{equation}
A straightforward calculation shows
\begin{equation}\label{Xi-formula}
  \begin{split}
    X_i &= \pp{x_i} + \frac{1}{2} \sum_{j=1}^m
    \inner{J_{u_j}x}{e_i} \pp{z_j} \\
    \hat{X}_i &= \pp{x_i} - \frac{1}{2} \sum_{j=1}^m
    \inner{J_{u_j}x}{e_i} \pp{z_j} \\
  \end{split}
\end{equation}

We note that $[X_i, \hat{X}_j] = 0$ for all $i,j$.  

The vector fields $X_i$ interact with the dilations $\varphi_\alpha,
\alpha > 0$ of
Definition \ref{dilation-def} via
\begin{equation}\label{Xi-dilation}
  X_i (f \circ \varphi_\alpha) = \alpha (X_i f)
\circ \varphi_\alpha.
\end{equation}
We also note that
\begin{equation}\label{Xi-inverse}
  X_i (f \circ j) = - (\hat{X}_i f) \circ j
\end{equation}
where $j(g) = g^{-1}$.

\begin{definition}
The left-invariant 
\define{gradient}{H-type group!subgradient $\grad$}
\index{H-type group!subgradient|seealso{gradient bounds}}
 (or
``subgradient'') $\grad$ on $G$ is given by $\grad f = (X_1 f, \dots,
X_{2n} f)$, with the right-invariant $\hat{\grad}$ defined
analogously.  We shall also use the notation $\grad_x f := \left(
\pp{x_i}f,\dots, \pp{x_{2n}}f\right)$ and $\grad_z f := \left(
\pp{z_1}f,\dots, \pp{z_m}f\right)$ to denote the usual Euclidean
gradients in the $x$ and $z$ variables, respectively.  Note that
$\grad_z$ is both left- and right-invariant.
\end{definition}

From (\ref{Xi-formula}) it is easy to verify that
\begin{equation}\label{gradient-formula}
  \begin{split}
    \grad f(x,z) &= \grad_x f(x,z) + \frac{1}{2} J_{\grad_z f(x,z)} x
    \\
    \hat{\grad} f(x,z) &= \grad_x f(x,z) - \frac{1}{2} J_{\grad_z f(x,z)} x.
  \end{split}
\end{equation}
As shorthand, we write 
\begin{equation}\label{gradient-bare}
\grad = \grad_x + \frac{1}{2} J_{\grad_z} x.
\end{equation}
In particular, since $J_z$ depends linearly on $z$ and is orthogonal
for $\abs{z}=1$, we have 
\begin{equation}\label{abs-grad-difference}
\abs{(\grad -
  \hat{\grad})f(x,z)} = \abs{x} \abs{\grad_z f(x,z)}.
\end{equation}

The gradient has an even nicer form when applied to functions with
appropriate symmetry.

\begin{definition}
  A function $f : G \to \R$ is \define{radial}{radial function} if
  $f(x,z) = \tilde{f}(\abs{x},\abs{z})$ for some $\tilde{f} :
  [0,\infty) \times [0,\infty) \to \R$. By abuse of notation, we will
      also write $f(x,z) = f(\abs{x}, \abs{z})$.
\end{definition}

For a radial function $f$, we have
\index{radial function!gradient of}
\begin{equation}\label{gradient-radial}
  \grad f(x,z) = f_{\abs{x}} (\abs{x}, \abs{z}) \hat{x} + \frac{1}{2}f_{\abs{z}}(\abs{x},
  \abs{z}) \abs{x} J_{\hat{z}} \hat{x} 
\end{equation}
where we use the notation $\hat{u} := \frac{u}{\abs{u}}$ to denote the
unit vector in the $u$ direction.  We draw attention to the fact that
$\hat{x}$ and $J_{\hat{z}}\hat{x}$ are orthogonal unit vectors in
$\R^{2n}$ for any nonzero $x,z$.

\begin{definition}
  The left-invariant 
  \define{sublaplacian}{H-type group!sublaplacian $L$}
  \index{sublaplacian|seealso{under H-type group}}
  $L$ is
  the second-order differential operator defined by
  \begin{equation}
    L := X_1^2 + \dots + X_{2n}^2.
  \end{equation}
\end{definition}

For convenience in computations involving $L$, we adopt the following
notation.  If $\vec{A}, \vec{B}$ are $k$-tuples of operators,
e.g. $\vec{A} = (A_1, \dots, A_k)$, we let $\ncinner{\vec{A}}{\vec{B}}
:= \sum_{i=1}^k A_i B_i$.  (Note that in general
$\ncinner{\vec{A}}{\vec{B}} \ne \ncinner{\vec{B}}{\vec{A}}$.)  We can
write $L$ in terms of the gradient $\grad$ as $L =
\ncinner{\grad}{\grad}$, which by (\ref{gradient-bare}) gives
\begin{equation}\label{L-threeterms}
  L = \ncinner{\grad_x + \frac{1}{2} J_{\grad_z}x}{\grad_x + \frac{1}{2}
    J_{\grad_z}x} = \Delta_x + \ncinner{\grad_x}{J_{\grad_z}x} +
  \frac{1}{4} \abs{x}^2 \Delta_z.
\end{equation}
We used the fact that $\ncinner{\grad_x}{J_{\grad_z}x} =
\ncinner{J_{\grad_z}x}{\grad_x}$, because $\pp{x_i} \inner{J_u x}{e_i}
= \inner{J_u e_i}{e_i} = 0$ for any $u \in \R^m$ by item \ref{Jz-skew}
of Proposition \ref{Jz-props}.

By H\"ormander's theorem (Theorem \ref{hormander-thm}), the bracket
generating condition on the vector fields $\{X_i\}$ implies  that
$L$ is hypoelliptic.  
\index{H-type group!sublaplacian $L$!hypoellipticity}
However, $L$ is not elliptic, as we now verify.
We give a definition here to ensure that there is no doubt about
terminology.

\begin{definition}\label{elliptic-def}
   Let
   \begin{equation*}
     A = \sum_{i,j=1}^n a_{ij}(\vec{x}) \pp{x_i} \pp{x_j} +
     \sum_{i=1}^n b_i(\vec{x}) \pp{x_i} + c(\vec{x})
   \end{equation*}
   be a general second-order partial differential operator.  The
   \define{principal symbol}{operator!principal symbol of} of $A$ is
   the quadratic form-valued function
   \begin{equation*}
     Q_A(\vec{x})(\vec{\xi}) = \sum_{i,j=1}^n a_{ij}(\vec{x}) \xi_i \xi_j.
\end{equation*}
   We will say $A$ is \define{elliptic}{operator!elliptic} if
   $Q_A(\vec{x})$ is (strictly) positive definite for all $\vec{x} \in
   \R^n$, i.e. $Q_A(\vec{x})(\vec{\xi}) > 0$.  
   $A$ is
   \define{degenerate elliptic}{operator!degenerate elliptic} if
   $Q_A(\vec{x})$ is positive semidefinite for all $\vec{x} \in \R^n$,
   i.e. $Q_A(\vec{x})(\vec{\xi}) \ge 0$.
 \end{definition}

The principal symbol of our sublaplacian $L$ is then given by
\begin{align*}
  Q_L(x,z)(\xi,\eta) = \abs{\xi + \frac{1}{2} J_{\eta}x}^2 \ge 0
\end{align*}
so that $Q_L(x,z)$ is positive semidefinite for all $(x,z)$.  However,
$Q(x,z)$ is also degenerate for all $(x,z)$, which can be seen by
taking $\xi = -\frac{1}{2} J_{\eta} x$.  Therefore $L$ is degenerate
elliptic, but not elliptic.

We record at this point some convolution formulas involving $X_i$,
which will be used later.

\begin{proposition}\label{Xi-integration}
  \begin{enumerate}
  \item If $f$ is a distribution on $G$ and $\psi \in C^\infty_c(G)$,
    then $X_i(f * \psi) = f * X_i \psi$.
    Hence also $L(f * \psi) = f *
    L\psi$.  The same holds if $f$ is a
    tempered distribution on $G$ and $\psi$ is a Schwartz
    function.  \label{Xi-convolve}
  \item (Integration by parts) If $f \in C^\infty(G)$, $\psi \in
    C^\infty_c(G)$ then $\hinner{X_i f}{\psi} = -\hinner{f}{X_i
      \psi}$, and $\hinner{\hat{X}_i f}{\psi} = -\hinner{f}{\hat{X}_i
      \psi}$. Hence also $\hinner{L f}{\psi} = \hinner{f}{L
      \psi}$.  \label{Xi-parts}
  \end{enumerate}
  \end{proposition}



We now mention some (rather weak) positivity-preserving properties of
the heat equation.

\begin{theorem}\label{inf-increases}
  Suppose $u \in C^{2,1}(G \times [0, T])$ has the following
    properties:
    \begin{enumerate}
    \item $u$ solves the heat equation $\left(L - \pp{t}\right)u = 0$;
    \item $u(\cdot, t)$ vanishes at infinity uniformly in $t$.  That
      is, for any $\epsilon > 0$ there exists a compact set $K \subset
      G$ such that $\sup_{K^C \times [0,T]} \abs{u} < \epsilon$. \label{vanish-at-infinity}
    \end{enumerate}
    Then $u \ge \inf_G u(\cdot, 0)$.
\end{theorem}
The proof is adapted from an argument in \cite{hunt}.
\begin{proof}

  Let $\bar{G} = G \cup \{\infty\}$ be the one-point compactification
  of $G$.  By condition \ref{vanish-at-infinity}, $u$ extends
  continuously to $\bar{G} \times [0,T]$ by setting $u(\infty, t) =
  0$.
  
  Let $A := \inf_G u(\cdot, 0)$.  Let $c > 0$ be a constant, and let
  $v(g, t) = e^{ct} (u(g,t) - A + 1)$.
  We show that $v \ge 1$.  Suppose the contrary; then $v(g_0, t_0) =
  a$ for some $0 < a < 1$, $(g_0, t_0) \in G \times (0,T]$.  Since $v$
    is a continuous function on the compact set $\bar{G} \times
    [0,T]$, there exists $h \in G$, $s > 0$ such that $v(h, s) = a$,
    and $v(g, t) > a$ for all $g \in G$, $t < s$.  (The set
    $v^{-1}((-\infty, a]) \subset \bar{G} \times [0,T]$ is compact and
  disjoint from $\bar{G} \times \{0\}$, hence there is some point $(h,
  s)$ of $v^{-1}((-\infty, a])$ whose distance from $\bar{G} \times
    \{0\}$ is minimum.)  In particular, $v_t(h, s) \le 0$.

    On the other hand, we must have $v(\cdot, s) \ge a$, so that
    $v(\cdot, s)$ has a global minimum at $h$; hence so does $u$.
    Since $L$ is an operator with positive semidefinite principal
    symbol, we must have $0 \le Lu(h, s) = u_t(h, s)$.  Then
    \begin{equation*}
      v_t(h,s) = c e^{ct} (u(h,s) - A +1) + e^{ct} u_t(h,s) \ge c e^{ct}
      (u(h,s) - A +1) = c a > 0
    \end{equation*}
    which is a contradiction.  Therefore $v \ge 1$.  By letting $c$
    tend to $0$, it follows that $u \ge A$.
\end{proof}

Moreover, nonnegative solutions of the heat equation immediately
become positive everywhere.  This is a manifestation of the idea that
``heat propagates at infinite speed.''  One way to show this is by
using the following Harnack inequality,
which is proved in \cite{purplebook}.  This seems to be a rather
larger hammer than should be needed to crush this insect, but the
author is not presently aware of a simpler approach.

\begin{theorem}[Special case of \cite{purplebook} Theorem III.2.1]\label{harnack}
\index{Harnack inequality}
  Let $M$ be a smooth manifold, $\{X_1, \dots, X_k\}$ a bracket generating
  set of vector fields on $M$, $K$ a compact subset of $M$, and $0 <
  s < t < \infty$.  Then there exists a constant $C$ such that if
  $u \in C^{2,1}(M \times [0, \infty))$ is a positive solution of $\left(\sum
    X_i^2 - \pp{t}\right)u = 0$, then
    \begin{equation*}
      \sup_{x \in K} \abs{u(x, s)} \le C \inf_{x \in K} \abs{u(x, t)}.
    \end{equation*}
\end{theorem}

We observe that the word ``positive'' in this theorem can immediately
be replaced with the word ``nonnegative,'' by replacing $u$ with $u +
\epsilon$ and letting $\epsilon \downto 0$.  (The constant $C$ is
independent of $u$.)

\begin{corollary}\label{infinite-speed}
  If $u \in C^{2,1}(M \times [0, \infty))$ is a nonnegative solution
    of $\left(\sum X_i^2 - \pp{t}\right)u = 0$, and $u(\cdot, 0)$ is
    not identically zero, then $u(\cdot, t) > 0$ for all $t > 0$.
\end{corollary}

\begin{proof}
  Fix $y \in M$, $t > 0$.  Since $u(\cdot, 0)$ is not identically
  zero, there exists $y_0 \in M$ with $u(y_0, 0) > 0$.  Then $u(y_0,
  s) > 0$ for some $0 < s < t$.  Let $K$ be a compact set containing
  both $y$ and $y_0$.  Then by Theorem \ref{harnack},
  \begin{equation*}
    0 < u(y_0, s) \le \sup_{x \in K} \abs{u(x, s)} \le C \inf_{x \in
      K} \abs{u(x, t)} \le C u(y,t).
  \end{equation*}
\end{proof}

To conclude this section, we remark that the operators $L$ and $\grad$
are not intrinsic to the group $G$, since they were constructed in
terms of the vector fields $\{X_i\}$ which in turn depend on the
choice of orthonormal basis $\{e_i\}$ for $\mathfrak{z}^\perp$.
Actually, $L$ and $f \mapsto \abs{\grad f}$ depend only on the choice
of admissible inner product $\inner{\cdot}{\cdot}$.  

Thus, given an abstract H-type group $G$, there is no canonical
sublaplacian $L$ unless further choices are made.  Selecting a
specific admissible inner product on $\mathfrak{g}$ will suffice.  (In
our treatment, with $G$ realized as $\R^{2n+m}$, we have selected the
Euclidean inner product, without loss of generality.)  

More generally, if $G$ is replaced with a subriemannian manifold $M$,
the subriemannian metric can be used to construct a canonical
sublaplacian.  The construction proceeds along similar lines to the
construction of the Laplace-Beltrami operator on a Riemannian
manifold, with the role of the Riemannian volume form taken instead by
the so-called Popp measure. 
\index{Popp measure}
 See sections 10.5--10.6 of
\cite{montgomery} for more details.  In the case of an H-type group,
the Popp measure corresponds to (a multiple of) the Haar measure, and
the sublaplacian thus obtained is the same as the one used here.

\section{The heat kernel $p_t$}\label{pt-sec}

\index{H-type group!heat kernel $p_t$}
The purpose of this section is to derive an explicit integral formula
for the heat kernel $p_t$ which is the fundamental solution of the
heat equation $\left(L - \pp{t}\right)u = 0$.  Loosely speaking, $p_t$
should be a solution with initial condition $p_0 = \delta_0$ a delta
distribution on $\R^{2n+m}$, so that solutions with other initial
conditions can be found via convolution.

We begin with an informal computation that yields a formula for a
candidate function $p_t$.  Afterwards, we verify that this function
has the properties that one would expect of a heat kernel.

Our computation proceeds by obtaining the heat kernel as the Fourier
transform of the Mehler kernel, similar in spirit to the computation
in \cite{cygan}.  For general step 2 nilpotent groups, \cite{gaveau77}
derived a formula probabilistically from a formula in \cite{levy50}
regarding the L\'evy area process; \cite{luan-zhu} extended it to the
Cayley Heisenberg group.  \cite{taylor} has a similar computation.
\cite{randall} obtains the formula for H-type groups as the Radon
transform of the heat kernel for the Heisenberg group.  In the case of
the Heisenberg group itself, \cite{klingler} has a computation using
magnetic field heat kernels.  \cite{beals-fundamental-solutions} gives
a short, direct proof by assuming \emph{a priori} that the function
should be Gaussian in form; \cite{lust-piquard} is similar in spirit
but covers a broader class of groups.  \cite{hulanicki} proceeds via
approximation of Brownian motion by random walks;
\cite{abgr-intrinsic} extends this to a broader class of nilpotent Lie
groups by using noncommutative Fourier transforms.  \cite{bgg} uses
complex Hamiltonian mechanics.

First, we record a couple of formal algebraic identities that will
help in the computations.  $A,B,C$ should be interpreted as operators;
bold indicates $k$-tuples of operators, e.g. $\vec{B} = (B_1, \dots,
B_{k})$.
\begin{align*}
  \left[A, BC\right] &= ABC - BCA + BAC - BAC = \left[A,B\right]C - B\left[C,A\right] \\
  \left[A, \ncinner{\vec{B}}{\vec{C}}\right] &= \left[A, \sum B_i C_i\right] = \sum \left[A, B_i
   C_i\right]    \\
  &= \sum (\left[A, B_i\right] C_i  -B_i\left[C_i, A\right]) =
  \ncinner{\left[A,\vec{B}\right]}{\vec{C}} -
\ncinner{\vec{B}}{\left[\vec{C},A\right]}.
\end{align*}
We also note that if $\Delta$ is the usual Laplacian on $\R^k$ and $f$
is a smooth function, we have
\begin{equation*}
  \left[\Delta, f\right] = \Delta f + 2\ncinner{\grad f}{\grad}.
\end{equation*}

Formally, we begin with the expression $p_t = e^{tL} \delta_0$.  Using
(\ref{L-threeterms}) we write $L = \Delta_x + M + \frac{1}{4}
\abs{x}^2 \Delta_z$, where $M := \ncinner{\grad_x}{J_{\grad_z} x}$ is an
``angular momentum'' operator.  We first note that $M$ commutes with
$\Delta_x$ and with $\abs{x}^2 \Delta_z$:
\begin{align*}
  \left[\Delta_x, M\right] &= -\left[\Delta_x,
    \ncinner{x}{J_{\grad_z}\grad_x}\right] \\
  &= -\ncinner{\left[\Delta_x,x\right]}{J_{\grad_z}\grad_x} +
  \ncinner{x}{\cancel{\left[\Delta_x,J_{\grad_z}\grad_x\right]}} \\
  &= -\ncinner{2\grad_x}{J_{\grad_z}\grad_x} = 0
\end{align*}
and
\begin{align*}
  \left[\ncinner{J_{\grad_z}x}{\grad_x},\frac{1}{4}\abs{x}^2 \Delta_z\right] &=
  \frac{1}{4} \left[\ncinner{J_{\grad_z}x}{\grad_x},\abs{x}^2\right] \Delta_z \\
  &= \frac{1}{4} \left(\ncinner{J_{\grad_z}x}{\left[\grad_x,\abs{x}^2\right]} -
  \ncinner{\cancel{\left[\abs{x}^2,J_{\grad_z}x\right]}}{\grad_x}\right) \Delta_z \\
&= \frac{1}{4} {\ncinner{J_{\grad_z}x}{2x}} \Delta_z = 0.
\end{align*}
Thus by the Baker-Campbell-Hausdorff formula we have
\begin{equation*}
  e^{tL} \delta_0 = e^{t(\Delta_x + \frac{1}{4} \abs{x}^2 \Delta_z)}
  e^{tM} \delta_0.
\end{equation*}

However, note that $M$ annihilates radial functions.  Indeed, if $f
\in C^{\infty}(\R^{2n+m})$ is radial in $x$, so that $f(x,z) =
f(\abs{x}, z)$, then $\grad_x f = \hat{x} f_{\abs{x}}$, where $\hat{x}
= \frac{1}{\abs{x}}x$, so
\begin{equation*}
\inner{J_{\grad_z f}x}{\grad_x f} = \inner{J_{\grad_z
    f}x}{\hat{x}f_{\abs{x}} } = \abs{x} f_{\abs{x}} \inner{J_{\grad_z
    f}\hat{x}}{\hat{x}} = 0.
\end{equation*}
Thus, since $\delta_0$ can be approximated by smooth radial functions,
it is reasonable to write $M \delta_0 = 0$ and thus $e^{tM} \delta_0 =
\delta_0$.

Now we have $p_t = e^{t(\Delta_x + \frac{1}{4}\abs{x}^2 \Delta_z)}
\delta_0$, i.e. 
\begin{equation*}
\left(\frac{\partial}{\partial t} - \left(\Delta_x +
\frac{1}{4}\abs{x}^2 \Delta_z\right)\right) p_t = 0, \quad p_0 = \delta_0.
\end{equation*}
Taking a Fourier transform in the $z$ variables, we see that
$\hat{p}_t(x,\lambda) := \int_{\R^m} e^{-i \inner{\lambda}{z}}
p_t(x,z)\,dz$ satisfies the \define{quantum harmonic
    oscillator}{quantum harmonic oscillator} equation
\begin{equation}\label{oscillator}
  \left(\frac{\partial}{\partial t}-\left(\Delta_x -
  \frac{1}{4}\abs{x}^2\abs{\lambda}^2\right)\right) u_t = 0,
\end{equation}
with initial condition $\hat{p}_0 = \delta_0^{(x)} \otimes 1$, 
where $\delta_0^{(x)}$ is the delta distribution on $\R^{2n}$.

(\ref{oscillator}) says that $\hat{p}_t$ is the \define{Mehler
  kernel}{Mehler kernel} $m_{t, \lambda} := e^{t\left(\Delta_x -
  \frac{1}{4}\abs{x}^2\abs{\lambda}^2\right)} \delta_0$, the
  fundamental solution to the \define{quantum harmonic
    oscillator}{quantum harmonic oscillator}.  We now
  derive Mehler's formula for $m_{t, \lambda}$.  Other derivations can be found
  in \cite[pp. 38, 55]{simon-functional-integration},
  \cite[p. 29]{simon-field-theory}, and references therein.

By the Trotter product formula, we expect to have
\begin{equation*}
  m_{t, \lambda} = \lim_{N \to \infty} \left(
  e^{-\frac{1}{4 N}\abs{x}^2\abs{\lambda}^2}
  e^{\frac{t}{N} \Delta} \right) \delta_0.
\end{equation*}
Since $e^{\frac{t}{N} \Delta} \delta_0$ is a Gaussian (i.e. of the
form $A(t) e^{-b(t) \abs{x}^2}$), and the family of Gaussians is
preserved by the operators $e^{\frac{t}{N} \Delta} \delta_0$ and
$e^{-\frac{1}{4 N}\abs{x}^2\abs{\lambda}^2}$, we expect that $m_{t, \lambda}$
should be a Gaussian as well.  Thus we assume 
\begin{equation}\label{mehler-gaussian}
m_{t, \lambda}(x) = A(t) e^{-b(t)
  \abs{x}^2}
\end{equation}
and solve for the functions $A(t), b(t)$.

Since $m_{t, \lambda}$ should solve the quantum harmonic oscillator
equation (\ref{oscillator}), we have
\begin{equation}
  (A'(t) - \abs{x}^2 b'(t)) e^{-b(t)  \abs{x}^2} = 
  (4 b(t)^2 \abs{x}^2 - 4 n b(t) - \frac{1}{4} \abs{\lambda}^2
  \abs{x}^2) A(t) e^{-b(t)  \abs{x}^2}
\end{equation}
whence, by equating like powers of $\abs{x}$,
\begin{align}
  A'(t) &= - 4 n A(t) b(t)  \label{A-ode} \\
  b'(t) &= - 4 b(t)^2 + \frac{1}{4} \abs{\lambda}^2. \label{b-ode}
\end{align}

We first solve the separable ODE (\ref{b-ode}).  Note that the initial
condition $m_0 = \delta_0$ suggests the initial condition $b(0) =
+\infty$, so we write
\begin{align*}
  \int_{\infty}^{b(t)} \frac{d\beta}{-4 \beta^2 + \frac{1}{4}
    \abs{\lambda}^2} &= \int_0^t \,d\tau \\
  \frac{1}{\abs{\lambda}} \coth^{-1}\left(\frac{4
    b(t)}{\abs{\lambda}}\right) &= t \quad \text{(see \cite[p. 451]{rogawski})}\\
  b(t) &= \frac{1}{4} \abs{\lambda} \coth(t \abs{\lambda}).
  \end{align*}

We substitute into (\ref{A-ode}) and separate variables to find
\begin{align*}
  A(t) &= c(\lambda) e^{-n \abs{\lambda} \int \coth(t
    \abs{\lambda})\,dt} \\
    &= c(\lambda) \sinh(t \abs{\lambda})^{-n}
\end{align*}
since $\int \coth(t \abs{\lambda}) \,dt = \frac{1}{\abs{\lambda}} \ln
\sinh(t \abs{\lambda})$ (see \cite[end pages]{rogawski}).  $c(\lambda)$ is
some ``constant'' depending on $\lambda$ but not on $t$.

Hence
\begin{equation}
  m_{t, \lambda}(x) = c(\lambda) \sinh(t \abs{\lambda})^{-n} e^{- \frac{1}{4}
    \abs{\lambda} \coth(t \abs{\lambda}) \abs{x}^2}.
\end{equation}
To determine the constant $c(\lambda)$, we assert that since $m_0 =
\delta_0$, we should have 
\begin{equation*}\lim_{t \to 0} \int_{\R^{2n}} m_{t, \lambda}(x)\,dx =
1.
\end{equation*}
  We compute
\begin{align*}
  \int_{\R^{2n}} m_{t, \lambda}(x)\,dx &= c(\lambda) \sinh(t \abs{\lambda})^{-n} \int_{\R^{2n}} e^{- \frac{1}{4}
    \abs{\lambda} \coth(t \abs{\lambda}) \abs{x}^2} \,dx \\
  &= c(\lambda) \left(\frac{4 \pi}{\abs{\lambda}\cosh(t
    \abs{\lambda})}\right)^{n} \\
  &\to c(\lambda) \left(\frac{4 \pi}{\abs{\lambda}}\right)^{n}
\end{align*}
as $t \to 0$.  
Thus we take $c(\lambda) =
\left(\frac{\abs{\lambda}}{4 \pi}\right)^{n}$ to obtain
\begin{equation}\label{mehler-formula}
  m_{t, \lambda}(x) = \left(\frac{\abs{\lambda}}{4
  \pi \sinh(t \abs{\lambda})}\right)^{n} e^{- \frac{1}{4}
    \abs{\lambda} \coth(t \abs{\lambda}) \abs{x}^2}.
\end{equation}
By construction, $m_{t,\lambda}(x)$ as defined by
(\ref{mehler-formula}) solves (\ref{oscillator}).
We remark immediately on the scaling property
\begin{equation}\label{mehler-scaling}
  m_{t, \lambda}(x) = t^{-n} m_{1, t\lambda}(x/\sqrt{t}).
\end{equation}

\begin{proposition}\label{mehler-schwartz}
  $m(t, x, \lambda) = m_{t,\lambda}(x) \in C^\infty((0,
  \infty) \times \R^{2n+m})$, and for all
  multi-indices $\alpha, \beta$ and
  all $t > 0$, $m(t, \cdot, \cdot)$ is a Schwartz function on
  $\R^{2n+m}$.
\end{proposition}

\begin{proof}
  By (\ref{mehler-scaling}) it suffices to take $t=1$.

  Set $a(s) = \left(\frac{s}{4 \pi \sinh s}\right)^n$, $b(s) =
  \frac{1}{4} s \coth s$, so that 
  \begin{equation}\label{m-a-b}
    m(1, x, \lambda) = 
    a(\abs{\lambda}) e^{-b(\abs{\lambda}) \abs{x}^2}.
  \end{equation}
  Note that $a,b$
  are entire even functions, and thus $a(\abs{\lambda}),
  b(\abs{\lambda})$ are entire functions of $\abs{\lambda}^2$, making
  them smooth functions of $\lambda$.

  We observe that $a$ is a Schwartz function.  Next, noting that
  $\coth' s= -\csch s$ and $\csch' s = -\csch s
  \coth s$, we see that the derivatives of $b(s)$ are of the form
  \begin{equation*}
    b^{(k)}(s) = s P_k(\coth s, \csch s) +
    Q_k(\coth s, \csch s)
  \end{equation*}
  for polynomials $P_k, Q_k$.  Since $\lim_{s \to \infty} \coth s = 1$
  and $\lim_{s \to \infty} \csch s = 0$, it follows that $b^{(k)}(s)$
  is of at most linear growth.  In particular, if $A(s)$ is a Schwartz
  function, so are $A(s) b(s)$ and $A(s) b'(s)$.  By induction, it
  follows that for any $\beta, k$ we have
  \begin{equation*}
    \partial_s^k \partial_\rho^l  a(s) e^{-b(s) \rho^2} =
    \sum_{i=0}^{l} A_i(s) \rho^i e^{-b(s) \rho^2}
  \end{equation*}
  where the $A_i$ are Schwartz functions.  Also, $b(s) \ge
  \frac{1}{4}$, \footnote{To see this, set $y(s) = s \cosh s - \sinh
    s$ and write $b(s) - \frac{1}{4} =
    \frac{1}{4}\csch(s) y(s)$.  We have $y(0)=0$ and $y'(s)=s \sinh s
    \ge 0$, so $y(s) \ge 0$.} so $\rho^i e^{-b(s) \rho^2} \le \rho^i
  e^{-\frac{1}{4}\rho^2}$, which is also rapidly decaying.  Thus $a(s)
  e^{-b(s) \rho^2}$ is a Schwartz function of $(\rho, s)$.  By
  (\ref{m-a-b}), the proof is complete.
\end{proof}

We now define the heat kernel to be the inverse Fourier transform of
the Mehler kernel.

\begin{definition}
  \index{H-type group!heat kernel $p_t$!formula}
  \define{}{H-type group!heat kernel $p_t$}
  The \define{heat kernel}{heat kernel} on an H-type group $G =
  \R^{2n+m}$ is the function $p_t$ given by 
  \begin{align}
    p_t(x,z) &= \frac{1}{(2\pi)^m} \int_{\R^m} e^{i \inner{\lambda}{z}}
    m_{t,\lambda}(x)\,d\lambda \nonumber \\
    &= \frac{1}{(2\pi)^m} \int_{\R^m} e^{i \inner{\lambda}{z} - \frac{1}{4}
    \abs{\lambda} \coth(t \abs{\lambda}) \abs{x}^2} \left(\frac{\abs{\lambda}}{4
  \pi \sinh(t
  \abs{\lambda})}\right)^{n}\,d\lambda. \label{pt-formula-1}
  \end{align}
\end{definition}

The next proposition suggests that $p_t$ deserves to be called the
heat kernel, since it is the fundamental solution to the heat equation.

\begin{proposition}\label{pt-props}
  \begin{enumerate}
  \item   $p_t(x, z) \in C^\infty((0,
    \infty) \times \R^{2n+m})$, and $p_t$ is a Schwartz function on
    $\R^{2n+m}$ for each $t > 0$. \label{pt-schwartz}
  \item $\left(L - \pp{t}\right)p_t = 0$.
  \item $\int_{G} p_t(x,z)\,d\haar = 1$ for all $t > 0$. \label{pt-integral}
  \item For $\alpha > 0$, $p_t(x,z) = \alpha^{2(n+m)} p_{\alpha^2
    t}(\varphi_\alpha(x,z))$.  In particular, $\lim_{t \to 0} p_t(x,z)
    = 0$ for all $(x,z) \ne 0$. \label{pt-scaling}
  \item $p_t$ is a radial function (i.e. $p_t(x,z)$ depends only on
    $\abs{x}, \abs{z}$) as is $\pp{t} p_t = L p_t$.  In particular,
    $p_t(g^{-1}) = p_t(g)$ and $L p_t (g^{-1}) = L p_t (g)$. \label{pt-radial}
  \item $p_t$ vanishes uniformly at infinity.  That is, for any
    $\epsilon > 0$ there exists a compact set $K \subset G$ such that
    $\sup_{t > 0,\,g \in K^C} \abs{p_t(g)} < \epsilon$.\label{pt-infinity}
  \end{enumerate}
\end{proposition}

\begin{proof}
  \begin{enumerate}
  \item Clear from Proposition \ref{mehler-schwartz} and properties of
    the Fourier transform.
  \item Clear by properties of the Fourier transform, since
    $m_{t,\lambda}$ solves (\ref{oscillator}).
  \item By Fourier inversion, we have
    \begin{align*}
      \int_{\R^{2n}} \int_{\R^m} p_t(x,z)\,dz\,dx &=
      \frac{1}{(2\pi)^m} \int_{\R^{2n}} \int_{\R^m} \int_{\R^m}
      e^{i \inner{\lambda}{z}} m_{t,\lambda}(x)\,d\lambda\,dz\,dx \\
      &=  \int_{\R^{2n}} m_{t, 0}(x)\,dx \\
      &=  \int_{\R^{2n}} \left(\frac{1}{4
  \pi t}\right)^{n} e^{- \frac{1}{4t} \abs{x}^2}. \,dx = 1.
    \end{align*}
    \item To see the scaling $p_t(x,z) = \alpha^{2(n+m)} p_{\alpha^2
    t}(\varphi_\alpha(x,z))$, make the change of variables $\lambda =
      \alpha^2 \lambda'$ in (\ref{pt-formula-1}).  Taking $\alpha =
      t^{-1/2}$ we get 
      \begin{equation*}\lim_{t \to 0} p_t(x,z) = \lim_{\alpha \to
        \infty}  \alpha^{2(n+m)} p_1(\alpha x,\alpha^2 z) = 0
      \end{equation*} by
      item 1.
    \item It is clear from (\ref{pt-formula-1}) that $p_t$ is radial
      in $x$.  To see it is radial in $z$, suppose $\abs{z'} =
      \abs{z}$, so that $z' = Tz$ for some orthogonal matrix $T$.
      Then making the change of variables $\lambda = T \lambda'$ in
      (\ref{pt-formula-1}) shows that $p_t(x, z) = p_t(x, z')$.
    \item Suppose first that $t \le 1$.  By item \ref{pt-scaling} and
      the fact that $p_1$ is a Schwartz function, we have
      \begin{align*}
        \abs{p_t(x,z)} &= t^{-(n+m)} \abs{p_1(t^{-1/2}x, t^{-1} z)} \\
        &\le C t^{-(n+m)} \left(\abs{t^{-1/2}x}^2 + \abs{t^{-1}
          z}^2\right)^{-(n+m)} \\
        &= C \left(\abs{x}^2 + t^{-1} \abs{z}^2\right)^{-(n+m)} \\
        &\le C \left(\abs{x}^2 + \abs{z}^2\right)^{-(n+m)}.
      \end{align*}
      When $t \ge 1$, we have
      \begin{align*}
        \abs{p_t(x,z)} &= t^{-(n+m)} \abs{p_1(t^{-1/2}x, t^{-1} z)} \\
        &\le C' t^{-(n+m)} \left(\abs{t^{-1/2}x}^2 + \abs{t^{-1}
          z}^2\right)^{-(n+m)/2} \\
        &= C' \left(t \abs{x}^2 + \abs{z}^2\right)^{-(n+m)/2} \\
        &\le C' \left(\abs{x}^2 + \abs{z}^2\right)^{-(n+m)/2}.
      \end{align*}
      Thus, taking $K = \{ \abs{x}^2 + \abs{z}^2 \le
      \min\{(C^{-1}\epsilon)^{n+m}, (C'^{-1}\epsilon)^{(n+m)/2}\} \}$ suffices.
  \end{enumerate}
\end{proof}

\begin{definition}
  The 
  \define{heat semigroup}{H-type group!heat semigroup $P_t$}
  $P_t, t \ge 0$ is the
  one-parameter family of operators given by the convolution (see
  Definition \ref{convolution-def})
  \begin{equation}\label{Pt-convolution}
    \begin{split}
    P_t f (g) &:= (f * p_t)(g) , \quad t > 0 \\
    P_0 f &:= f
    \end{split} 
  \end{equation}
  for any distribution $f \in \mathcal{D}(G)$ such that the integral
  makes sense for
  all $g \in G$ and all $t > 0$.  Note that we can make the change of
  variables $k = k^{-1}$ in Definition \ref{convolution-def} and use
  the fact that $p_t(k) = p_t(k^{-1})$ (item \ref{pt-radial} of
  Proposition \ref{pt-props}) to write
  \begin{equation}\label{Pt-integral-symmetry}
    P_t f (g) = \int_G f(g k) p_t(k)\,d\haar(k).
  \end{equation}
\end{definition}

The following facts about $P_t$ are almost immediate.

\begin{proposition}\label{Pt-props}
  \begin{enumerate}
  \item If $f$ is a tempered distribution on $G$, then $P_t f \in
    C^\infty(G)$ and $P_t f$ satisfies the heat equation $\left(L -
    \pp{t}\right)P_t f = 0$.
  \item If $f$ is a (tempered distribution, $L^p(G)$ function,
    uniformly continuous function, bounded continuous function), then
    as $t \to 0$, $P_t f \to f$ (in the sense of distributions, in
    $L^p$, uniformly, uniformly on compact subsets of $G$)
    respectively. \label{Pt-continuous}
  \end{enumerate}
\end{proposition}

\begin{proof}
  \begin{enumerate}
  \item That $P_t f \in C^\infty(G)$ follows from item
    \ref{convolution-smoothing} of Proposition
    \ref{convolution-props}.  By differentiating under the integral
    sign with respect to $t$ and using item \ref{Xi-convolve} of
    Proposition \ref{Xi-integration}, we have
    \begin{equation*}
      \pp{t} P_t f = f * \pp{t} p_t = f * L p_t = L P_t f.
    \end{equation*}
  \item This is a standard ``approximate delta function'' argument,
    making use of items \ref{pt-integral} and \ref{pt-scaling} of
    Proposition \ref{pt-props}.
  \end{enumerate}
\end{proof}

\begin{theorem}\label{pt-positive}
  \index{H-type group!heat kernel $p_t$!positivity}
  $p_t > 0$ for all $t > 0$.
\end{theorem}

\begin{proof}
  Let $f \in C^\infty_c(G)$ be nonnegative, with compact support $K$.
  We first show that $u(g,t) := P_t f(g)$ satisfies the hypotheses of
  Theorem \ref{inf-increases}.  Certainly $u$ solves the heat
  equation, and $u(\cdot, 0) = f \ge 0$.  To show $u$ vanishes
  uniformly at infinity, fix $\epsilon > 0$.  Let $K$ be a compact set
  such that $\abs{p_t(g)} \le \epsilon/\norm{f}_\infty$ for all $t >
  0$, $g \in K^C$ whose
  existence is guaranteed by item \ref{pt-infinity} of Proposition;
  note that $0 \in K$.
  Let $K' = K \groupop \supp f = \{ k \groupop h : k \in K, h \in
  \supp f\}$; note that $\supp f \subset K'$.  Suppose $g \notin K'$.
  If $t = 0$ we have $P_t f(g) = f(g) = 0$ since $g \notin \supp f$.
  If $t > 0$ we have
  \begin{align*}
    \abs{P_t f(g)} = \abs{p_t * f(g)} &\le \norm{f}_\infty \sup_{k \in
      \supp f} \abs{p_t(g
      \groupop k^{-1})}.
  \end{align*}
  But for $k \in \supp f$, we have $g \groupop k^{-1} \notin K$.  (If
  $g \groupop k^{-1} \in K$, then $g \in K \groupop k \subset K
  \groupop \supp f = K'$, contrary to our choice $g \notin K'$.)
  Therefore by definition of $K$ we have $\abs{p_t(g
      \groupop k^{-1})} \le \epsilon/\norm{f}_\infty$, so that
  $\abs{P_t f(g)} \le \epsilon$.

  Thus by Theorem \ref{inf-increases} we have $P_t f \ge 0$
  for $t \in [0, T]$ for any $T > 0$.  We now replace $f$ with
  a sequence of nonnegative approximate delta functions to see that
  $p_t \ge 0$ for all $t > 0$.  Corollary \ref{infinite-speed} does
  not apply directly to $p_t$, because $p_t$ is not continuous up to
  $t = 0$; however, it does apply to $p_{t+t_0}$ for arbitrary fixed
  $t_0 > 0$.  We conclude that $p_{t + t_0} > 0$ for all $t$, and
  since $t_0$ was arbitrary, $p_t > 0$ for all $t$.
\end{proof}

The reader unfamiliar with the functional analysis machinery in the
following theorem may omit it, as it is not essential to the remainder
of the dissertation, or refer to Chapter VIII of
\cite{reed-simon-vol1}.

\begin{theorem}\label{Pt-semigroup}
  \index{H-type group!heat semigroup $P_t$}
  $P_t$ is a strongly continuous self-adjoint contraction semigroup on
  $L^2(G)$, and if $L$ is viewed as a unbounded operator on $L^2(G)$
  with domain $\mathcal{D}(L) = C^\infty_c(G)$, then the infinitesimal
  generator of $P_t$ is $\bar{L}$, the closure of $L$, which is
  self-adjoint.  That is, $P_t = e^{t\bar{L}}$.
\end{theorem}

\begin{proof}
  It is trivial that $P_0 = I$ is a self-adjoint contraction.

  It follows from Theorem \ref{pt-positive} and item
  \ref{pt-integral} of Proposition \ref{pt-props} that $\int_G
  \abs{p_t}\,d\haar = \int_G
  p_t\,d\haar = 1$, so by Young's inequality (item
  \ref{convolution-young} of Proposition \ref{convolution-props}) $P_t$ is a
  contraction on $L^2(G)$ for each $t > 0$.  

  By item
  \ref{convolution-adjoint} of Proposition \ref{convolution-props}, we
  have for any $f_1, f_2 \in L^2(G)$ that
  \begin{equation*}
    \hinner{P_t f_1}{f_2} = \hinner{f_1 * p_t}{f_2} = \hinner{f_1}{f_2
      * \tilde{p}_t}.
  \end{equation*}
  But $p_t$ is a real-valued radial function (see item \ref{pt-radial}
  of Proposition \ref{pt-props}), so $\tilde{p}_t = p_t$ and thus
  $\hinner{P_t f_1}{f_2} = \hinner{f_1}{P_t f_2}$.  $P_t$ is
  self-adjoint for each $t > 0$.

  To show $P_t$ is a semigroup, it suffices to verify that $p_s * p_t
  = p_{s+t}$.  One approach is to notice that for any $s > 0$, $u(g,
  t) := p_s * p_t - p_{s+t}$ is a solution of the heat equation which
  vanishes uniformly at infinity and satisfies $u(\cdot, 0) \equiv 0$.
  By applying Theorem \ref{inf-increases} to $u$ and $-u$, we must
  have $u \equiv 0$.  

  That $P_t$ is strongly continuous in $t$ follows from item
  \ref{Pt-continuous} of Proposition \ref{Pt-props}.

  Since $P_t$ is a strongly continuous self-adjoint contraction
  semigroup, it has a self-adjoint infinitesimal generator $A$.  To
  show that $A = \bar{L}$, we show first that $L \subset A$, so that
  \begin{equation}\label{operator-inclusions}
    \bar{L} \subset A = A^* \subset L^*
  \end{equation}
  (since $A$ is closed).  Next, we show that $L$ is essentially
  self-adjoint, i.e. $\bar{L} = L^*$, so that equality must hold in
  (\ref{operator-inclusions}).
  (Note that item \ref{Xi-parts} of
  Proposition \ref{Xi-integration} verifies explicitly that $L$ is
  symmetric, i.e. $L \subset L^*$, which also follows from $L \subset
  A$ since $A$ is self-adjoint.)

  To see that $L \subset A$, let $f \in C^\infty_c(G)$.  First, we
  observe that for $t > 0$, $\frac{1}{\epsilon}(p_{t + \epsilon}
  - p_t) \to \pp{t}p_t = L p_t$ as $\epsilon \downto 0$, not only
  pointwise but also in $L^1(G)$.  To verify the latter, we note by the
  usual combination of the mean value theorem and dominated
  convergence that it suffices to show 
  \begin{equation*}\int_G \sup_{\tau \in [t, t+\epsilon]}
    \abs{\pp{\tau} p_\tau(g)}\,d\haar(g) < \infty.
  \end{equation*}
  By item \ref{pt-scaling} of Proposition \ref{pt-props}, we have for
  $\tau \in [t, t+\epsilon]$:
  \begin{align*}
    \abs{\pp{\tau} p_\tau(x,z)} &= \abs{\pp{\tau} \tau^{-(n+m)}
      p_1(\tau^{-1/2} x, \tau^{-1} z)} \\
    &\le \abs{ -(n+m)\tau^{-(n+m+1)}
      p_1(\tau^{-1/2} x, \tau^{-1} z)} \\
    &\quad+ \abs{\tau^{-(n+m)} \inner{\grad_x
      p_1(\tau^{-1/2} x, \tau^{-1} z)}{x}
      \left(-\frac{1}{2}\tau^{-3/2}\right)} \\
    &\quad+   \abs{\tau^{-(n+m)} \inner{\grad_z
      p_1(\tau^{-1/2} x, \tau^{-1} z)}{z} (-\tau^{-2})} \\
    &\le C \abs{-(n+m)\tau^{-(n+m+1)} - \frac{1}{2}\tau^{-(n+m+3/2)} -
      \tau^{-(n+m+2)}} 
    \\ &\quad\times \left(1 + \tau^{-1/2} \abs{x} + \tau^{-1}
    \abs{z}\right)^{-(2n+m+2)} \\
    &\le  C' t^{-(n+m+1)} \left(1 + (t+\epsilon)^{-1/2} \abs{x} + (t+\epsilon)^{-1}
    \abs{z}\right)^{-(2n+m+2)}
  \end{align*}
  whose integral over $G = \R^{2n+m}$ is finite, where we used the fact
  that $p_1$ is a Schwartz function to bound $p_1$ and its derivatives
  in terms of the integrable function 
\begin{equation*}
  \left(1 + \tau^{-1/2} \abs{x} + \tau^{-1}
    \abs{z}\right)^{-(2n+m+2)}.
\end{equation*}
  Therefore, by using Young's inequality and item \ref{Xi-convolve} of
  Proposition \ref{Xi-integration}, we have that 
  \begin{equation}\label{Pt-L2-diff}
    \frac{1}{\epsilon} (P_{t + \epsilon}
  - P_t) f \to f * (L p_t) = L P_t f \text{ in } L^2(G).
  \end{equation}

  Next, we note that
  \begin{align*}
    L P_t f(g) = f * (L p_t) &= \int_G f(g \groupop k^{-1}) L
    p_t(k)\,d\haar(k) \\
    &= \int_G f(g \groupop k) L
    p_t(k)\,d\haar(k)
    \intertext{(where we have made the change of variables $k \to
      k^{-1}$ and used item \ref{pt-radial} of Proposition
      \ref{pt-props} to see $L p_t (k^{-1}) = L p_t(k)$)}
    &= \hinner{f \circ L_g}{L p_t} && \text{where $L_g$ is left translation} \\
    &= \hinner{L(f \circ L_g)}{p_t}  && \text{by item \ref{Xi-parts}
      of Proposition \ref{Xi-integration}} \\
    &= \hinner{(Lf) \circ L_g}{p_t} \\
    &= \int_G Lf(g \groupop k) p_t(k) \,d\haar(k) \\
    &= \int_G Lf(g \groupop k^{-1}) p_t(k) \,d\haar(k) && \text{as
      before, since $p_t(k^{-1}) = p_t(k)$} \\
    &= P_t L f(g).
  \end{align*}
  Thus $P_t$ commutes with $L$, as one would expect.

  Now for any $t > 0$ we have
  \begin{equation*}
    (P_{s+t} - P_s)f = \int_s^t L P_\tau f\,d\tau = \int_s^t P_\tau L f\,d\tau
  \end{equation*}
  where the integral is a Riemann integral of an $L^2(G)$-valued
  continuous function of $\tau$, and we have used a corresponding
  version of the fundamental theorem of calculus thanks to
  (\ref{Pt-L2-diff}).  The integral is $L^2$-continuous in $s$, so
  letting $s \downto 0$ and using the semigroup property and the
  strong continuity of $P_t$, we have
  \begin{equation*}
    (P_t - I)f = \int_0^t P_\tau L f\,d\tau.
  \end{equation*}
  So
  \begin{align*}
    \norm{\frac{P_t f - f}{t} - Lf} &= \frac{1}{t} \norm{\int_0^t
      (P_\tau - I) Lf\,d\tau} \\
    &\le \sup_{\tau \in [0,t]} \norm{(P_\tau - I) Lf} \to 0
  \end{align*}
  by the strong continuity of $P_t$.  Therefore $L$ agrees with $A$,
  the generator of $P_t$, on $\mathcal{D}(L) = C^\infty_c(G)$, so $L
  \subset A$.

  \index{H-type group!sublaplacian $L$!essentially self-adjoint|(}

  To conclude, we show that $\bar{L} = L^*$, so that $L$ is
  essentially self-adjoint.  I learned the following
  standard argument from L. Gross.    Consider
  the vector fields $X_i$ as acting on $C^\infty(G)$, and likewise
  $L_0 := \sum X_i^2$, which is an extension of $L$.  Suppose
  first that $f \in \mathcal{D}(L^*) \cap C^\infty(G)$.  Then if $h
  \in C^\infty_c(G)$, we have
  \begin{equation*}
    \hinner{L^* f}{h} = \hinner{f}{L h} = \hinner{L_0 f}{h}
  \end{equation*}
  by item \ref{Xi-parts} of Proposition \ref{Xi-integration}.  

  Next, we show that $\abs{\grad f} \in L^2(G)$.  Use the Urysohn
  lemma to choose $\psi \in C^\infty_c(G)$ such that $\psi \equiv 1$
  on some neighborhood of $0$, and let $\psi_n = \psi \circ
  \varphi_{1/n}$.  We then have $\psi_n \to 1$ boundedly, and by
  (\ref{Xi-dilation}) we have $\grad \psi_n = \frac{1}{n} (\grad \psi)
  \circ \varphi_{1/n} \to 0$ uniformly and $L \psi_n = \frac{1}{n^2}
  (L \psi) \circ \varphi_{1/n} \to 0$ uniformly.  Then, integrating by
  parts as in item
  \ref{Xi-parts} of Proposition \ref{Xi-integration}, we have
  \begin{align*}
    \int_G \psi_n \abs{\grad{f}}^2\,d\haar &= \sum_i \int_G \psi_n
    (X_i f)^2\,d\haar \\
    &= -\sum_i \int_G f X_i(\psi_n X_i f)\,d\haar \\
    &= -\sum_i \int_G f \psi_n X_i^2 f\,d\haar - f X_i \psi_n
    X_i f \,d\haar \\
    &= - \int_G \psi_n f L_0 f + \sum_i X_i \psi_n X_i (f^2)\,d\haar \\
    &= - \int_G \psi_n f L^* f + f^2 L \psi_n \,d\haar.    
  \end{align*}
  Letting $n \to \infty$, so that $\psi_n \to 1$ and $L \psi_n \to 0$
  boundedly, the dominated convergence theorem gives
  \begin{equation*}
    \norm{\grad f} = -\hinner{f}{L^* f} < \infty.
  \end{equation*}

  Now we have $\psi_n f \in C^\infty_c$ and $\psi_n f \to f$ by
  dominated convergence.  We also have
  \begin{align*}
    L^* (\psi_n f) = L_0 (\psi_n f) = (L_0 \psi_n) f + \psi_n L_0 f +
    2 \inner{\grad \psi_n}{\grad f}.
  \end{align*}
  As $n \to \infty$, we find $L (\psi_n f) = L^* (\psi_n f) \to L_0 f
  = L^* f$.  Thus, $\mathcal{D}(L^*) \cap C^\infty \subset
  \mathcal{D}(\bar{L})$.

  Finally, suppose $f \in \mathcal{D}(L^*)$, and let $\phi_n \in
  C^\infty_c(G)$ be a sequence of ``approximate delta functions,'' so
  that $\phi_n * f \to f$ and $\phi_n * L^* f \to L^* f$.
  For any $h \in C^\infty_c(G)$ we then have
  \begin{align*}
    \hinner{\phi_n * L^* f}{h} &= \hinner{L^* f}{\tilde{\phi_n} * h}
    &&\text{(item \ref{convolution-adjoint} of Proposition
      \ref{convolution-props})} \\
      &= \hinner{f}{L(\tilde{\phi_n} * h)} && \text{as $\tilde{\phi_n},
        h, \tilde{\phi_n} * h \in C^\infty_c$} \\
      &= \hinner{f}{\tilde{\phi_n} * Lh} \\
      &= \hinner{\phi_n * f}{Lh}
  \end{align*}
  so that $\phi_n * f \in \mathcal{D}(L^*) \cap C^\infty(G)$ and
  $\bar{L}(\phi_n * f) = L^*(\phi_n * f) = \phi_n * L^* f \to L^* f$.
  Thus $\mathcal{D}(L^*) \subset \mathcal{D}(\bar{L})$, which completes
  the proof.
  \index{H-type group!sublaplacian $L$!essentially self-adjoint|)}

\end{proof}

Chapter \ref{h-type-chapter}, in part, is adapted from material
awaiting publication as \heatcite and \gradcite.  The dissertation
author was the sole author of these papers.

\chapter{Subriemannian Geometry}\label{subriemannian-chapter}

H-type groups lend themselves naturally to the structure of a
subriemannian manifold.  The geometry that arises in this sense will
be crucial in the sequel, particularly its geodesics and the
corresponding Carnot-Carath\'eodory distance function.  The goal of
this chapter will be to describe the necessary theory and obtain
explicit formulas for the geodesics and the distance function.  The
computation is a straightforward application of Hamiltonian mechanics,
but we have not previously seen it appear in the literature in the
case of H-type groups.  The corresponding computation for the
Heisenberg groups (where the center has dimension $m=1$) appeared in
\cite{bgg} as well as \cite{calin-book}; a computation for $m \le 7$,
which could be extended without great difficulty, can be found in the
preprint \cite{calin-H-type}.

\section{Subriemannian manifolds, geodesics and Hamiltonian mechanics}\label{subriemannian-sec}

\begin{definition}
  A \define{subriemmanian manifold}{subriemannian manifold} is a
  smooth manifold $Q$ together with a subbundle $\mathcal{H}$ of $TQ$
  (the \define{horizontal bundle}{horizontal bundle} or
  \define{horizontal distribution}{horizontal distribution}, whose
  elements are \define{horizontal vectors}{horizontal vectors}) and a
  metric $\inner{\cdot}{\cdot}_q$ on each fiber $\mathcal{H}_q$,
  depending smoothly on $q \in Q$.  $\mathcal{H}$ is
  \define{bracket-generating}{subriemannian
    manifold!bracket-generating} at $q$ if there is a local frame
  $\{X_i\}$ for $\mathcal{H}$ near $q$ such that
\begin{equation*}
  \spanop\{X_i(q), [X_i,
    X_j](q), [X_i, [X_j, X_k]](q), \dots\} = T_q Q.
\end{equation*}
\end{definition}

An H-type group $G$ can naturally be equipped as a subriemannian
manifold, by letting $\mathcal{H}_g := \{X(g) : X \in
\mathfrak{z}^\perp\}$, and using the inner product on $\mathfrak{g}$
as the metric on $\mathcal{H}$.  In other words, $\mathcal{H}_g$ is
spanned by $\{X_1(g), \dots, X_{2n}(g)\}$, which give it an
orthonormal basis.  The bracket generating condition is
obviously satisfied, since $\mathfrak{g} = \mathfrak{z}^\perp \oplus
[\mathfrak{z}^\perp,\mathfrak{z}^\perp]$.

\begin{definition}
  Let $\gamma : [0,1] \to Q$ be an absolutely continuous path.  We say
  $\gamma$ is \define{horizontal}{horizontal path} if $\dot{\gamma}(t)
  \in \mathcal{H}_{\gamma(t)}$ for almost every $t \in [0,1]$.  In such a case
  we define the \emph{length} of $\gamma$ as $\ell(\gamma) := \int_0^1
  \sqrt{\inner{\dot{\gamma}(t)}{\dot{\gamma}(t)}_{\gamma(t)}} \,dt$.
  The \define{Carnot-Carath\'eodory distance}{Carnot-Carath\'eodory
    distance} $d : Q \times Q \to [0,\infty]$ is defined by
  \begin{equation}
    d(q_1,q_2) = \inf\{\ell(\gamma) : \gamma(0)=q_1, \gamma(1)=q_2, \gamma
    \text{ horizontal}\}.
  \end{equation}
\end{definition}

Under the bracket generating condition, the Carnot-Carath\'eodory
distance is well behaved.  We refer the reader to Chapter 2 and
Appendix D of \cite{montgomery} for proofs of the following two
theorems, the first of which is largely a restatement of Chow's
theorem (Theorem \ref{chow-thm}).

\fixnotme{ball-box theorem here}

\begin{theorem}[Chow]\label{chow-thm-subriemannian}
  If $\mathcal{H}$ is bracket generating and $Q$ is connected, then
  any two points $q_1, q_2 \in Q$ are joined by a horizontal path whose length
  is finite.  Thus $d(q_1, q_2) < \infty$, and $d$ is easily seen to
  be a distance function on $Q$.  The topology induced by $d$ is equal
  to the manifold topology for $Q$.
\end{theorem}

\fixnotme{Discuss Frobenius theorem as a partial converse?}

\begin{theorem}
  \index{geodesic!existence of}
  If $Q$ is complete under the Carnot-Carath\'eodory distance $d$,
  then the infimum in the definition of $d$ is achieved; that is, any
  two points $q_1, q_2 \in Q$ are joined by at least one shortest
  horizontal path.
\end{theorem}

\fixnotme{Include proofs?}

One way to compute the Carnot-Carath\'eodory distance is to find such
a shortest path and compute its length.  To find a shortest path, we
use Hamiltonian mechanics, following Chapters 1 and 5 of
\cite{montgomery}.  Roughly speaking, it can be shown that a length
minimizing path also minimizes the energy $\frac{1}{2}\int_0^1
\norm{\dot{\gamma}(t)}\,dt$, and as such should solve Hamilton's
equations of motion.  The argument uses the method of Lagrange
multipliers, and requires that the endpoint map taking horizontal
paths to their endpoints has a surjective differential.  This always
holds in the Riemannian setting, but is not generally true in
subriemannian geometry; the Martinet distribution (see Chapter 3 of
\cite{montgomery}) is a counterexample in which some shortest paths do
not satisfy Hamilton's equations.
\index{Martinet distribution}
    Additional assumptions on
$\mathcal{H}$ are needed.  One which is sufficient (but certainly not
necessary) is that the distribution be \emph{fat}:

\begin{definition}\label{fat-def}
  Let $\Theta$ be the canonical $1$-form on the cotangent bundle
  $T^*Q$ \fixnotme{definition would be nice}, $\omega = d\Theta$ the
  canonical symplectic $2$-form, and let $\mathcal{H}^0 := \{ p_q \in
  T^*Q : p_q(\mathcal{H}_q) = 0\}$ be the annihilator of
  $\mathcal{H}$.  (Note $\mathcal{H}^0$ is a sub-bundle, and hence
  also a submanifold, of $T^*Q$.)  We say $\mathcal{H}$ is
  \define{fat}{fat} if $\mathcal{H}^0$ is symplectic away from
  the zero section.  That is, if $p_q \in \mathcal{H}^0$ is not in the
  zero section, $v \in T_{p_q}\mathcal{H}^0$, and $\omega(v, w)=0$ for
  all other $w \in T_{p_q}\mathcal{H}^0$, then $v=0$.
\end{definition}

\begin{definition}
  If $(Q, \mathcal{H}, \inner{\cdot}{\cdot})$ is a subriemannian
  manifold, the subriemannian \define{Hamiltonian}{Hamiltonian} $H : T^*Q \to \R$ is
  defined by
  \begin{equation}\label{hamiltonian-subriemannian}
    H(p_q) = \sum_i p_q(v_i)^2
  \end{equation}
  where $\{v_i\}$ is an orthonormal basis for $(\mathcal{H}_q,
  \inner{\cdot}{\cdot}_q)$.  It is clear that this definition is
  independent of the chosen basis.  Let the \define{Hamiltonian vector
    field}{Hamiltonian vector field} $X_H$ on $T^*Q$ be the unique
  vector field satisfying $dH + \omega(X_H, \cdot) = 0$ (as elements
  of $T^*T^*Q$).  $X_H$ is well defined because $\omega$ is
  symplectic.  \define{Hamilton's equations of motion}{Hamilton's
    equations of motion} are the ODEs for the integral curves of
  $X_H$.
\end{definition}

The following theorem summarizes (a special case of) the argument of
Chapters 1 and 5 of \cite{montgomery}.

\begin{theorem}\label{fat-hamilton}
  If $\mathcal{H}$ is fat, then any length minimizing path $\sigma :
  [0,1] \to Q$, when parametrized with constant speed, is also energy
  minimizing and is the projection onto $Q$ of a path ${\gamma} :
  [0,1] \to T^*Q$ which satisfies Hamilton's equations of motion:
  $\dot{\gamma}(t) = X_H(\gamma(t))$.
\end{theorem}

\fixnotme{include a proof}

\section{Geodesics for H-type groups}

In this section, we will verify that Theorem \ref{fat-hamilton}
applies to H-type groups, and then find a formula for the geodesics
(shortest paths) by
solving Hamilton's equations.  To begin, we adopt a coordinate system
for the cotangent bundle $T^*G$.

\begin{notation}
  Let $(x, z, \xi, \eta) : T^*G \to \R^{2n} \times \R^{m} \times
  \R^{2n} \times \R^m$ be the coordinate system on $T^*G$ such that
  $x^i(p_g) = x^i(g)$, $z^j(p_g)=z^j(g)$, $\xi_i(p_g) =
  p(\frac{\partial}{\partial x^i})$, $\eta_j(p_g) =
  p(\frac{\partial}{\partial z^j})$.  That is,
  \begin{equation*}
    p_g = \left(x(g), z(g),
   \sum_i \xi_i dx^i + \sum_j \eta_j dz^j\right).
  \end{equation*}
  In these
  coordinates, the canonical $2$-form $\omega$ has the expression
  $\omega = \sum_i  d\xi_i \wedge dx^i + \sum_j d\eta_j \wedge dz^j$.
\end{notation}

\begin{proposition}
If $G$ is an H-type group with horizontal distribution $\mathcal{H}$
spanned by the vector fields $X_i$, then $\mathcal{H}$ is fat.
\index{H-type group!fat}
\index{fat}
\end{proposition}

\begin{proof}
For an H-type group $G$, we have $p_g \in \mathcal{H}^0$ iff
$p_g(X_i(g)) = 0$ for all $i$.  We can thus form a basis for
$\mathcal{H}^0_g \subset T^*_g G$ by
\begin{align*}
  w^j &= dz^j - \sum_i dz^j(X_i(g)) dx^i \\
  &= dz^j - \frac{1}{2}\sum_i \innerp{J_{u_j} x(g)}{e_i} dx^i.
\end{align*}
Expressing $p_g$ in this basis as $p_g = \sum_j \theta_j w^j$ yields a
system of coordinates $(x, z, \theta)$ for $\mathcal{H}^0$, where
$\theta$ can be identified with the element $(\theta^1, \dots,
\theta^m)$ of $\R^m$.  In terms of the coordinates $(x,z,\xi,\eta)$
for $T^*G$, we have $\eta = \theta$, $\xi = -\frac{1}{2}J_\theta x$.

So let $\gamma : (-\epsilon, \epsilon) \to \mathcal{H}^0$ be a curve
in $\mathcal{H}^0$ which avoids the zero section.  $\dot{\gamma}(0)$
is thus a generic element of $T \mathcal{H}^0$.  We write
$\gamma(t)$ in coordinates as $(x(t), z(t), \theta(t))$, where
$\theta(t) \ne 0$. In terms of the coordinates $(x,z,\xi,\eta)$ on
$T^*G$, we have $\eta(t)=\theta(t), \xi(t) = -\frac{1}{2}J_{\theta(t)}x(t)$.
Differentiating the latter gives 
\begin{equation*}
\dot{\xi}(t) =
-\frac{1}{2} (J_{\dot{\theta}(t)}x(t) + J_{\theta(t)}\dot{x}(t)).
\end{equation*}

Suppose that for all other such curves $\gamma'$ which avoid the zero
section and satisfy $\gamma'(0)=\gamma(0)$, we have
$\omega(\dot{\gamma}(0), \dot{\gamma}'(0)) = 0$.  In terms of
coordinates,
\begin{align*}
  0 = \omega(\dot{\gamma}(0), \dot{\gamma}'(0)) 
  &= \sum_i (\dot{\xi}_i(0) \dot{x}'^i(0) -
  \dot{\xi}'_i(0)\dot{x}^i(0)) + \sum_j (\dot{\eta}_j(0)\dot{z}'^j(0)
  - \dot{\eta}'_j(0)\dot{z}^j(0)) \\
  &= \inner{\dot{\xi}(0)}{\dot{x}'(0)} -
  \inner{\dot{\xi}'(0)}{\dot{x}(0)} +
  \inner{\dot{\eta}(0)}{\dot{z}'(0)} -
  \inner{\dot{\eta}'(0)}{\dot{z}(0)} \\
  &= -\frac{1}{2} \inner{J_{\dot{\theta}(0)}x(0) +
    J_{\theta(0)}\dot{x}(0)}{\dot{x}'(0)} +\frac{1}{2} 
  \inner{J_{\dot{\theta}'(0)}x'(0) +
    J_{\theta'(0)}\dot{x}'(0)}{\dot{x}(0)} \\
  &\quad +
  \inner{\dot{\theta}(0)}{\dot{z}'(0)} -
  \inner{\dot{\theta}'(0)}{\dot{z}(0)} \\
  &= \frac{1}{2} \inner{x(0)}{J_{\dot{\theta}(0)} \dot{x}'(0) +
    J_{\dot{\theta}'(0)} \dot{x}(0)} + \inner{J_{\theta(0)} \dot{x}'(0)}{\dot{x}(0)} \\
  &\quad +
  \inner{\dot{\theta}(0)}{\dot{z}'(0)} -
  \inner{\dot{\theta}'(0)}{\dot{z}(0)}.
\end{align*}

For arbitrary $u \in \R^m$, take $\gamma'(t) = (x(0), z(0)+tu,
\theta(0))$; then $0 = \omega(\dot{\gamma}(0), \dot{\gamma}'(0)) =
\inner{\dot{\theta}(0)}{u}$, so we must have $\dot{\theta}(0)=0$.
Next, for arbitrary $v \in \R^{2n}$, take $\gamma'(t) = (x(0) + tv,
z(0), \theta(0))$; then we have $0 =
\inner{J_{\theta(0)}u}{\dot{x}(0)}$.  But $\theta(0)\ne 0$ by
assumption, so $J_{\theta_0}$ is nonsingular and we must have
$\dot{x}(0)=0$.  Finally, take $\gamma'(t) = (x(0), z(0),
\theta(0)+tu)$; then $\inner{u}{\dot{z}(0)}=0$, so $\dot{z}(0)=0$.
Thus we have shown that if $\omega(\dot{\gamma}(0), \dot{\gamma}'(0))
= 0$ for all $\gamma'$, we must have $\dot{\gamma}(0)=0$, which
completes the proof.
\end{proof}

We now proceed to compute and solve Hamilton's equations of motion for
an H-type group $G$.

The subriemannian Hamiltonian on $T^*G$ is defined by (cf. (\ref{hamiltonian-subriemannian}))
\index{Hamiltonian}
\index{H-type group!Hamiltonian $H$}
\begin{equation}\label{hamiltonian-h-type}
  H(p_g) := \frac{1}{2}\sum_{i=1}^{2n} p_g(X_i(g))^2, \quad p_g \in T^*_g G.
\end{equation}

In terms of the above coordinates, we may compute
\begin{equation*}
  p_g(X_i(g)) = p_g\left(\frac{\partial}{\partial
    x^i} + \frac{1}{2} \sum_j \inner{J_{u_j} x}{e_i}
  \frac{\partial}{\partial z^j}\right) = \xi_i(p_g) + \frac{1}{2}\inner{J_{\eta(p_g)}
    x(g)}{e_i}
\end{equation*}
so that
\begin{equation*}
  H(p_g) = \frac{1}{2}\abs{\xi(p_g) + \frac{1}{2} J_{\eta(p_g)} x(g)}^2.
\end{equation*}

\index{Hamilton's equations of motion}
Recall that a path $\gamma : [0,T] \to T^*Q$ satisfies
Hamilton's equations iff $\dot{\gamma}(t) =
X_H(\gamma(t))$, i.e. $dH_{\gamma(t)} + \omega(\dot{\gamma}(t),
\cdot) = 0$.

In an $H$-type group $G$, we write $\gamma$ in coordinates as
$\gamma(t) = (x(t), z(t), \xi(t), \eta(t)) : [0,T] \to T^*G$, so that
we have
\begin{align*}
  \omega(\dot{\gamma}(t), \cdot) = \sum_i (\dot{\xi}_i(t) dx^i -
  \dot{x}^i(t) d\xi_i) + \sum_j (\dot{\eta}_j(t) dz^j - \dot{z}^j(t) d\eta_j).
\end{align*}
Thus Hamilton's equations of motion read
\index{Hamilton's equations of motion}
\index{H-type group!Hamilton's equations of motion}
\begin{equation}
  \dot{x}^i = \frac{\partial H}{\partial \xi^i}, \quad
  \dot{\xi}_i = -\frac{\partial H}{\partial x^i}, \quad
  \dot{z}_j = \frac{\partial H}{\partial \eta^j},\quad
  \dot{\eta}_j = -\frac{\partial H}{\partial z^j}.
\end{equation}

To compute the derivatives, we note that $\frac{1}{2} \grad_x
\abs{Ax+y}^2 = A^* Ax + A^*y$.  If we write $B_x \eta =
J_\eta x$, then $\inner{B_x \eta}{y} = \inner{\eta}{[x,y]}$,
so $B_x^* = [x,\cdot]$, and $B_x^* B_x = \abs{x}^2 I$.  So for a
path $\gamma(t) = (x(t), z(t), \xi(t), \eta(t)) : [0,T] \to T^*G$,
Hamilton's equations of motion read
\begin{align}
  \dot{x} &= \grad_\xi H = \xi + \frac{1}{2} J_\eta x \label{Hx} \\
  \dot{z} &= \grad_\eta H = \frac{1}{2} \grad_\eta \abs{\xi + \frac{1}{2} B_x \eta}^2 =
  \frac{1}{4} \abs{x}^2 \eta + \frac{1}{2}[x,\xi] \label{Hz} \\
  \dot{\xi} &= -\grad_x H = -\frac{1}{4}\abs{\eta}^2 x + \frac{1}{2}J_\eta \xi \label{Hxi}\\ 
  \dot{\eta} &= -\grad_z H = 0. \label{Heta}
\end{align}

\begin{notation}
  Define the function $\nu : \R \to \R$ by
  \begin{equation} \label{nu-def}
  \nu(\theta) = \frac{2\theta - \sin 2 \theta}{1-\cos 2 \theta} 
  = 
  \frac{\theta}{\sin^2 \theta} - \cot\theta = -\frac{d}{d\theta} [\theta
  \cot \theta]
  \end{equation}
  where the alternate form comes from the double-angle identities.
\end{notation}

\begin{theorem}\label{ham-soln}
  \index{H-type group!Hamilton's equations of motion!solutions}
  $(x(t), z(t))$ is the projection of a solution to Hamilton's
  equations with $x(0)=0$, $z(0)=0$ and $x(1)$, $z(1)$ given if and only if:
  \begin{enumerate}
  \item If $z(1)=0$, we have
    \begin{equation}\label{ham-line}
      x(t) = tx(1),\quad z(t)=0.
    \end{equation}
  \item If $z(1) \ne 0$, we have
\begin{align}
  x(t) &= 
  \frac{1}{\abs{\eta_0}^2}J_{\eta_0}(I-e^{tJ_{\eta_0}})\xi_0 \label{ham-arc-x}
  \\  
  z(t) &=
  \frac{\abs{\xi_0}^2}{2\abs{\eta_0}^3}(\abs{\eta_0}t-\sin(\abs{\eta_0}t))
  \eta_0 \label{ham-arc-z}
\end{align}
where, if $x(1) \ne 0$ we have
\begin{align}
  \eta_0 &= 2\theta \frac{z(1)}{\abs{z(1)}} \\
  \xi_0 &= -\abs{\eta_0}^2 (J_{\eta_0} (e^{J_{\eta_0}}-I))^{-1} x(1)
       \intertext{where $\theta$ is a solution to}
       \nu(\theta) &= \frac{4\abs{z(1)}}{\abs{x(1)}^2} \label{nu-dist}
\end{align}
with $\nu$ as given by (\ref{nu-def}); and if $x(1) = 0$ we have
\begin{align*}
  \eta_0 &= 2 \pi k \frac{z(1)}{\abs{z(1)}} \\
  \abs{\xi_0} &= \sqrt{4 k \pi \abs{z(1)}}
\end{align*}
for some integer $k \ge 1$.
  \end{enumerate}
\end{theorem}

\begin{proof}
  
We solve (\ref{Hx}--\ref{Heta}), assuming $x(0)=0$, $z(0)=0$.  By
(\ref{Heta}) we have $\eta(t) \equiv \eta(0) = \eta_0$.  If $\eta_0 =
0$, we can see by inspection that the solution is 
\begin{equation}\label{eta0-0-soln}
\eta(t) = 0, \quad \xi(t) = \xi_0, \quad x(t)=t\xi_0, \quad z(t)=0,
\end{equation}
namely, a straight line from the origin, whose length is clearly
$\abs{x(1)}$.  This is (\ref{ham-line}), which
we shall see is forced when $z(1)=0$.

Otherwise, assume $\eta_0 \ne 0$.  We may solve (\ref{Hx}) for $\xi$
to see that
\begin{equation} \label{xi-solved}
\xi = \dot{x} - \frac{1}{2} J_{\eta_0} x.
\end{equation}  
Notice that substituting (\ref{xi-solved}) into (\ref{Hz}) shows that
\begin{equation}\label{zdot-bracket}
  \dot{z} = \frac{1}{2}[x, \dot{x}]
\end{equation}
from which an easy computation
verifies that $(x(t),z(t))$ is indeed a horizontal path.  (This is
analogous to the formula (\ref{H1-green}) for $\mathbb{H}_1$.)

Substituting (\ref{xi-solved}) into the right side of (\ref{Hxi}) shows that
\begin{equation*}
  \dot{\xi} = -\frac{1}{4} \abs{\eta_0}^2 x + \frac{1}{2} J_{\eta_0}
  (\dot{x} - \frac{1}{2} J_{\eta_0} x) =
  \frac{1}{2} J_{\eta_0} \dot{x}
\end{equation*}
since $J_{\eta_0}^2 x = -\abs{\eta_0}^2 x$.  Thus 
\begin{equation}\label{Hxi-x}
  \xi = \frac{1}{2} J_{\eta_0} x + \xi_0
\end{equation}
where $\xi_0 = \xi(0)$.  If $\xi_0 = 0$, it is easily seen that we
have the trivial solution $x(t)=0$, $z(t)=0$, $\xi(t)=0$, $\eta(t)=\eta_0$,
so we assume now that $\xi_0 \ne 0$.  (\ref{Hxi-x}) may be substituted back
into (\ref{Hx}) to get
\begin{equation}
  \dot{x} =  J_{\eta_0} x + \xi_0
\end{equation}
so that
\begin{equation}\label{Hxsol}
  x = (J_{\eta_0})^{-1} (e^{tJ_{\eta_0}}-I)\xi_0 =
  -\frac{1}{\abs{\eta_0}^2}J_{\eta_0}(e^{tJ_{\eta_0}}-I)\xi_0.
\end{equation}
Differentiation (or substitution) shows 
\begin{equation}\label{xdot}
\dot{x} = e^{tJ_{\eta_0}} \xi_0.
\end{equation}

Note that
\begin{equation}\label{normx-formula}
  \abs{x}^2 = \frac{1}{\abs{\eta_0}^2}
  \abs{(e^{tJ_{\eta_0}}-I)\xi_0}^2 = \frac{2}{\abs{\eta_0}^2}(1-\cos(\abs{\eta_0}t))\abs{\xi_0}^2.
\end{equation}

It is easy to see from (\ref{Hxsol}) that $x(t)$ lies in the plane
spanned by $\xi_0$ and $J_{\eta_0} \xi_0$, and $x(t)$ sweeps out a
circle centered at $\frac{1}{\abs{\eta_0}^2}J_{\eta_0} \xi_0$ and
passing through the origin.  In particular, the radius of the circle
is ${\abs{\xi_0}}/{\abs{\eta_0}}$.

Now substituting (\ref{Hxsol}) and (\ref{xdot}) into
(\ref{zdot-bracket}), we have
\begin{align*}
  \dot{z} &= -\frac{1}{2 \abs{\eta_0}^2}\left([J_{\eta_0} e^{tJ_{\eta_0}}
    \xi_0, e^{tJ_{\eta_0}} \xi_0] - [J_{\eta_0} \xi_0, e^{tJ_{\eta_0}}
    \xi_0]\right) \\
  &= \frac{1}{2 \abs{\eta_0}^2}\left(\abs{e^{tJ_{\eta_0}}
    \xi_0}^2 \eta_0 + [J_{\eta_0}\xi_0, e^{tJ_{\eta_0}} \xi_0]\right) \\
  &= \frac{\abs{\xi_0}^2}{2 \abs{\eta_0}^2}\left(1 -
  \cos(\abs{\eta_0} t)\right) \eta_0.
\end{align*}

By integration,
\begin{equation}\label{z-formula}
  z = \frac{\abs{\xi_0}^2}{2\abs{\eta_0}^3}(\abs{\eta_0}t-\sin(\abs{\eta_0}t)) \eta_0.
\end{equation}
In particular,
\begin{equation}\label{normz-formula}
  \abs{z} = \frac{\abs{\xi_0}^2}{2\abs{\eta_0}^2}( \abs{\eta_0}t-\sin(\abs{\eta_0}t)).
\end{equation}
We note that inspection of (\ref{normz-formula}) shows that $z(t) \ne
0$ for $t > 0$.  Thus the only solution with $z(1)=0$ is that of
(\ref{ham-line}). 

To make more sense of this, let $r=\abs{\xi_0}/\abs{\eta_0}$ be the
radius of the arc swept out by $x(t)$, and $\phi =  \abs{\eta_0} t$
be the angle subtended by the arc.  Then
\begin{equation*}
  \abs{z} = \frac{1}{2} r^2 \phi - \frac{1}{2} r^2 \sin \phi
\end{equation*}
which is the area of the region between an arc of radius $r$
subtending an angle $\phi$ and the chord which spans it.

We must determine $\xi_0, \eta_0$ in terms of $x(1), z(1)$.  We have
already ruled out the case $z(1)=0$.  If $x(1)=0$, then
(\ref{normx-formula}) shows we must have $\abs{\eta_0} = 2 k \pi$ for
some integer $k \ge 1$.  (\ref{z-formula}, \ref{normz-formula}) then
shows $\eta_0 = 2k \pi z(1) / \abs{z(1)}$, and $\abs{\xi_0} =
\sqrt{4k\pi\abs{z(1)}}$, as desired.  In this case the direction of
$\xi_0$ is not determined and $\xi_0$ may be any vector with the given
length.

On the other hand, if $x(1) \ne 0$, then $\abs{\eta_0}$ is not an
integer multiple of $2\pi$, so we may divide (\ref{normz-formula}) by
(\ref{normx-formula}) to obtain
\begin{equation}\label{eta0-eqn}
  \frac{\abs{z(1)}}{\abs{x(1)}^2} = \frac{\abs{\eta_0} - \sin
    \abs{\eta_0}}{4(1-\cos \abs{\eta_0})} = \frac{1}{4}
  \nu(\theta)
\end{equation}
taking $\theta = \frac{1}{2} \abs{\eta_0}$, where $\nu$ is as in (\ref{nu-def}).
Then by (\ref{normx-formula}) we have
\begin{equation}\label{norm-xi0}
  \abs{\xi_0}^2 = \frac{1}{2} \abs{x(1)}^2 \frac{
    \abs{\eta_0}^2}{1-\cos(\abs{\eta_0})} = \abs{x(1)}^2
  \frac{\theta^2}{\sin^2\theta}.
\end{equation}

Note that once the magnitudes of $\eta_0$, $\xi_0$ are known, their
directions are determined: $\eta_0 = z(1) \abs{\eta_0} / \abs{z(1)}$
by (\ref{z-formula}), while $\xi_0$ can be recovered from (\ref{Hxsol}):
\begin{align*}
  \xi_0 &= -\eta_0^2 (J_{\eta_0} (e^{J_{\eta_0}}-I))^{-1} x(1).
\end{align*}
So $\eta_0, \xi_0$ and hence $x(t), z(t)$ are all determined by a
choice of $\abs{\eta_0}$ satisfying (\ref{eta0-eqn}).  Writing $\theta
= \abs{\eta_0}$ gives (\ref{ham-arc-x}--\ref{ham-arc-z}).

The ``if'' direction of the theorem requires verifying that the given
formulas in fact satisfy Hamilton's equations, which is routine.
\end{proof}


To compute the Carnot-Carath\'eodory distance function for $G$, we must decide which of the
solutions given in Theorem \ref{ham-soln} is the shortest, and compute
its length.  We collect, for future reference, some facts about the
function $\nu$ of (\ref{nu-def}).

\begin{lemma}\label{nu-increase}
  There is a constant $c > 0$ such that $\nu'(\theta) > c$ for all
  $\theta \in [0,\pi)$.
\end{lemma}

\begin{proof}
  By direct computation, $\nu'(\theta) =
  \frac{2(\sin\theta-\theta\cos\theta)}{\sin^3\theta}$.  By Taylor
  expansion of the numerator and denominator we have $\nu'(0) = 2/3 >
  0$.  For all $\theta \in (0,\pi)$ we have $\sin^3\theta > 0$, so it
  suffices to consider $y(\theta) := \sin\theta-\theta\cos\theta$.
  Now $y(0)=0$ and $y'(\theta) = \theta\sin\theta > 0$ for $\theta \in
  (0,\pi)$, so $y(\theta) > 0$ for $\theta \in (0,\pi)$.  Thus
  $\nu'(\theta) > 0$ for $\theta \in [0,\pi)$, and continuity and the
    fact that $\lim_{\theta \uparrow \pi} \nu'(\theta) = +\infty$
    establishes the existence of the constant $c$.
\end{proof}

\begin{corollary}\label{nu-c}
  $\nu(\theta) \ge c\theta$ for all
  $\theta \in [0,\pi)$, where $c$ is the constant from Lemma \ref{nu-increase}.
\end{corollary}

\begin{proof}
  Integrate the inequality in Lemma \ref{nu-increase}.  Note that $\nu(0)=0$.
\end{proof}

For an H-type group, we obtain the following explicit formula for the
distance.  Note that since our notions of horizontal paths, length and
distance are all defined in terms of the left-invariant vector fields
$X_i$, these concepts are all left-invariant.  That is, if $\gamma$ is
a horizontal path in $G$, then for any $k \in G$, $L_k \gamma$ is a
horizontal path with $\ell(L_k \gamma) = \ell(\gamma)$, and therefore
we have $d(g,h) = d(kg, kh)$ for all $g, h, k \in G$.  Thus the
distance function is completely determined by distance from the
identity.  We write this as $d_0(g) = d(0,g)$ for short.

\index{H-type group!Carnot-Carath\'eodory distance $d$}
\index{H-type group!Carnot-Carath\'eodory distance $d$!formula}
\begin{theorem}\label{distance}
  In an H-type group, the Carnot-Carath\'eodory distance from the
  identity $0$ to a point $(x,z)$ is given by 
  \begin{equation}\label{distance-formula}
    d_0(x,z) = d(0, (x,z)) = 
    \begin{cases}
      \abs{x} \frac{\theta}{\sin\theta}, &  z \ne 0, x \ne 0 \\
      \abs{x}, & z=0 \\
      \sqrt{4\pi\abs{z}}, & x=0
    \end{cases}
  \end{equation}
  where $\theta$ is the unique solution in $[0,\pi)$ to $\nu(\theta) =
    \frac{4 \abs{z}}{\abs{x}^2}$.
\end{theorem}

\begin{proof}[Proof of Theorem \ref{distance}]
  \index{H-type group!geodesics!length}
  We compute the lengths of the paths given in Lemma \ref{ham-soln}.
  The $z=0$ case is obvious.  Observe that for a horizontal path
  $\sigma(t) = (x(t), z(t))$, we have $\dot{\sigma}(t) =
  \sum_{i=1}^{2n} \dot{x}^i(t) X_i(\gamma(t))$, so that
  $\norm{\dot{\sigma}(t)} = \abs{\dot{x}(t)}$.  For paths solving
  Hamilton's equations, (\ref{xdot}) shows that $\abs{\dot{x}(t)} =
  \abs{\xi_0}$, so $\ell(\gamma)=\abs{\xi_0}$.  In the case $x=0$, we
  have $\abs{\xi_0} = \sqrt{4k\pi\abs{z(1)}}$, where $k$ may be any
  positive integer; clearly this is minimized by taking $k=1$.

  Now we must handle the case $x \ne 0$, $z \ne 0$.  In this case we
  have $\ell(\gamma) = \abs{\xi_0} = \abs{x}
  \frac{\theta}{\sin\theta}$, by (\ref{norm-xi0}), where $\theta$
  solves (\ref{nu-dist}) (recall $\theta = \frac{1}{2} \abs{\eta_0}$).  The
  function $\nu$ has $\nu(0)=0$, $\nu(\pi)=+\infty$, and by Lemma
  \ref{nu-increase} $\nu$ is strictly increasing on $[0,\pi)$.
    \ignorethis{(]} Thus among the solutions of (\ref{nu-dist}) there
  is exactly one in $[0,\pi)$.  \ignorethis{(]} We show this is the
  solution that minimizes $\left(\frac{\theta}{\sin\theta}\right)^2$
  and hence also minimizes $\ell(\gamma)$.

  For brevity, let $y=\frac{4 \abs{z}}{\abs{x}^2}$.  If $y \in
  [0,\pi/2]$ then $y=\nu(\theta)$ for a unique $\theta \in
  [0,\infty)$.  \ignorethis{(]} This is because $\nu(\theta) >
  \nu(\pi/2) = \pi/2$ for $\theta > \pi/2$.  Since $\theta$ is
  increasing on $[0,\pi)$ \ignorethis{(]} it suffices to show this for
  $\theta > \pi$.  But for such $\theta$ we have
  \begin{align*}
    \nu(\theta) = \frac{\theta - \sin\theta \cos\theta}{\sin^2\theta}
    \ge \frac{\theta - \frac{1}{2}}{\sin^2\theta} \ge \theta -
    \frac{1}{2} > \pi-\frac{1}{2} > \frac{\pi}{2}
  \end{align*}
  since $\sin\theta\cos\theta \le \frac{1}{2}$ for all $\theta$.
  
  Otherwise, suppose $y > \pi/2$.  Let
  \begin{equation*}
    F(\theta) :=
    \frac{\left(\frac{\theta}{\sin\theta}\right)^2}{\nu(\theta)} =
    \frac{\theta^2}{\theta - \sin\theta\cos\theta}
  \end{equation*}
  which is smooth on $(\pi/2,\infty)$ after removing the removable
  singularities.  We will show that if $\pi/2 < \theta_1 < \pi <
  \theta_2$, then $F(\theta_1) < F(\theta_2)$.  Thus if $\theta_1$ is
  the unique solution to $y=\nu(\theta)$ in $(\pi/2,\pi)$ and
  $\theta_2 > \pi$ is another solution, we will have
  \begin{equation*}
    \left(\frac{\theta_1}{\sin\theta_1}\right)^2 = \nu(\theta_1)
    F(\theta_1) = y F(\theta_1) < y F(\theta_2) = \nu(\theta_2)
    F(\theta_2) = \left(\frac{\theta_2}{\sin\theta_2}\right)^2
  \end{equation*}

Toward this end, we compute
  \begin{align*}
    F'(\theta) &= \frac{2\theta(\theta-\sin\theta\cos\theta) -
      \theta^2(1-\cos^2\theta+\sin^2\theta)}{(\theta-\sin\theta\cos\theta)^2}
    \\
    &= \frac{2\theta\cos\theta(\theta\cos\theta-\sin\theta)}{(\theta-\sin\theta\cos\theta)^2}.
  \end{align*}
  For $\theta \in (\pi/2,\pi)$ we have $\cos\theta < 0$, $\sin\theta >
  0$ and thus $F'(\theta) > 0$.  So $F(\theta_1) < F(\pi)$ and it
  suffices to show $F(\pi) = \pi < F(\theta_2)$.  We have $F'(\pi)=2>0$ so
  this is true for $\theta_2$ near $\pi$, and $F(+\infty)=+\infty$ so
  it is also true for large $\theta_2$.  To complete the argument we
  show that it holds at critical points of $F$.  Suppose
  $F'(\theta_c)=0$ where $\theta_c > \pi$; then either $\cos\theta_c=0$ or
  $\theta_c\cos\theta_c-\sin\theta_c=0$.  If the former then
  $F(\theta_c)=\theta_c>\pi$.  If the latter, then $\theta_c =
  \tan\theta_c$, so
  \begin{align*}
    F(\theta_c) = 
    \frac{\theta_c^2}{\theta_c - \sin\theta_c\cos\theta_c} =
    \frac{\theta_c^2}{\theta_c - \tan\theta_c \cos^2\theta_c} =
    \frac{\theta_c^2}{\theta_c(1- \cos^2\theta_c)} \ge \theta_c > \pi
  \end{align*}
which completes the proof.      
\end{proof}

We note that it is apparent from (\ref{distance-formula}) that we have
the scaling property
\begin{equation}\label{distance-dilate}
  d_0(\varphi_\alpha(x,z)) = \alpha d_0(x,z)
\end{equation}
with $\varphi$ as in Definition \ref{dilation-def}.

\begin{notation}
  If $f,h : G \to \R$, we write $f(g) \asymp
h(g)$ to mean there exist finite positive constants $C_1, C_2$ such
that $C_1 h(g) \le f(g) \le C_2 h(g)$ for all $g \in G$, or some
specified subset thereof.
\define{}{$\asymp$}
\end{notation}

\index{H-type group!Carnot-Carath\'eodory distance $d$!asymptotics}
\begin{corollary}\label{distance-estimate}
  $d_0(x,z) \asymp \abs{x} + \abs{z}^{1/2}$.  Equivalently, $d_0(x,z)^2
  \asymp \abs{x}^2 + \abs{z}$.
\end{corollary}

\begin{proof}
  (Based on \cite[Proposition 5.4]{blu-book}.)  We know $d_0(x,z)$ is a continuous function (with respect to the
  manifold topology on $G$) which is positive except at $(0,0)$.
  $d'(x,z) := \abs{x} + \abs{z}^{1/2}$ is another such function, so
  the conclusion obviously holds on the unit sphere of $d'$.  Now
  $d'(\varphi_\alpha(x,z)) = \alpha d'(x,z)$, and inspection of
  (\ref{distance-formula}) shows that the same holds for $d$, so for
  general $(x,z)$ it suffices to apply the previous statement with
  $\alpha = d'(x,z)^{-1}$.
\end{proof}

This can also be verified by direct computation.  By continuity we can
assume $x \ne 0$, $z \ne 0$.  If $\theta$ is the unique solution in
$[0,\pi)$ to $\nu(\theta)=\frac{4\abs{z}}{\abs{x}^2}$, \ignorethis{(]}
we have $d_0(x,z)^2 = \abs{x}^2
\left(\frac{\theta}{\sin\theta}\right)^2$, so if we let
\begin{equation}
  F(\theta) :=
  \frac{\left(\frac{\theta}{\sin\theta}\right)^2}{1+\nu(\theta)} =
  \frac{d_0(x,z)^2}{\abs{x}^2 + 4\abs{z}}
\end{equation}
it will be enough to show there exist $D_1, D_2$ with $0 < D_1 \le
F(\theta) \le D_2$ for all $\theta \in [0,\pi)$.  \ignorethis{(]} $F$
is obviously continuous and positive on $(0,\pi)$.  We can simplify
$F$ as
\begin{equation*}
  F(\theta) = \frac{\theta^2}{\sin^2\theta + \theta - \sin\theta
    \cos\theta}
\end{equation*}
from which it is obvious that $\lim_{\theta \uparrow \pi} F(\theta) =
\pi > 0$, and easy to compute that $\lim_{\theta \downarrow 0}
F(\theta) = 1 > 0$, which is sufficient to establish the corollary.

Results of this form apply to general stratified Lie groups.  A
standard argument, paraphrased from \cite{blu-book}, where many more
details can be found, is as follows.  Once it is known that $d$
generates the Euclidean topology on $G$, then $d_0(x,z)$ is a continuous
function which is positive except at $(0,0)$.  $d'(x,z) := \abs{x} +
\abs{z}^{1/2}$ is another such function, so the conclusion obviously
holds on the unit sphere of $d'$.  Now $d'(\varphi_\alpha(x,z)) =
\alpha d'(x,z)$, and inspection of (\ref{distance-formula}) shows that
the same holds for $d$, so for general $(x,z)$ it suffices to apply
the previous statement with $\alpha = d'(x,z)^{-1}$.

\fixnotme{
\section{Probabilistic interpretation of the heat kernel}\label{prob-sec}

Define horizontal Brownian motion.  Point back to Figure
\ref{bm-heis-fig}.  Show that it's the diffusion
corresponding to the sublaplacian.  Possibly give a probabilistic
proof that the heat kernel is smooth and positive.}

Chapter \ref{subriemannian-chapter}, in part, is adapted from material
awaiting publication as \heatcite.  The
dissertation author was the sole author of this paper.

\chapter{Heat Kernel Estimates}\label{heat-chapter}

\section{Statement of results}\label{heat-kernel-statement-sec}

The goal of this section is to establish pointwise upper and lower
estimates on the heat kernel $p_t$ on an H-type group $G$, as well as
its gradient $\grad p_t$.  See Corollary \ref{main-corollary} and
Theorems \ref{main-gradient-theorem} and
\ref{vertical-gradient-theorem} below.

\begin{theorem}\label{main-theorem}
  There exists $D_0 > 0$ such that
  \begin{equation}\label{main-theorem-eqn}
  p_1(x,z) \asymp \frac{d_0(x,z)^{2n-m-1}}{1+(\abs{x}d_0(x,z))^{n-\frac{1}{2}}}
    e^{-\frac{1}{4}d_0(x,z)^2}.
  \end{equation}
  for $d_0(x,z) \ge D_0$.
\end{theorem}
\index{H-type group!heat kernel $p_t$!pointwise estimates}
\index{heat kernel bounds!H-type group}

\begin{corollary}\label{main-corollary}
  \begin{equation}\label{main-eqn-t}
      p_t(x,z) \asymp t^{-m-n} \frac{1 + (t^{1/2}d_0(x,z))^{2n-m-1}}{1+(t\abs{x}d_0(x,z))^{n-\frac{1}{2}}}
    e^{-\frac{1}{4t}d_0(x,z)^2}
  \end{equation}
  for $(x,z) \in G$, $t > 0$, with the implicit constants independent
  of $t$ as well as $(x,z)$.
\end{corollary}

\begin{proof}
  Theorem \ref{main-theorem} establishes (\ref{main-eqn-t}) for $t=1$
  and $d_0(x,z) \ge D_0$.  For $d_0(x,z) \le D_0$ the estimate follows
  from continuity and the fact that $p_t(x,z) > 0$ (Theorem \ref{pt-positive}).

  Once (\ref{main-eqn-t}) holds for all $(x,z)$ and $t=1$, item
  \ref{pt-scaling} of Proposition \ref{pt-props} and
  (\ref{distance-dilate}) show that it holds for all $t$, with the
  same constants.
\end{proof}

We also obtain precise upper and lower estimates on the gradient of
the heat kernel.  Again we work only on $d_0(x,z) \ge D_0$, and since
$\grad p_t$ vanishes for $x=0$, it is not as clear how to extend to
all of $G$.  However, the upper bound is sufficient to establish
(\ref{gradient-crude-upper}), which is of interest itself.

\begin{theorem}\label{main-gradient-theorem}
\index{H-type group!heat kernel $p_t$!pointwise gradient estimates}
  There exists $D_0 > 0$ such that
  \begin{equation}\label{main-gradient-eqn}
  \abs{\grad p_1(x,z)} \asymp \abs{x}
  \frac{d_0(x,z)^{2n-m+1}}{1+(\abs{x}d_0(x,z))^{n+\frac{1}{2}}}e^{-\frac{1}{4}d_0(x,z)^2}
  \end{equation}
  for $d_0(x,z) \ge D_0$.  In particular, we can combine this with the
  lower bound of Theorem \ref{main-theorem} to see that there exists
  $C > 0$ such that
  \begin{equation}\label{gradient-crude-upper}
    \abs{\grad p_1(x,z)} \le C(1+d_0(x,z))p_1(x,z).
  \end{equation}
\end{theorem}

By (\ref{gradient-radial}) and differentiation under the integral
sign, we have
\begin{align}
  \grad p_1(x,z) &= -\frac{1}{2}(2\pi)^{-m}(4\pi)^{-n}\abs{x} (q_1(x,z) \hat{x} + q_2(x,z)
  J_{\hat{z}}\hat{x}) \label{grad-q1-q2}\\
\intertext{where}
q_1(x,z) &= -\frac{2}{\abs{x}} \frac{\partial p_1(x,z)}{\partial
  \abs{x}} = \int_{\R^m} e^{{i} \inner{\lambda}{z} - \frac{1}{4} \abs{\lambda}\coth{\abs{\lambda}} \abs{x}^2}
\left(\frac{\abs{\lambda}}{\sinh(\abs{\lambda})}\right)^{n+1}\cosh(\abs{\lambda})\,d\lambda \label{q1-def}
\\
q_2(x,z) &=\frac{\partial p_1(x,z)}{\partial
  \abs{z}}= \int_{\R^m} e^{{i} \inner{\lambda}{z} - \frac{1}{4} \abs{\lambda}\coth{\abs{\lambda}} \abs{x}^2}
\left(\frac{\abs{\lambda}}{\sinh(\abs{\lambda})}\right)^{n} (-i)
\inner{\lambda}{\hat{z}}\,d\lambda \label{q2-def}
\end{align}

As before, (\ref{q1-def}) and (\ref{q2-def}) do not really depend on
$\hat{z}$ but only on $\abs{x},\abs{z}$.

The function $q_2$ is of interest in its own right, because it gives
the norm of the ``vertical gradient'' of $p_1$: $\abs{q_2} =
\abs{\grad_z p_1}$.  The proof of Theorem
\ref{main-gradient-theorem} includes estimates on $q_2$; we record
here the upper bound.
\begin{theorem}\label{vertical-gradient-theorem}
  There exists $D_0 \ge 0$ and a constant $C > 0$ such that
  \begin{equation}
    \abs{\grad_z p_1(x,z)} = \abs{q_2(x,z)} \le C
    \frac{d_0(x,z)^{2n-m-1}}{1+(\abs{x}d_0(x,z))^{n-\frac{1}{2}}} e^{-\frac{1}{4}d_0(x,z)^2}.
  \end{equation}
  whenever $d_0(x,z) \ge D_0$.  In particular, for all $(x,z) \in G$ we have
  \begin{equation}\label{vertical-gradient-crude}
    \abs{\grad_z p_1(x,z)} \le C p(x,z).
  \end{equation}
  \end{theorem}

\begin{remark*}
  Since our proof is based on analysis of the formula
  (\ref{pt-formula-1}), we will henceforth treat (\ref{pt-formula-1}) as
  the definition of a function $p_1$ on $\R^{2n+m}$.  In particular,
  it makes sense for all $n,m$, whether or not an H-type group of the
  corresponding dimension actually exists (which can be ascertained
  via Theorem \ref{dimension-classification}).  The proofs of Theorems
  \ref{main-theorem} and \ref{main-gradient-theorem} do not depend on
  the values of $n$ and $m$, so they likewise remain valid for all
  $n,m$.  The estimates given are in terms of the distance function
  $d$, which likewise should be taken as a function defined by the
  formula (\ref{distance-formula}).  Indeed, the only place where we need
  $p_1$ to be a heat kernel is in the proof of Corollary
  \ref{main-corollary}, where we use the positivity of $p_1$ which
  follows from the general theory (Theorem \ref{pt-positive}).

  In particular, in Section \ref{hadamard-sec} we shall make use of
  estimates on $p_1$ for values of $n,m$ not necessarily corresponding
  to H-type groups.
\end{remark*}

The proofs of these two theorems are broken into two cases, depending
on the relative sizes of $\abs{x}$ and $\abs{z}$.
Section \ref{steepest-descent-sec} deals with the case when $\abs{z}
\lesssim \abs{x}^2$; here we apply a steepest descent type argument to
approximate the desired function by a Gaussian.  Section
\ref{polar-sec} handles the case $\abs{z} \gg \abs{x}^2$ by a
transformation to polar coordinates and a residue computation which
only works for odd $m$.  The result for $m$ even can be deduced from
that for $m$ odd by a Hadamard descent approach, which is contained in
Section \ref{hadamard-sec}.

\section{Previous work}\label{heat-previous-sec}

Estimates of the form (\ref{main-theorem-eqn}) for the classical
Heisenberg group first appeared in \cite{li-jfa}, in the context of a
gradient estimate for the heat semigroup, as did an estimate
equivalent to (\ref{gradient-crude-upper}).  A proof
for Heisenberg groups in all dimensions followed in
\cite{li-heatkernel}.  Our proof is similar in spirit to the latter,
in that it relies on the analysis of an explicit formula for $p_t$
using steepest descent methods and elementary complex analysis.

\index{heat kernel bounds!nilpotent Lie group}
\index{Lie group!heat kernel bounds}
Less precise versions of the inequalities (\ref{main-theorem-eqn}) are
known to hold in more general settings.  Using Harnack
inequalities one can show that for general nilpotent Lie groups,
\begin{equation}
  C_1 R_1(t) e^{-\frac{d^2}{ct}} \le p_t \le C_2(\epsilon) R_2(t) e^{-\frac{d^2}{(4+\epsilon)t}}
\end{equation}
for some constants $c, C_1, C_2$ and functions $R_1, R_2$, where $C_2$
depends on $\epsilon > 0$; see chapter IV of \cite{purplebook}.
\cite{davies-pang}, among others, improves the upper bound to
\begin{equation}
   p_t(g) \le C R_3(g,t) e^{-\frac{d_0(g)^2}{4t}},
\end{equation}
with $R$ a polynomial correction, using logarithmic Sobolev
inequalities, whereas \cite{varopoulos-II} improves the lower bound to
\begin{equation}
  p_t \ge C(\epsilon) R_4(t) e^{-\frac{d^2}{(4-\epsilon)t}}. 
\end{equation}
Similar but slightly weaker estimates were shown for more general
sum-of-squares operators satisfying H\"ormander's condition in
\cite{kusuoka-stroock-III} by means of Malliavin calculus, and in
\cite{jerison-sanchez} by more elementary methods involving
homogeneity and the regular dependence of $p_t$ on $t$.

\index{heat kernel bounds!Heisenberg group}
\index{Heisenberg group!heat kernel bounds}
In the specific case of the classical Heisenberg group, asymptotic
results similar to (\ref{main-theorem-eqn}) had been previously obtained
in \cite{gaveau77} and \cite{hueber-muller}, but without the necessary
uniformity to translate them into pointwise estimates.  A precise
upper bound equivalent to that of (\ref{main-theorem-eqn}) was given in
\cite{bgg} for Heisenberg groups of all dimensions.  All three of
these works, like \cite{li-heatkernel} and the present proof, were
based on an explicit formula for $p_t$ and involved steepest descent
type methods.  In \cite{garofalo-segala}, similar techniques were used
to obtain a Li-Yau-Harnack inequality for the heat equation on
Heisenberg groups.

\section{Steepest descent}\label{steepest-descent-sec}

\index{steepest descent|(}\ignorethis{)}
We first handle the region where $\abs{z} \le B_1 \abs{x}^2$ for some
constant $B_1$.  If $\theta = \theta(x,z)$ is as in Theorem
\ref{distance}, this implies $\nu(\theta) \le 4 B_1$; since $\nu$
increases on $[0,\pi)$ \ignorethis{(]} we have $0 \le \theta \le
\theta_0$ in this region.  Note also that by Corollary
\ref{distance-estimate} we have $d_0(x,z)^2 \le D_2(1+B_1)\abs{x}^2$, as
well as $d_0(x,z)^2 \ge \abs{x}^2$ which is clear from
(\ref{distance-formula}).  Thus for this region the bounds of Theorems
\ref{main-theorem}, \ref{main-gradient-theorem} and
\ref{vertical-gradient-theorem} are implied by the following:
\begin{theorem}\label{region-I-theorem}
  For each constant $B_1 > 0$ there exists $D_0 > 0$ such that
  \begin{align}
    p_1(x,z) &\asymp \frac{1}{\abs{x}^m} e^{-\frac{1}{4} d_0(x,z)^2} \label{I-p-both} \\
    \abs{q_i(x,z)} &\le  \frac{C_2}{\abs{x}^{m}} e^{-\frac{1}{4}
      d_0(x,z)^2}, \quad i=1,2 \label{I-q-upper} \\
    \frac{C_1}{\abs{x}^{m}} e^{-\frac{1}{4} d_0(x,z)^2} &\le
    \max\{\abs{q_1(x,z)}, \abs{q_2(x,z)}\} \label{I-q-lower}
  \end{align}
  for all $x,z$ with $d_0(x,z) \ge D_0$ and $\abs{z} \le B_1 \abs{x}^2$.
\end{theorem}

Our approach here will be a steepest descent argument.  Very
informally, the motivation is as follows: given a function $F(x) =
\int_{\R} e^{-x^2 f(\lambda)} a(\lambda)\,d\lambda$, move the contour
of integration to a new contour $\Gamma$ which passes through a
critical point $\lambda_c$ of $f$, so that $f(\lambda) \approx
f(\lambda_c) + \frac{1}{2} f''(\lambda_c)(\lambda-\lambda_c)^2$.  Then
we have
\begin{equation*}
F(x) \approx
e^{-x^2 f(\lambda_c)} \int_{\Gamma} e^{-x^2 f''(\lambda_c)
  (\lambda-\lambda_c)^2/2} a(\lambda)\,d\lambda.
\end{equation*}
For large $x$ the integrand looks like a Gaussian concentrated near
$\lambda_c$, so $F(x) \asymp e^{-x^2 f(\lambda_c)}
\frac{a(\lambda_c)}{x \sqrt{f''(\lambda_c)}}$.  Our proof essentially
follows this line, in $\R^m$ instead of $\R$, but more care is
required to establish the desired uniformity.

Our first task is to extend the integrand to a meromorphic function on
$\C^m$, so that we may justify moving the contour of integration.

Let $\cdot$ denote the bilinear (not sesquilinear) dot product on
$\C^m$, and for $\lambda \in \C^m$ write $\lambda^2 := \lambda \cdot
\lambda$; this defines an analytic function from $\C^m$ to $\C$, and
$\lambda^2 = \abs{\lambda}^2$ iff $\lambda \in \R^m$.  For $w \in \C$,
let $\sqrt{w}$ denote the branch of the square root function
satisfying $\Im \sqrt{w} \ge 0$ and $\sqrt{w} > 0$ for $w > 0$ (so the
branch cut is the positive real axis).  Thus if $g : \C \to \C$ is an
analytic even function, $\lambda \mapsto g(\sqrt{\lambda^2})$ is
analytic as well, and satisfies $g(\sqrt{\lambda^2}) =
g(\abs{\lambda})$ for $\lambda \in \R^m$.  This holds in particular
for the function $\frac{\sinh w}{w}$, and thus the functions
$\frac{\sqrt{\lambda^2}}{\sinh \sqrt{\lambda^2}}$ and $\sqrt{\lambda^2} \coth
\sqrt{\lambda^2}$ are analytic away from points with $\sqrt{\lambda^2}
= ik\pi$, $k = 1, 2, \dots$.  

Using this notation, we let
\begin{align*}
  a_0(\lambda) &:= \left(\frac{\sqrt{\lambda^2}}{\sinh \sqrt{\lambda^2}}\right)^n \\
  a_1(\lambda) &:=  \cosh \sqrt{\lambda^2} \left(\frac{\sqrt{\lambda^2}}{\sinh \sqrt{\lambda^2}}\right)^{n+1} \\
  a_2(\lambda) &:= -i \left(\frac{\sqrt{\lambda^2}}{\sinh
    \sqrt{\lambda^2}}\right)^n \lambda \cdot \hat{z} \in \C^{2n}.
\end{align*}
As mentioned previously, $\hat{z}$ may be any unit vector in $\R^m$
without affecting the computation.  Therefore we shall treat it as
fixed, while $\abs{z}$ is allowed to vary.

Also, for $\lambda \in \C^m, \theta \in [0,\theta_0], \hat{z} \in
S^{m-1} \subset \R^m$, we define
\begin{equation}\label{fdef}
  f(\lambda, \theta, \hat{z}) := -{i \nu(\theta)} \lambda \cdot
  \hat{z} + \sqrt{\lambda^2} \coth \sqrt{\lambda^2}
\end{equation}
so that
\begin{equation*}
  \frac{\abs{x}^2}{4} f(\lambda, \theta(x,z), \frac{z}{\abs{z}}) =
  -i \lambda \cdot z + \frac{1}{4} \sqrt{\lambda^2} \coth
  \sqrt{\lambda^2} \abs{x}^2.
\end{equation*}
We henceforth write $\theta$ for $\theta(x,z)$.  Thus we now have
\begin{align}
  p_1(x,z) &= (4 \pi)^{-m-n} \int_{\R^m}
  e^{-\frac{\abs{x}^2}{4} f(\lambda, \theta, \hat{z})}
  a_0(\lambda)\,d\lambda \\
  q_i(x,z) &= (4 \pi)^{-m-n} \int_{\R^m}
  e^{-\frac{\abs{x}^2}{4} f(\lambda, \theta, \hat{z})}
  a_i(\lambda)\,d\lambda , && i=1,2
\end{align}
Written thus, the integrands have obvious meromorphic extensions to
$\lambda \in \C^n$, analytic away from the set $\{ \sqrt{\lambda^2}  =
ik\pi,\,k=1,2,\dots\}$.

A simple calculation verifies that $\frac{d}{dw} w \coth w = i
\nu(-iw)$, so we can compute the gradient of $f$ with respect to
$\lambda$ as
\begin{equation}
  \grad_\lambda f(\lambda, \theta, \hat{z}) = -i \nu(\theta) \hat{z} +
  i \nu(-i \sqrt{\lambda^2}) \hat{\lambda}
\end{equation}
which vanishes when $\lambda = i \theta \hat{z}$.  Thus
$i \theta \hat{z}$ is the desired critical point.  We observe that
\begin{equation}
  f(i \theta \hat{z}, \theta, \hat{z}) = \theta \nu(\theta) + i \theta
  \coth(i \theta) = \theta(\nu(\theta) + \cot(\theta)) = \frac{\theta^2}{\sin^2\theta}
\end{equation}
so by (\ref{distance-formula}),
\begin{equation}
  \abs{x}^2 f(i \theta \hat{z}, \theta, \hat{z}) = d_0(x,z)^2.
\end{equation}
Thus we define
\begin{equation}\label{psidef}
  \psi(\lambda, \theta, \hat{z}) := f(\lambda, \theta, \hat{z}) - f(i
  \theta \hat{z}, \theta, \hat{z}) = -i\nu(\theta)\lambda\cdot\hat{z}
  + \sqrt{\lambda^2}\coth\sqrt{\lambda^2} - \frac{\theta^2}{\sin^2\theta}.
\end{equation}
We then have
\begin{equation}
  p_t(x,z) = (4 \pi)^{-m-n}e^{-d_0(x,z)^2/4} \int_{\R^m}
  e^{-\frac{\abs{x}^2}{4} \psi(\lambda, \theta, \hat{z})} a_0(\lambda)\,d\lambda
\end{equation}
and analogous formulas for $q_1, q_2$.  Thus let
\begin{equation}\label{hi-def}
  h_i(x,z) := \int_{\R^m}
  e^{-\frac{\abs{x}^2}{4} \psi(\lambda, \theta, \hat{z})} a_i(\lambda)\,d\lambda.
\end{equation}
It will now suffice to estimate $h_i$.


The first step in the steepest descent method is to move the
``contour'' of integration to pass through $i \theta \hat{z}$.  Some
preliminary computations are in order.

\begin{lemma}\label{sab}
  For $a, b \in \R^m$, we have
  \begin{equation}
\abs{a} - \abs{b} \le \abs{\Re \sqrt{(a+bi)^2}} \le \abs{a}, \quad 0 \le \Im
\sqrt{(a+bi)^2} \le \abs{b}.
  \end{equation}
  Equality holds in the upper bounds if and only if $a$ and $b$ are
  parallel, i.e. $a=rb$ for some $r \in \R$.
\end{lemma}
\begin{proof}
  First note that $(a+bi)^2 = \abs{a}^2 - \abs{b}^2 + 2i a \cdot b$.
  So by the Cauchy-Schwarz inequality,
  \begin{equation}\label{ab2-upper}
    \begin{split}
    \abs{(a+bi)^2}^2 &= (\abs{a}^2 - \abs{b}^2)^2 + (2 a \cdot b)^2
    \\
    &\le (\abs{a}^2 - \abs{b}^2)^2 + 4 \abs{a}^2 \abs{b}^2 \\
    &= (\abs{a}^2+\abs{b}^2)^2
    \end{split}
\end{equation}
  so that $\abs{(a+bi)^2} \le \abs{a}^2+\abs{b}^2$.  Equality holds in
  the Cauchy-Schwarz inequality iff $a$ and $b$ are parallel.  On the
  other hand,
  \begin{equation}\label{ab2-lower}
  \abs{(a+bi)^2} \ge \Re (a+bi)^2 = \abs{a}^2 - \abs{b}^2.    
  \end{equation}
  Now we can write
  \begin{align*}
    \left(\Re \sqrt{(a+bi)^2}\right)^2 &= \frac{1}{4} \left(\sqrt{(a+bi)^2} +
    \conj{\sqrt{(a+bi)^2}}\right)^2 \\
    &= \frac{1}{4} \left((a+bi)^2 + \conj{(a+bi)^2} + 2
    \abs{\sqrt{(a+bi)^2}}^2\right) \\
    &= \frac{1}{2} (\abs{a}^2 - \abs{b}^2 + \abs{(a+bi)^2}).
  \end{align*}
  The upper bound for $\abs{\Re \sqrt{(a+bi)^2}}$ then follows
  from (\ref{ab2-upper}).  The lower bound is trivial if $\abs{a} \le
  \abs{b}$, and otherwise we have by (\ref{ab2-lower}) that
  \begin{equation*}
    \left(\Re \sqrt{(a+bi)^2}\right)^2 \ge \abs{a}^2 - \abs{b}^2 \ge
    (\abs{a} - \abs{b})^2.
  \end{equation*}
  The lower bound for $\Im \sqrt{(a+bi)^2}$ holds by our definition of
  $\sqrt{\cdot}$, and the upper bound is similar to the previous one.
\end{proof}

\begin{lemma}\label{integrand-bounds}
  For each $\theta_0 \in [0,\pi)$ there exists $c(\theta_0) > 0$ such that if $a,b \in \R^n$ with
  $\abs{a} \ge c(\theta_0)$, $\abs{b} \le 2\pi$, we have
  \begin{align}
    \Re \psi(a+ib, \theta, \hat{z}) &\ge \abs{a}/2
    \intertext{and}
    \abs{a_i(a+ib)} &\le 1
  \end{align}
  for all $\theta \in [0,\theta_0]$, $\hat{z}, \hat{x} \in S^{m-1} \subset \R^{m}$.
\end{lemma}
\begin{proof}
  Fix $\theta_0 \in [0,\pi)$.  Note first that
  \begin{equation}
    \Re \psi(a + i b, \theta, \hat{z}) = \nu(\theta) b
    \cdot \hat{z} - \Re f(i \theta \hat{z}, \theta, \hat{z}) +
    \Re\left[\sqrt{(a+bi)^2} \coth \sqrt{(a+bi)^2}\right].
  \end{equation}
  By continuity, $\nu(\theta) b
    \cdot \hat{z} - \Re f(i \theta \hat{z}, \theta, \hat{z})$ is
    bounded below by some constant independent of $a$ for all $\theta
    \in [0,\theta_0]$, $\abs{b} \le 2\pi$.  Thus it suffices to
      show that for sufficiently large $\abs{a}$,
  \begin{equation}\label{foo}
    \Re\left[\sqrt{(a+bi)^2} \coth \sqrt{(a+bi)^2}\right] \ge \frac{2}{3}\abs{a}.
  \end{equation}
  Now for $\alpha \in \R$, $\beta \in [-2\pi,2\pi]$ we have
  \begin{align*}
  \Re((\alpha+i\beta) \coth (\alpha + i\beta)) &= \frac{\alpha \sinh \alpha \cosh \alpha + \beta
      \sin \beta \cos \beta}{\cosh^2 \alpha - \cos^2 \beta} \\
    &\ge \alpha \coth \alpha - \frac{\beta}{\cosh^2 \alpha} \\
    &\ge \alpha \coth \alpha - \frac{2\pi}{\cosh^2 \alpha} \\
  &\ge \frac{3}{4} \abs{\alpha}
  \end{align*}
  for sufficiently large $\abs{\alpha}$.  (Recall that $\lim_{\alpha
    \to \pm \infty} \coth\alpha = \pm 1$.)  Thus, since 
  \begin{align*}
    \abs{\Re
  \sqrt{(a+bi)^2}} &\ge \abs{a} - \abs{b} \ge \abs{a} - 2\pi
    \intertext{and}
    \abs{\Im \sqrt{(a+bi)^2}} &\le 2\pi,
  \end{align*}
 it is clear that (\ref{foo}) holds for sufficiently large $\abs{a}$.

  For the bound on $a_i$, note that the $\sinh$ factor in the
  denominator of each $a_i$ can be estimated by
  \begin{equation*}
    \abs{\sinh(\alpha+i\beta)} = \abs{\frac{e^{\alpha+i\beta} -
        e^{-{\alpha+i\beta}}}{2}} \ge \abs{\frac{\abs{e^{\alpha+i\beta}} -
        \abs{e^{-\alpha+i\beta}}}{2}} = \abs{\sinh \alpha}
  \end{equation*}
  so that $\abs{\sinh \sqrt{(a+bi)^2}} \ge \abs{\sinh \Re
    \sqrt{(a+bi)^2}} \ge \abs{\sinh(\abs{a}-2\pi)}$ for $\abs{a} \ge
  2\pi$.  This grows exponentially with $\abs{a}$, so it certainly dominates
  the polynomial growth of the numerator, and we have $\abs{a_i(a+ib)}
  \le 1$ for large enough $\abs{a}$.
\end{proof}

\begin{lemma}\label{move-contour-Rm}
  Let $F(\lambda) := e^{-\frac{\abs{x}^2}{4} \psi(\lambda, \theta, \hat{z})} a_i(\lambda)$ be the
  integrand in (\ref{hi-def}), where $x,z$ are fixed.  If $\tau \in
  \R^m$ with $\abs{\tau} < \pi$, then
  \begin{equation}
    h_i(x,z) = \int_{\R^m} F(\lambda)\,d\lambda = 
    \int_{\R^m} F(\lambda + i \tau)\,d\lambda.
  \end{equation}
\end{lemma}

\begin{proof}
  Note first that $F$ is analytic at $\lambda + ib$ when $\abs{b} <
  \pi$, by the second inequality in Lemma \ref{sab}.  Also, by Lemma
  \ref{integrand-bounds}, we have 
  \begin{equation}\label{F-bound}
\abs{F(\lambda+ib)} \le e^{-\abs{x}^2 \abs{\lambda}/8}
  \end{equation}
as soon as $\abs{\lambda} > c(\theta)$.

  We view $\int_{\R^m} F(\lambda)\,d\lambda$ as $m$ iterated integrals
  and handle them one at a time.  For $1 \le k \le m$, suppose we
  have shown that
  \begin{equation}
    \int_{\R^m} F(\lambda)\,d\lambda
    = \int_\R \dots \int_\R F(\lambda_1 + i\tau_1, \dots, \lambda_{k-1} +
    i\tau_{k-1}, \lambda_{k}, \dots, \lambda_m)\,d\lambda_1 \dots d\lambda_m.
  \end{equation}
  Continuity of $F$ and (\ref{F-bound}) show that $F$ is
  integrable, so we may apply Fubini's theorem and evaluate the
  $d\lambda_k$ integral first:
  \begin{equation*}
     \int_{\R^m} F(\lambda)\,d\lambda
    = \int_\R \dots \int_\R F(\lambda_1 +i \tau_1, \dots, \lambda_{k-1} +
   i \tau_{k-1}, \lambda_{k}, \dots, \lambda_m)\,d\lambda_k d\lambda_1 \dots d\lambda_m.
  \end{equation*}
  Now
  \begin{align*}
    &\quad \int_\R F(\lambda_1 + i\tau_1, \dots, \lambda_{k-1} +
    i\tau_{k-1}, \lambda_{k}, \dots, \lambda_m)\,d\lambda_k  \\
    &= \lim_{\alpha \to \infty} \int_{-\alpha}^{\alpha} F(\lambda_1 +i \tau_1, \dots, \lambda_{k-1} +
   i \tau_{k-1}, \lambda_{k}, \dots, \lambda_m)\,d\lambda_k.
  \end{align*}
  Since $\lambda_k \mapsto F(\lambda_1 + i\tau_1, \dots, \lambda_{k-1} +
    i\tau_{k-1}, \lambda_{k}, \dots, \lambda_m)$ is analytic for
    $\abs{\Im \lambda_k} \le \tau_k$ (which holds because $\abs{(\tau_1, \dots,
    \tau_k)} \le \abs{\tau} < \pi$), we have
    \begin{align*}
      \int_{-\alpha}^{\alpha} F(\dots, \lambda_{k}, \dots)\,d\lambda_k
      = \int_{-\alpha}^{-\alpha+i\tau_k} F +
      \int_{-\alpha+i\tau_k}^{\alpha+i\tau_k} F + \int_{\alpha+i\tau_k}^{\alpha} F
    \end{align*}
    where the contour integrals are taken along straight (horizontal
    or vertical) lines.  But as soon as $\alpha$ exceeds $c(\theta)$ from
    Lemma \ref{integrand-bounds}, (\ref{F-bound}) gives
\begin{align*}
&\quad \int_{-\alpha}^{-\alpha+i\tau_k} \abs{F(\lambda_1 + i\tau_1, \dots, \lambda_{k-1} +
    i\tau_{k-1}, \lambda_{k}, \dots, \lambda_m)}\,d\lambda_k \\ 
  &\le
  \tau_k e^{-\abs{x}^2 \abs{(\lambda_1, \dots, \lambda_{k-1},
      -\alpha, \lambda_k, \dots, \lambda_m)}/8} \\
  &\le \pi e^{-\abs{x}^2 \abs{\alpha}/8} \to 0 \text{ as } \alpha \to \infty.
\end{align*}
A similar argument shows the same for
$\int_{\alpha+i\tau_k}^{\alpha} F$, so we have
\begin{align*}
  &\quad\int_\R F(\lambda_1 + i\tau_1, \dots, \lambda_{k-1} +
  i\tau_{k-1}, \lambda_{k}, \dots, \lambda_m)\,d\lambda_k \\&=
  \int_{-\infty+i\tau_k}^{\infty+i\tau_k} F(\lambda_1 + i\tau_1, \dots, \lambda_{k-1} +
  i\tau_{k-1}, \lambda_{k}, \dots, \lambda_m)\,d\lambda_k \\
  &= \int_\R F(\lambda_1 + i\tau_1, \dots, \lambda_{k-1} +
  i\tau_{k-1}, \lambda_{k}+i\tau_k, \dots, \lambda_m)\,d\lambda_k.
\end{align*}
Thus applying Fubini's theorem again, we have shown
\begin{equation}
    \int_{\R^m} F(\lambda)\,d\lambda
    = \int_\R \dots \int_\R F(\lambda_1 + i\tau_1, \dots, \lambda_{k-1} +
    i\tau_{k-1}, \lambda_{k}+i\tau_k, \dots, \lambda_m)\,d\lambda_1 \dots d\lambda_m.
\end{equation}
Applying this argument successively for $k=1, 2, \dots, m$ establishes
the lemma.
\end{proof}

For the remainder of this section, we assume that $\abs{z} \le
B_1\abs{x}^2$, so that $\theta \le \theta_0(B_1)$. We next show that
the contribution from $\lambda$ far from the origin is negligible.

\begin{lemma}\label{far-from-origin}
  There exist $r > 0$ and a constant $C > 0$ such that
  \begin{equation}
    \abs{\int_{B(0,r)^C} e^{-\frac{\abs{x}^2}{4}\psi(\lambda + i \theta \hat{z}, x, z)} a_i(\lambda +
      i \theta \hat{z})\,d\lambda} \le \frac{C}{\abs{x}^{2m}}. 
  \end{equation}
\end{lemma}

\begin{proof}
  From Lemma \ref{integrand-bounds}, if $r \ge c(\theta_0)$ we have
  \begin{align*}
    \int_{B(0,r)^C} \abs{e^{-\frac{\abs{x}^2}{4}\psi(\lambda + i \theta \hat{z}, x, z)} a_i(\lambda +
      i \theta \hat{z})}\,d\lambda &\le \int_{B(0,r)^C}
    e^{-\frac{\abs{x}^2}{8}\abs{\lambda}}\,d\lambda \\
    &= \omega_{m-1} \int_r^\infty e^{-\abs{x}^2\rho/8} \rho^{m-1}\,d\rho \\
    &\le \omega_{m-1} \int_0^\infty e^{-\abs{x}^2\rho/8} \rho^{m-1}\,d\rho \\
    &= \omega_{m-1} (b \abs{x}^2)^{-m} \int_0^\infty e^{-\rho}
    \rho^{m-1}\,d\rho \\
    &= \frac{C}{\abs{x}^{2m}}
  \end{align*}
  where $\omega_{m-1}$ is the hypersurface measure of $S^{m-1}$.
\end{proof}

We can now apply a steepest descent argument.  As a similar argument
will be used later in this paper (see Proposition \ref{Fprop-III}), we
encapsulate it in the following lemma.

\index{steepest descent}
 \begin{lemma}\label{abstract-lemma-multi}
   Let $\Sigma \subset \R^k$ for some $k$, $r > 0$, $B(0,r)$ the ball of
   radius $r$ in $\R^m$, and $g : B(0,r) \times \Sigma \to \R$, $k :
   \R^{2n} \times [-r,r] \times \Sigma \to \C$ be measurable.  Define $F
   : \R^{2n} \times \Sigma \to \C$ by
   \begin{equation}
     F(x,\sigma) := \int_{B(0,r)} e^{-\abs{x}^2 g(\lambda, \sigma)} k(x,\lambda,\sigma)\,d\lambda.
   \end{equation}
   Suppose:
   \begin{compactenum}
   \item \label{abstract-b1} There exists a positive constant $b_1$ such that
     $g(\lambda,\sigma) \ge b_1 \abs{\lambda}^2$ for all $\lambda \in B(0,r), \sigma
     \in \Sigma$;  
   \item \label{abstract-k2} $k$ is bounded, i.e. $k_2 := \sup_{x \in \R^{2n}, \lambda \in
     B(0,r), \sigma \in \Sigma} \abs{k(x,\lambda,\sigma)} < \infty$.
     \setcounter{continuehere}{\value{enumi}}
   \end{compactenum}

   Then there exists a positive constant $C_2'$ such that
   \begin{equation}
     \abs{F(x,\sigma)} \le \frac{C_2'}{\abs{x}^m}
   \end{equation}
   for all $x > 0$, $\sigma \in \Sigma$.

   If additionally we have:
   \begin{compactenum}\setcounter{enumi}{\value{continuehere}}
   \item \label{abstract-b2} There exists a positive constant $b_2$ such that
     $g(\lambda,\sigma) \le b_2 \abs{\lambda}^2$ for all $\lambda \in B(0,r), \sigma
     \in \Sigma$;
   \item \label{abstract-k1} There exists a function $\epsilon : \R^+ \to [0,r]$ such that
     $\lim_{\rho \to +\infty} \rho \epsilon(\rho) = +\infty$, and
     \begin{equation}
       k_1 := \inf_{x \in \R^{2n}, \lambda \in B(0,\epsilon(\abs{x})),
         \sigma \in \Sigma} \Re k(x,\lambda,\sigma) > 0.
     \end{equation}
   \end{compactenum}

   Then there exist positive constants $C_1'$ and $x_0$ such that
   for all $\abs{x} \ge x_0$ and $\sigma \in \Sigma$ we have 
   \begin{equation}
     \Re F(x,\sigma) \ge \frac{C_1'}{\abs{x}^m}.
   \end{equation}
 \end{lemma}

 \begin{proof}
   The upper bound is easy, since
   \begin{align*}
     \abs{F(x,\sigma)} &\le k_2 \int_{B(0,r)} e^{-\abs{x}^2 b_1 \abs{\lambda}^2}\,d\lambda
     \\
     &= \frac{k_2}{\abs{x}^m} \int_{B(0,rx)} e^{-b_1 \abs{\lambda}^2}\,d\lambda \\
     &\le \frac{k_2}{\abs{x}^m} \int_{\R^m} e^{-b_1 \abs{\lambda}^2}\,d\lambda
     \\
     &= \frac{k_2 (\pi/b_1)^{m/2}}{\abs{x}^m}.
   \end{align*}
   For the lower bound, let
   \begin{align*}
     F_1(x,\sigma) &:= \int_{B(0,r) \backslash B(0,\epsilon(\abs{x}))}
     e^{-\abs{x}^2 g(\lambda, \sigma)} k(x,\lambda,\sigma)\,d\lambda \\
     F_2(x,\sigma) &:= \int_{B(0, \epsilon(\abs{x}))} e^{-\abs{x}^2 g(\lambda, \sigma)}
     k(x,\lambda,\sigma)\,d\lambda
   \end{align*}
   so that $F = F_1 + F_2$.  Now we have
   \begin{align*}
     \abs{F_1(x,\sigma)} &\le k_2 \int_{B(0,r) \backslash B(0,\epsilon(\abs{x}))} e^{-\abs{x}^2 b_1 \abs{\lambda}^2}\,d\lambda \\
     &\le k_2 \int_{\R^m \backslash B(0,\epsilon(\abs{x}))} e^{-\abs{x}^2 b_1
       \abs{\lambda}^2}\,d\lambda \\
     &\le \frac{k_2}{\abs{x}^m} \int_{\R^m \backslash B(0,\abs{x}\epsilon(\abs{x}))} e^{- b_1
       \abs{\lambda'}^2}\,d\lambda'
   \end{align*}
   where we make the change of variables $\lambda' = \abs{x} \lambda$.
   For $F_2$ we have
   \begin{align*}
     \Re{F_2(x,\sigma)} &\ge k_1 \int_{B(0,\epsilon(\abs{x}))}
       e^{-\abs{x}^2 b_2 \abs{\lambda}^2}\,d\lambda \\
       &= \frac{1}{\abs{x}^m} k_1 \int_{B(0, \abs{x}\epsilon(\abs{x}))} e^{-b_2 \abs{\lambda'}^2}\,d\lambda'.
   \end{align*}
 So we have
 \begin{align*}
   \abs{x}^m \Re F(x,\sigma) &\ge \abs{x}^m \Re F_2(x,\sigma) - \abs{\abs{x}^m F_1(x,\sigma)} \\
   &\ge k_1 \int_{B(0,\abs{x}\epsilon(\abs{x}))} e^{-b_2
   \abs{\lambda'}^2}\,d\lambda' - k_2 \int_{\R^m \backslash B(0,\abs{x}\epsilon(\abs{x}))} e^{- b_1 \abs{\lambda'}^2}\,d\lambda' \\
   &\to  k_1 (\pi/b_2)^{m/2} - 0 > 0
 \end{align*}
 as $\abs{x} \to \infty$.
 So there exists $x_0$ so large that for all $\abs{x} \ge x_0$,
 \begin{equation}\label{F-lower}
   \Re F(x,\sigma) \ge \frac{1}{2} k_1 (\pi/b_2)^{m/2} \frac{1}{\abs{x}^m}
 \end{equation}
 as desired.
 \end{proof}

We need another computation before being able to apply this lemma.

  \begin{lemma}\label{zcothz}
    $\Re \sqrt{(\lambda + i\theta\hat{z})^2} \coth \sqrt{(\lambda +
      i\theta\hat{z})^2} \ge \theta \cot \theta$, with equality iff
    $\lambda = 0$.
  \end{lemma}

\begin{proof}
  We first note that the function $\beta
  \cot \beta$ is strictly decreasing on $[0,\pi)$.  To see this, note
    $\frac{d}{d\beta} \beta \cot \beta = -\nu(\beta)$.  By Corollary
    \ref{nu-c} $\nu(\beta) > 0$.  In particular, $\beta \cot \beta \le 1$.

    Next we observe that for $\alpha \in \R$, $\beta \in [0,\pi)$ we have
  \begin{equation}\label{claim1}
    \Re((\alpha+i\beta) \coth (\alpha + i\beta)) \ge \beta \cot \beta
  \end{equation}
  with equality iff $\alpha = 0$.  This can be seen by verifying that
  \begin{equation}
    \Re((\alpha+i\beta) \coth (\alpha + i\beta)) - \beta \cot \beta =
    \frac{\sinh^2 \alpha (\alpha \coth \alpha - \beta \cot
      \beta)}{\cosh^2 \alpha - \cos^2 \beta}
  \end{equation}
  which is a product of positive terms when $\alpha \ne 0$, since
  $\alpha \coth \alpha > 1 \ge \beta \cot \beta$ and $\cosh^2 \alpha > 1
  \ge \cos^2\beta$.
 
  Therefore, we have
    \begin{align}
      \Re \sqrt{(\lambda + i\theta\hat{z}^2} \coth \sqrt{(\lambda +
      i\theta\hat{z})^2}  &\ge \left(\Im \sqrt{(\lambda +
        i\theta\hat{z})^2}\right) \cot \left(\Im \sqrt{(\lambda +
        i\theta\hat{z})^2}\right) \label{claim-ineq1}\\
      &\ge \theta \cot \theta \label{claim-ineq2}
    \end{align}
    because $0 \le \Im \sqrt{(\lambda +
        i\theta\hat{z})^2} \le \theta < \pi$ by Lemma \ref{sab}.

    If equality holds in (\ref{claim-ineq2}), it must be that $\Im
    \sqrt{(\lambda + i\theta\hat{z})^2} = \theta$.  By Lemma \ref{sab}
    $\lambda$ and $\hat{z}$ are parallel, so $\sqrt{(\lambda +
      i\theta\hat{z})^2} = \pm\abs{\lambda} + i \theta$.  If equality also
    holds in (\ref{claim-ineq1}), we have
    \begin{equation*}
            \Re (\pm\abs{\lambda} + i\theta) \coth(\pm\abs{\lambda} +
            i\theta) = \theta \cot \theta
    \end{equation*}
    so by (\ref{claim1}) it must be that $\abs{\lambda}=0$.  This
    proves the claim.
\end{proof}

\begin{lemma}\label{psi-quadratic}
  Given $r > 0$, there exist constants $b_1, b_2, b_3 > 0$ depending only on
  $r$ and $\theta_0$ such that
  \begin{align}
    b_1 \abs{\lambda}^2 \le \Re \psi(\lambda + i \theta \hat{z}, \theta,
    \hat{z}) &\le b_2 \abs{\lambda}^2 \label{re-psi-quadratic} \\
\intertext{and}
    \abs{\Im \psi(\lambda + i \theta \hat{z}, \theta,
    \hat{z})} &\le b_3 \abs{\lambda}^3 \label{im-psi-cubic}
  \end{align}
  for all $\lambda \in B(0,r) \subset \R^m, \theta \in [0,\theta_0], \hat{z} \in
  S^{m-1} \subset \R^m$.
\end{lemma}

\begin{proof}
  Note first that $\psi(\lambda+i\theta\hat{z}, \theta, \hat{z})$ is
  smooth for $\theta \in [0,\theta_0]$ since $\Im
  \sqrt{(\lambda+i\theta\hat{z})} \le \theta \le \theta_0 < \pi$, so
  that we are avoiding the singularities of $w \coth w$.

  We have $\psi(i\theta\hat{z}, \theta, \hat{z}) = 0$ and
  $\grad_{\lambda} \psi(i\theta\hat{z}, \theta, \hat{z}) = 0$.  We now
  show the Hessian $H(i \theta \hat{z})$ of $\psi$ at $i\theta\hat{z}$
  is real and uniformly positive definite.
  \index{Hessian}

  By direct computation, we can find
  \begin{equation}
    \frac{\partial^2}{\partial \lambda_i \partial \lambda_j}
    \psi(\lambda, \theta, \hat{z}) = \nu'(-i \sqrt{\lambda^2})
    \frac{\lambda_i \lambda_j}{\lambda^2} + i \frac{\nu(-i\sqrt{\lambda^2})}{\sqrt{\lambda^2}}
    \left({\delta_{ij}} - \frac{\lambda_i
      \lambda_j}{\lambda^2}\right)
  \end{equation}
  so that for $u \in \R^m$,
  \begin{equation}
    H(\lambda) u \cdot u = \nu'(-i \sqrt{\lambda^2})
    \frac{(\lambda \cdot u)^2}{\lambda^2} + i \frac{\nu(-i\sqrt{\lambda^2})}{\sqrt{\lambda^2}}
    \left(\abs{u}^2 - \frac{(\lambda \cdot u)^2}{\lambda^2}\right)
  \end{equation}
  and in particular
  \begin{align*}
    H(i \theta \hat{z}) u \cdot u &= \nu'(\theta)
    (\hat{z} \cdot u)^2 + \frac{\nu(\theta)}{\theta}
    \left(\abs{u}^2 - \hat{z} \cdot u)^2\right) \\
    &= \abs{u}^2 \left(s \nu'(\theta)
    + \frac{\nu(\theta)}{\theta}(1-s)\right)
  \end{align*}
  where $s := \left(\frac{\hat{z} \cdot u}{\abs{u}}\right)^2$, so $0
  \le s \le 1$.  Note this is a real number whenever $u \in \R^m$.
  Thus we have $H(i\theta\hat{z}) u \cdot u$ written as a convex
  combination of two real functions of $\theta$, so
  \begin{equation}
    H(i\theta\hat{z}) u \cdot u \ge \abs{u}^2 
    \min\{\frac{\nu(\theta)}{\theta}, \nu'(\theta)\} \ge c \abs{u}^2 
  \end{equation}
  where $c$ is the lesser of the two constants provided by Lemma
  \ref{nu-increase} and Corollary \ref{nu-c} respectively.  This is
  valid for $\theta > 0$ and hence by continuity also for $\theta =
  0$.

  By Taylor's theorem, this shows that (\ref{re-psi-quadratic}) and
  (\ref{im-psi-cubic}) hold for small $\lambda$.  The upper bounds
  thus automatically hold for all $\lambda \in B(0,r)$ by continuity.
  To obtain the lower bound on $\Re \psi$, it will suffice to show
  $\Re \psi > 0$ for all $\lambda \ne 0$.  But we have 
    \begin{align*}
      \Re \psi(\lambda + i \theta \hat{z}, \theta, \hat{z}) &= \theta \nu(\theta) - \Re f(i \theta \hat{z}, \theta, \hat{z}) +
      \Re\left[\sqrt{(\lambda+i\theta\hat{z})^2} \coth
        \sqrt{((\lambda+i\theta\hat{z})^2}\right] \\
      &= \theta \nu(\theta) - \frac{\theta^2}{\sin^2\theta} +
      \Re\left[\sqrt{(\lambda+i\theta\hat{z})^2} \coth
        \sqrt{((\lambda+i\theta\hat{z})^2}\right] \\
      &= -\theta \cot \theta + \Re\left[\sqrt{(\lambda+i\theta\hat{z})^2} \coth
        \sqrt{((\lambda+i\theta\hat{z})^2}\right] \\ &\ge 0
    \end{align*}
    by Lemma \ref{zcothz}, with equality iff $\lambda = 0$.
\end{proof}

The proof of Theorem \ref{region-I-theorem} can now be completed.

\begin{proof}[Proof of Theorem \ref{region-I-theorem}]
  
We establish (\ref{I-p-both}) first.  We can apply Lemma
\ref{abstract-lemma-multi} with $\Sigma := [0, \theta_0] \times
S^{m-1}$, $\sigma = (\theta, \hat{z})$, $r$ the value from Lemma
\ref{far-from-origin}, and
\begin{align*}
  g(\lambda, (\theta, \hat{z}))
&:= \frac{1}{4} \Re\psi(\lambda + i \theta \hat{z}, \theta, \hat{z}) \\
  k(x, \lambda, (\theta, \hat{z})) &:= e^{i \frac{\abs{x}^2}{4}
    \Im\psi(\lambda + i \theta \hat{z}, \theta, \hat{z})} a_0(\lambda+i\theta\hat{z}).
\end{align*}
The necessary bounds on $g$ come from (\ref{re-psi-quadratic}).  For
an upper bound on $k$, we have $\abs{k(x, \lambda, (\theta, \hat{z}))}
= \abs{a_0(\lambda+i\theta\hat{z})}$, which is bounded by the fact
that $(\lambda, \theta, \hat{z})$ ranges over the bounded region
$B(0,r) \times [0,\theta_0] \times S^{m-1}$ which avoids the
singularities of $a_0$.

Now for the lower bound on $k$.  By direct
computation, we have $a_0(i \theta \hat{z}) =
\left(\frac{\theta}{\sin\theta}\right)^n \ge 1$; by continuity there
exists $\delta$ such that $\Re e^{is} a_0(\lambda +i \theta \hat{z}) \ge \frac{1}{2}$
for all $\abs{\lambda} \le \delta$ and $\abs{s} \le \delta$, where $s
\in \R$.  If $\abs{\lambda} \le \abs{x}^{-2/3} \delta / b_3$, where
$b_3$ is as in (\ref{im-psi-cubic}), we will have $\abs{x}^2 \abs{\Im
  \psi(\lambda + i \theta \hat{z})} \le \delta$.  Thus set
$\epsilon(x) := \min\{\delta, \abs{x}^{-2/3} \delta / b_3\}$, so that
$\Re k(x, \lambda, (\theta, \hat{z}) \ge \frac{1}{2}$ for all $\abs{\lambda}
\le \epsilon(x)$ and all $(\theta, \hat{z}) \in \Sigma$, and
$\lim_{\rho \to \infty} \rho \epsilon(\rho) = \lim_{\rho \to \infty}
\rho^{1/3} \delta / b_3 = +\infty$.

Thus Lemma \ref{abstract-lemma-multi} applies, and so combining it
with Lemmas \ref{move-contour-Rm} and \ref{far-from-origin} we have
that there exist positive constants $C, C_1', C_2', x_0$ such that
\begin{equation}
  \left(\frac{C_1'}{\abs{x}^m} - \frac{C}{\abs{x}^{2m}}\right)
  e^{-\frac{1}{4}d_0(x,z)^2} \le p_t(x,z) \le \left(\frac{C_2'}{\abs{x}^m} + \frac{C}{\abs{x}^{2m}}\right)
  e^{-\frac{1}{4}d_0(x,z)^2}.
\end{equation}
whenever $\abs{x} \ge x_0$.  We can choose $x_0$ larger if necessary
so that $\abs{x}^{-m} \gg \abs{x}^{-2m}$.  Then taking $D_0 = x_0$
will establish (\ref{I-p-both}).

For $q_i$, the upper bound is similar; $\abs{a_i}$ is bounded above
just like $\abs{a_0}$, establishing (\ref{I-q-upper}).

For (\ref{I-q-lower}), we cannot necessarily bound both $\abs{q_i}$
below simultaneously, but it suffices to take them one at a time.  For
$0 \le \theta(x,z) \le \frac{\pi}{4}$, we have $a_1(i \theta \hat{z})
= \cos \theta \left(\frac{\theta}{\sin\theta}\right)^{n+1} \ge
\frac{1}{\sqrt{2}}$, so by the above logic we obtain the desired lower
bound on $\abs{q_1}$ for such $\theta$.  If $\frac{\pi}{4} \le \theta
\le \theta_0$, we estimate $q_2$ in the same way, since we have $a_2(i
\theta \hat{z}) = \left(\frac{\theta}{\sin\theta}\right)^n \theta \ge
\frac{\pi}{4}.$
\end{proof}
\ignorethis{(}
\index{steepest descent|)}

\section{Polar coordinates}\label{polar-sec}

\index{polar coordinates|(}
In this section, we obtain estimates for $p_1(x,z)$ and $\abs{\grad
  p_1(x,z)}$ when $\abs{z} \ge B_1 \abs{x}^2$, where $B_1$ is
sufficiently large.  This means that $\theta(x,z) \ge \theta_0$ for
some $\theta_0$ near $\pi$.  Note that by Corollary
\ref{distance-estimate}, we have  $d_0(x,z) \asymp \sqrt{\abs{z}}$ in
this region.

We first consider $p_1$ and show the following.

\begin{theorem}\label{region-II-III-theorem}
  For $m$ odd, there exist constants $B_1, D_0$ such that
  \begin{equation}
    p_1(x,z) \asymp
    \frac{\abs{z}^{n-\frac{m+1}{2}}}{1+(\abs{x}\sqrt{\abs{z}})^{n-\frac{1}{2}}}
    e^{-\frac{1}{4}d_0(x,z)^2}
  \end{equation}
  or, equivalently,
  \begin{equation}
  p_1(x,z) \asymp \frac{d_0(x,z)^{2n-m-1}}{1+(\abs{x}d_0(x,z))^{n-\frac{1}{2}}}
    e^{-\frac{1}{4}d_0(x,z)^2}
  \end{equation}
  for $\abs{z} \ge B_1 \abs{x}^2$ and $\abs{z} \ge D_0$
  (equivalently, $d_0(x,z) \ge D_0$).
\end{theorem}

The effect of the requirement that $\abs{z} \le B_1 \abs{x}^2$ in the
previous section was to ensure that the critical point
$i\theta\hat{z}$ stayed away from the singularities of the integrand.
As $B_1 \to \infty$, the critical point approaches the set of
singularities, and the change of contour we used is no longer
effective; the constants in the estimates of Theorem
\ref{region-I-theorem} blow up.  In the case of the Heisenberg groups,
where the center of $G$ has dimension $m=1$, the singularity is a
single point, and the technique used in \cite{hueber-muller} and
\cite{bgg} is to move the contour past the singularity and concentrate
on the resulting residue term.  For $m > 1$, the singularities form a
large manifold and this technique is not easy to use directly.
However, by making a change to polar coordinates, we can reduce the
integral over $\R^m$ to one over $\R$; this replaces the Fourier
transform by the so-called Hankel transform.  (A similar approach is
used in \cite{randall} in the context of $L^p$ estimates for the
analytic continuation of $p_t$.)  When $m$ is odd, we recover a
formula very similar to that for $m=1$, and the above-mentioned
technique is again applicable.

For the rest of this section, we assume that $m$ is odd.

For $m \ge 3$, we write (\ref{pt-formula-1}) in polar coordinates to obtain
\begin{align}
  p_1(x,z) &= (2\pi)^{-m}(4\pi)^{-n} \int_0^\infty \int_{S^{m-1}} e^{{i}
    \rho \sigma \cdot z}\,d\sigma e^{-\frac{\abs{x}^2}{4} \rho \coth
    \rho} \left(\frac{\rho}{\sinh\rho}\right)^n \rho^{m-1} \,d\rho \\
  &= \frac{(2\pi)^{-m}(4\pi)^{-n}}{2} \int_{-\infty}^\infty \int_{S^{m-1}} e^{{i}
    \rho \sigma \cdot z}\,d\sigma e^{-\frac{\abs{x}^2}{4} \rho \coth
    \rho} \left(\frac{\rho}{\sinh\rho}\right)^n \rho^{m-1} \,d\rho
\end{align}
since the integrand is an even function of $\rho$.  (To see this, make
the change of variables $\sigma \to -\sigma$ in the $d\sigma$
integral.  It is not true when $m$ is even.)

The $d\sigma$ integral can be written in terms of a Bessel function.
\index{Bessel function}
Using spherical coordinates, we can write, for arbitrary $\hat{v} \in
S^{m-1}$ and $w \in \C$,
\begin{align*}
  \int_{S^{m-1}} e^{i w \sigma \cdot \hat{v}}\,d\sigma
  &= \frac{2 \pi^{\frac{m-1}{2}}}{\Gamma\left(\frac{m-1}{2}\right)}
  \int_{0}^{\pi} e^{iw \cos\varphi} \sin^{m-2}\varphi \,d\varphi \\
  &= \frac{4 \pi^{\frac{m-1}{2}}}{\Gamma\left(\frac{m-1}{2}\right)}
  \int_{0}^{\frac{\pi}{2}} \cos(w \cos\varphi) \sin^{m-2}\varphi
  \,d\varphi && \text{(by symmetry)}\\
  &= \frac{(2\pi)^{m/2}}{w^{m/2-1}}J_{m/2-1}(w) && \text{(see page 79
    of \cite{magnus})} \\
  &= \Re \frac{(2\pi)^{m/2}}{w^{m/2-1}} H_{m/2-1}^{(1)}(w)
\end{align*}
where $H_\nu(w)$ is the Hankel function of the first kind, defined by
$H_\nu(w) = J_\nu(w) + i Y_\nu(w)$, with $Y_\nu$ the Bessel function
of the second kind.  
\index{Hankel function}
Page 72 of \cite{magnus}
has a closed-form expression for $H_\nu$ which yields
\begin{equation}\label{hankel-expansion}
  \int_{S^{m-1}} e^{i w \sigma \cdot \hat{v}}\,d\sigma =  2(2\pi)^{\frac{m-1}{2}} \Re \left[\frac{e^{iw}}{w^{m-1}}
  \sum_{k=1}^{\frac{m-1}{2}} c_{m,k} (-iw)^k\right]
\end{equation}
where the coefficients are
\begin{align*}
  c_{m,k} &=  \frac{(m-k-2)!}{ 2^{\frac{m-1}{2}-k}
    \left(\frac{m-1}{2}-k\right)! (k-1)!} > 0.
\end{align*}

The reason for the use of the Hankel function is the appearance of the
$e^{iw}$ factor, which gives us an integrand looking much like that
for $p_t$ when $m=1$.  This will allow us to apply similar techniques
to those which have been used previously for $m=1$.
We have
\begin{align}
  p_1(x,z) &= (\Re) \sum_{k=1}^{(m-1)/2} c_{m,k} \abs{z}^{k-m+1} \int_{-\infty}^\infty
  e^{{i} \rho \abs{z} -\frac{\abs{x}^2}{4} \rho \coth
    \rho} \frac{\rho^{n}}{\sinh^n\rho} (-i\rho)^{k}
  \,d\rho \label{p-expansion} \\
  &= \sum_{k=1}^{(m-1)/2} c_{m,k} \abs{z}^{k-m+1} e^{-\frac{1}{4}d_0(x,z)^2}\int_{-\infty}^\infty
  e^{-\frac{\abs{x}^2}{4} \psi(\rho, \theta)} a_k(\rho)
  \,d\rho \label{p-expansion-phi}
\end{align}
where, using similar notation as before,
\begin{align}
  \psi(\rho, \theta) &:= -i \nu(\theta) \rho + \rho \coth \rho -
  \frac{\theta^2}{\sin^2\theta} \\
  a_k(\rho) &:= \left(\frac{\rho}{\sinh\rho}\right)^n (-i\rho)^k.
\end{align}
The constants and coefficients have all been absorbed into the
$c_{m,k}$; we note that $c_{1,0} > 0$, $c_{m,k} > 0$ for $k \ge 1$,
and $c_{m,0}=0$ for $m > 1$. We dropped the $(\Re)$ because the
imaginary part vanishes, being the integral of an odd function.

For $m=1$, we can write
\begin{equation}\label{p-expansion-m1}
  p_1(x,z) = (4\pi)^{-n} e^{-\frac{1}{4}d_0(x,z)^2}\int_{-\infty}^\infty
  e^{-\frac{\abs{x}^2}{4} \psi(\rho, \theta)} a_0(\rho)
  \,d\rho.
\end{equation}

The integrals appearing in the terms of the sum in
(\ref{p-expansion-phi}), as well as in (\ref{p-expansion-m1}), are all
susceptible to the same estimate, as the following theorem shows.

\begin{theorem}\label{h-estimate}
  Let $S \subset \C$ be the strip $S = \{ 0 \le \Im \rho \le 3
  \pi/2\}$.  Suppose $a(\rho)$ is a function analytic on $S \backslash
  \{i\pi\}$, with a pole of order $n$ at $\rho = i\pi$, $a(i \theta)
  \ge 1$ for $\theta_0 \le \theta < \pi$, and $\int_\R
  \abs{a(\rho+3i\pi/2)}\,d\rho < \infty$.  Let
  \begin{equation}
    h(x,z) := \int_{-\infty}^\infty e^{-\frac{\abs{x}^2}{4} \psi(\rho,
      \theta)} a(\rho) \,d\rho.
  \end{equation}
 There exist $B_1, D_0$ such that
 \begin{equation}
\Re  h(x,z) \asymp \frac{\abs{z}^{n-1}}{1+(\abs{x}\sqrt{\abs{z}})^{n-\frac{1}{2}}}
 \end{equation}
 for all $(x,z)$ with $\abs{z} \ge B_1 \abs{x}^2$ and $\abs{z} \ge D_0$.
\end{theorem}

The proof of Theorem \ref{h-estimate} occupies the rest of this
section.  Theorem \ref{region-II-III-theorem} follows, since Theorem
\ref{h-estimate} applies to each term of (\ref{p-expansion-phi}) (note
each $a_k$ satisfies the hypotheses), and the $k=(m-1)/2$ term will
dominate for large $\abs{z}$.

An argument similar to Lemma \ref{move-contour-Rm}, using the fact
that Lemma \ref{integrand-bounds} applies for $\abs{b} \le 2\pi$, will
allow us to move the contour to the line $\Im \rho = 3\pi/2$,
accounting for the residue at $i\pi$:
\begin{equation}
      h(x,z) := \underbrace{\int_{-\infty}^\infty e^{-\frac{\abs{x}^2}{4} \psi(\rho+3i\pi/2,
      \theta)} a(\rho+3i\pi/2) \,d\rho}_{h_l(x,z)} + \underbrace{\Res(e^{-\frac{\abs{x}^2}{4} \psi(\rho,
      \theta)} a(\rho); \rho = i\pi)}_{h_r(x,z)}.
\end{equation}

The following lemma shows that $h_l(x,z)$, the integral along the
horizontal line, is negligible.

 \begin{lemma}\label{hl-estimate}
   There exists $\theta_0 < \pi$ and a constant $C > 0$ such that for
   all $(x,z)$ with $\theta(x,z) \in [\theta_0,\pi)$ we have
   \begin{equation}
     \abs{h_l(x,z)} \le C e^{-d_0(x,z)^2/8}.
   \end{equation}
 \end{lemma}

 \begin{proof}
   Observe that $\coth(\rho+3 i
 \pi/2)=\tanh \rho$.  So
 \begin{align*}
   \Re\psi(\rho+3i\pi/2,\theta) &= \rho \tanh \rho + \frac{3\pi}{2}
   \nu(\theta) - \frac{\theta^2}{\sin^2\theta}.
 \end{align*}
 Therefore we have
 \begin{align*}
   \abs{h_l(x,z)} &\le e^{-\frac{\abs{x}^2}{4} \left(\frac{3\pi}{2}\nu(\theta)
   - \frac{\theta^2}{\sin^2\theta}\right)} \int_\R
   e^{-\frac{\abs{x}^2}{4} \rho \tanh \rho}
   \abs{a(\rho+3i\pi/2)} \,d\rho \\
   &\le e^{-\frac{\abs{x}^2}{4} \left(\frac{3\pi}{2}\nu(\theta)
   - \frac{\theta^2}{\sin^2\theta}\right)} \int_\R
   \abs{a(\rho+3i\pi/2)} \,d\rho
 \end{align*}
 as $\tau\tanh\tau \ge 0$.  The integral in the last line is a finite
 constant, since $a(\cdot + 3i\pi/2)$ is integrable by assumption.

 However, for $\theta$ sufficiently close to $\pi$, we have
 $\nu(\theta) \ge \frac{1}{\pi} \frac{\theta^2}{\sin^2\theta}$.  (If
 $\beta(\theta) := \nu(\theta)
 \left(\frac{\theta^2}{\sin^2\theta}\right)^{-1}$, we have
 $\lim_{\theta \uparrow \pi} \beta(\theta) = 1/\pi$ and $\lim_{\theta
   \uparrow \pi} \beta'(\theta) = -2/\pi^2 < 0$.  Indeed, $\theta >
 0.51$ suffices.)  Thus for such $\theta$ we have
 \begin{equation}
   \abs{h_l(x,z)} \le C e^{-\frac{\abs{x}^2}{8}
     \frac{\theta^2}{\sin^2\theta}} = C e^{-d_0(x,z)^2/8}. 
 \end{equation}
 \end{proof}

\newcommand{\radius}{r}
\newcommand{\dummy}{\rho}
\newcommand{\oldlambda}{y}
\newcommand{\circlevar}{w}
\newcommand{\indexvar}{k}

 To handle the residue term $h_r$, write it as
 \begin{equation}\label{loop-integral}
   h_r(x,z) = \oint_{\partial B(i\pi, \radius)}
   e^{-\frac{\abs{x}^2}{4}\psi(\dummy,\theta)/4}(\dummy)\,d\dummy.
 \end{equation}
 We can choose any $\radius \in (0,\pi)$ because the integrand is analytic
 on the punctured disk.  To facilitate dealing with the singularity at
 $\theta = \pi$, we adopt the parameters
 \begin{equation}\label{lambda-s-def}
   \begin{split}
     s &:= \pi - \theta(x,z) \\
     \oldlambda &:= \pi \abs{x}^2/s.
   \end{split}
 \end{equation}
 Note that 
\begin{equation}\label{z-comparisons}
  \oldlambda/s \asymp \abs{z},\quad \oldlambda \asymp
  \abs{x}\sqrt{\abs{z}}.
\end{equation}
If we let (compare (\ref{psidef}))
\begin{align}
  \phi(\circlevar, s) &:= \frac{1}{4\pi}s
    \psi(i(\pi-\circlevar),\pi-s) \nonumber \\
    &= \frac{s}{4\pi}
    \left(\nu(\pi-s)(\pi-\circlevar) + (\pi-\circlevar)\cot(\pi-\circlevar)
    - \frac{(\pi-s)^2}{\sin^2 s}\right)\label{phidef} \\
  F(\oldlambda,s) &:= s^{n-1} \oint_{\partial B(0,\radius)}
  e^{-\oldlambda \phi(\circlevar, s)}
  a(i(\pi-\circlevar))(-i)\,d\circlevar \label{Fdef}
\end{align}
we have
\begin{equation}\label{hr-F}
  h_r(x,z) = s^{-(n-1)} F(y,s).
\end{equation}
Note we have made the change of variables $\dummy = i(\pi-\circlevar)$ from
(\ref{loop-integral}) to (\ref{Fdef}).  

Observe that $F$ is analytic in $\oldlambda$ and $s$ for $s \ne k\pi$,
$k \in \Z$, so we shall now consider $\oldlambda$ and $s$ as complex
variables.  The factor of $s^{n-1}$ in $F$ was inserted to clear a
pole of order $n-1$ at $s=0$, whose presence will be apparent later.

 Computing a Laurent series for $\phi$
\index{Laurent series}
 about $(i\pi,\pi)$, which converges for $0 < \abs{s} < \pi$, $0 <
 \abs{\circlevar} < \pi$, we find
 \begin{equation}\label{phi-series}
   \phi(\circlevar, s) = \frac{1}{2} - \frac{\circlevar}{4s} -
   \frac{s}{4\circlevar} - sU(\circlevar,s)
 \end{equation}
 with $U$ analytic for $\abs{s} < \pi$, $\abs{\circlevar} < \pi$.
 Also, by the hypotheses on $a$,
 \begin{equation}\label{a-series}
   a(i(\pi-\circlevar)) = \circlevar^{-n} V(\circlevar)
 \end{equation}
 where $V$ is analytic for $\abs{\circlevar} < \pi/2$ and
 $V(0)> 0$.  Thus we have
 \begin{equation}\label{FUV}
   F(\oldlambda,s) = s^{n-1} \oint_{\partial B(0,\radius)} e^{-\oldlambda\left(\frac{1}{2}- \frac{\circlevar}{4s} -
   \frac{s}{4\circlevar} - sU(\circlevar,s)\right)} \circlevar^{-n} V(\circlevar) (-i)\,d\circlevar
 \end{equation}

 The constant term in the expansion of $\psi$ is slightly inconvenient, so let
 $G(\oldlambda,s) = e^{\oldlambda/2} F(\oldlambda,s)$.  Then:
 \begin{align}
   G(\oldlambda,s) &= s^{n-1} \oint_{\partial B(0,\radius)}
     e^{\oldlambda\left(\frac{\circlevar}{4s} + \frac{s}{4\circlevar} + sU(\circlevar,s)\right)}
     \circlevar^{-n} V(\circlevar) (-i) \,d\circlevar \nonumber \\
   &= s^{n-1} \oint \sum_{\indexvar=0}^{\infty}
   \frac{\oldlambda^\indexvar}{\indexvar!} \left(\frac{\circlevar}{4s} +
   \frac{s}{4\circlevar} + sU(\circlevar,s)\right)^\indexvar \circlevar^{-n} V(\circlevar)\,d\circlevar (-i) \label{fubini-used} \\
   &= s^{n-1} \sum_{\indexvar=0}^{\infty}
   \frac{\oldlambda^\indexvar}{\indexvar!} \oint  \left(\frac{\circlevar}{4s} +
   \frac{s}{4\circlevar} + sU(\circlevar,s)\right)^\indexvar \circlevar^{-n} V(\circlevar) (-i) \,d\circlevar \nonumber\\ 
   &=: \sum_{\indexvar=0} \frac{\oldlambda^\indexvar g_\indexvar(s)}{\indexvar!} \label{Gsum}
 \end{align}
 where we let 
 \begin{equation}\label{g_m-def}
   g_\indexvar(s) := s^{n-1} \oint  \left(\frac{\circlevar}{4s} +
   \frac{s}{4\circlevar} + sU(\circlevar,s)\right)^\indexvar \circlevar^{-n} V(\circlevar) (-i) \,d\circlevar.
 \end{equation}
 The interchange of sum and integral in (\ref{fubini-used}) is
 justified by Fubini's theorem, since for fixed $s$ $U(s,\cdot)$ and
 $V$ are bounded on $B(0,\radius)$, and thus
 \begin{align*}
   &\quad \sum_{\indexvar=0}^{\infty} \oint_{B(0, \radius)}
   \abs{\frac{\oldlambda^\indexvar}{\indexvar!} \left(\frac{\circlevar}{4s} +
     \frac{s}{4\circlevar} + sU(\circlevar,s)\right)^\indexvar \left(\frac{\pi}{\circlevar} +
     V(\circlevar)\right)^n}\,d\circlevar \\
    &\le \sum_{\indexvar=0}^{\infty} \frac{\abs{\oldlambda}^\indexvar}{\indexvar!} 2 \pi \radius \left(\frac{\radius}{4\abs{s}} +
    \frac{\abs{s}}{4\radius} + \abs{s} \sup_{\abs{\circlevar}=\radius} \abs{U(\circlevar,s)}\right)^\indexvar \left(\frac{\pi}{\radius} +
    \sup_{\abs{\circlevar}=\radius} \abs{V(\circlevar)} \right)^n \\
    &= 2 \pi \radius  \left(\frac{\pi}{\radius} +
    \sup_{\abs{\circlevar}=\radius} \abs{V(\circlevar)}\right)^n \exp\left(\abs{\oldlambda} \left(\frac{\radius}{4\abs{s}} +
    \frac{\abs{s}}{4\radius} + \abs{s}\sup_{\abs{\circlevar}=\radius} \abs{U(\circlevar,s)}\right)\right) < \infty.
 \end{align*}

 We now examine more carefully the terms $g_\indexvar$ in (\ref{Gsum}--\ref{g_m-def}).
 \begin{lemma}\label{g_m}
   If $g_\indexvar$ is defined by (\ref{g_m-def}), then:
   \begin{enumerate}
     \item \label{g_m-analytic} $g_\indexvar$ is analytic for $\abs{s} \le s_0$;
     \item \label{g_m-size} There exists $C = C(s_0) \ge 0$ independent of $\indexvar$
     such that $\abs{g_\indexvar(s)} \le C^\indexvar$ for each $\indexvar$ and all $\abs{s} \le
     s_0$;
     \item \label{g_m-zero-low} For $\indexvar \le n-1$, $g_\indexvar(s) = s^{n-1-\indexvar}
     h_\indexvar(s)$, where $h_\indexvar$ is analytic for $\abs{s} \le s_0$.  In
     particular, $g_\indexvar(0)=0$ for $\indexvar < n-1$.
     \item \label{g_m-zero-high} For $\indexvar \ge n-1$, $g_\indexvar(0) > 0$ when
       $\indexvar+n$ is odd, and $g_\indexvar(0)=0$ when $\indexvar+n$ is even.
   \end{enumerate}
 \end{lemma}

 \begin{proof}
   By the multinomial theorem,
   \begin{align}
     g_\indexvar(s) &= \sum_{a+b+c=\indexvar} \binom{\indexvar}{a,b,c} s^{n-1} \oint_{\partial B(0,\radius)}
     \left(\frac{\circlevar}{4s}\right)^a 
     \left(\frac{s}{4\circlevar}\right)^b 
     \left(sU(\circlevar,s)\right)^c \circlevar^{-n} V(\circlevar) (-i)\,d\circlevar
     \\
     &= \sum_{a+b+c=\indexvar} \binom{\indexvar}{a,b,c} 4^{-(a+b)} \oint_{\partial B(0,\radius)}
     \circlevar^{a-b-n} s^{-(a-b-n)-1} \left(sU(\circlevar,s)\right)^c V(\circlevar) (-i) \,d\circlevar
     \\
     &= \sum_{\substack{a+b+c=\indexvar\\a-b-n \le -1}} \binom{\indexvar}{a,b,c} 4^{-(a+b)} \oint_{\partial B(0,\radius)}
     \circlevar^{a-b-n} s^{-(a-b-n)-1} \left(sU(\circlevar,s)\right)^c V(\circlevar) (-i) \,d\circlevar \label{abc}
   \end{align}
   since for terms with $a-b-n \ge 0$, the integrand is analytic in
   $\circlevar$ and the integral vanishes.  Now the integrand of each term of
   (\ref{abc}) is clearly analytic in $s$, hence so is $g_\indexvar$ itself,
   establishing item \ref{g_m-analytic}.  

   For item \ref{g_m-size}, let $U_0 := \sup_{\abs{\circlevar}=\radius, \abs{s}
     \le s_0} \abs{U(\circlevar,s)}$, and $V_0 := \sup_{\abs{\circlevar}=\radius}
   \abs{V(\circlevar)}$.  Then for $\abs{s} \le s_0$,
   \begin{align*}
     \abs{g_\indexvar(s)} &\le \sum_{\substack{a+b+c=\indexvar\\a-b-n \le -1}} \binom{\indexvar}{a,b,c}
     4^{-(a+b)} (2 \pi \radius) 
     \radius^{a-b-n} s_0^{-(a-b-n)-1} \left(s_0 U_0\right)^c V_0
     \\
     &\le 2 \pi \radius V_0 \frac{s_0^{n-1}}{\radius^n} \sum_{\substack{a+b+c=\indexvar\\a-b-n \le -1}}
     \binom{\indexvar}{a,b,c} \left( \frac{\radius}{4 s_0}\right)^a
     \left(\frac{s_0}{4 \radius}\right)^b \left(s_0 U_0\right)^c
     \\
     &\le 2 \pi \radius V_0 \frac{s_0^{n-1}}{\radius^n} \sum_{a+b+c=\indexvar}
     \binom{\indexvar}{a,b,c} \left( \frac{\radius}{4 s_0}\right)^a
     \left(\frac{s_0}{4 \radius}\right)^b \left(s_0 U_0\right)^c
     \\
     &\le 2 \pi \radius V_0 \frac{s_0^{n-1}}{\radius^n} \left(\frac{\radius}{4
       s_0} + \frac{s_0}{4 \radius} + s_0 U_0\right)^\indexvar
   \end{align*}
   so that a constant $C$ can be chosen with $g_\indexvar(s) \le C^\indexvar$,
   establishing item \ref{g_m-size}.

   For item \ref{g_m-zero-low}, suppose $\indexvar \le n-1$ and let $h_\indexvar(s) =
   s^{\indexvar-n+1} g_\indexvar(s)$, so that
   \begin{align*}
     h_\indexvar(s) = \sum_{\substack{a+b+c=\indexvar\\a-b-n \le -1}} \binom{\indexvar}{a,b,c} 4^{-(a+b)} \oint_{\partial B(0,\radius)}
     \circlevar^{a-b-n} s^{-(a-b-\indexvar)} \left(sU(\circlevar,s)\right)^c V(\circlevar) (-i) \,d\circlevar.
   \end{align*}
   But $a-b-\indexvar \le a-\indexvar \le 0$ since $a \le \indexvar$ by definition, so only
   positive powers of $s$ appear, and $h_\indexvar$ is analytic in $s$.

   For item \ref{g_m-zero-high}, we see that when $s=0$, each term of
   (\ref{abc}) will vanish unless $c=0$ and $a-b-n=-1$, i.e. $a+b=\indexvar$
   and $a-b=n-1$.  If $\indexvar$ and $n$ have the same parity, this happens
   for no term, so $g_\indexvar(0)=0$.  If $\indexvar$ and $n$ have opposite parity,
   this forces $a=(\indexvar+n-1)/2$, $b=(\indexvar-n+1)/2$, both of which are
   nonnegative integers.  In this case
   \begin{align*}
     g_\indexvar(s) &= \binom{\indexvar}{(\indexvar+n-1)/2}
     4^{-\indexvar} \oint \circlevar^{-1} V(\circlevar)
     (-i)\,d\circlevar \\
     &= \binom{\indexvar}{(\indexvar+n-1)/2}
     4^{-\indexvar} 2 \pi V(0) > 0
   \end{align*}
   since $V(0)>0$.
\end{proof}

 From this we derive corresponding properties of the function $F$.

 \begin{corollary}\label{Fstructure}
   Let $F(\oldlambda,s)$ be defined as in (\ref{Fdef}).  Then for all $s_0
   < \pi$:
   \begin{enumerate}
     \item $F$ is analytic for all $\oldlambda$ and all $0 \le s \le s_0$.
     \item We may write
       \begin{equation}
         F(\oldlambda,s) = e^{-\oldlambda/2} \left[ \sum_{\indexvar=0}^{n-1} \frac{\oldlambda^{\indexvar}
           s^{n-1-\indexvar}}{\indexvar!} h_\indexvar(s) + \oldlambda^n H(\oldlambda,s) \right]
       \end{equation}
       with $h_\indexvar, H$ analytic for all $\oldlambda$ and all $0 \le s \le s_0$.
       Furthermore, $h_{n-1}(0) > 0$
     \item $F(\oldlambda, 0) > 0$ for all $\oldlambda > 0$.
   \end{enumerate}
 \end{corollary}

 \begin{proof}
 We prove the corresponding facts about $G=e^{\oldlambda/2}F$.  By items
 \ref{g_m-analytic} and \ref{g_m-size} of Lemma \ref{g_m}, we have
 that $G$ is analytic for $\abs{s} \le s_0$ and all $\oldlambda$, since
 the sum in (\ref{Gsum}) is a sum of analytic functions and converges
 uniformly.  By item \ref{g_m-zero-low} we have that
 \begin{equation*}
   G(\oldlambda, s) = \sum_{\indexvar=0}^{n-1} \frac{\oldlambda^\indexvar s^{n-1-\indexvar}}{\indexvar!}
   h_\indexvar(s) + \oldlambda^n \sum_{\indexvar=0}^\infty \frac{\oldlambda^\indexvar}{(n+\indexvar)!} g_{n+\indexvar}(s).
 \end{equation*}
 And by items \ref{g_m-zero-low} and \ref{g_m-zero-high}, $G(\oldlambda,0)
 = \sum_{\indexvar=n-1}^\infty \frac{\oldlambda^\indexvar g_\indexvar(0)}{\indexvar!} > 0$ for all
 $\oldlambda > 0$.
 \end{proof}

 \begin{proposition}\label{Fprop-II}
   For all $\oldlambda_1 > 0$, there exist $\delta > 0$, and $0 <
   C_1' \le C_2' < \infty$ such that
   \begin{equation}\label{Fprop-II-eqn}
     C_1' \oldlambda^{n-1} \le \Re F(\oldlambda,s) \le \abs{F(\oldlambda,s)} \le
     C_2' \oldlambda^{n-1}
   \end{equation}
   for all $0 \le \oldlambda < \oldlambda_1$, $0 \le s < \delta
   \oldlambda$.  (Here we are treating $\oldlambda$ and $s$ as real variables.)
 \end{proposition}

 \begin{proof}
    Let $K$ be a positive constant so large that $\abs{h_\indexvar(s)}
    \le K$ and $\abs{H(\oldlambda,s)} \le K$ for all $0 \le \oldlambda
    < \oldlambda_1$, $0 \le s < \oldlambda_1$, $\indexvar \le n-1$.
    For any $\delta < 1$ and all $s \le \delta \oldlambda < \oldlambda_1$, we
    have
   \begin{align*}
     \Re G(\oldlambda, s) &= \frac{\oldlambda^{n-1}}{(n-1)!} \Re
     h_{n-1}(s) + 
            \sum_{\indexvar=0}^{n-2} \frac{\oldlambda^{\indexvar} s^{n-1-\indexvar}}{\indexvar!}
           \Re {h_\indexvar(s)} + \oldlambda^n \Re H(\oldlambda,s)
           \\ &\ge \frac{\oldlambda^{n-1}}{(n-1)!} \Re h_{n-1}(s) -
           \sum_{\indexvar=0}^{n-2} \frac{\oldlambda^{n-1} \delta^{n-1-\indexvar} K}{\indexvar!}
           - \oldlambda^n K \\
           &= \oldlambda^{n-1} \left[ \frac{\Re h_{n-1}(s)}{(n-1)!} - K \sum_{\indexvar=0}^{n-2} \frac{\delta^{n-1-\indexvar}}{\indexvar!}
           \right] - \oldlambda^n K.
   \end{align*}
   Since $h_{n-1}(0) > 0$, we may now choose $\delta$ so small that the bracketed term is positive
   for all $0 \le s \le \delta \oldlambda_1$.  Then there exists $\oldlambda_0 >
   0$ so small that for all $0 \le \oldlambda \le \oldlambda_0$, we have $\Re
   F(\oldlambda,s) \ge e^{-\oldlambda_0/2} \Re G(\oldlambda,s) \ge C_1'
   \oldlambda^{n-1}$ for some $C_1' > 0$.  On the other hand,
   \begin{align*}
     \abs{F(\oldlambda, s)} &\le \abs{G(\oldlambda,s)} \\
     &\le \sum_{\indexvar=0}^{n-1} \frac{\oldlambda^{\indexvar} s^{n-1-\indexvar}}{\indexvar!}
           \abs{h_\indexvar(s)} + \oldlambda^n \Re H(\oldlambda,s) \\
           &\le \oldlambda^{n-1} \sum_{\indexvar=0}^{n-1} \frac{K \delta^{n-1-\indexvar}}{\indexvar!}
            + \oldlambda^n K.
   \end{align*}
   Again, for small $\oldlambda$ (take $\oldlambda_0$ smaller if necessary),
   we have $\abs{F(\oldlambda,s)} \le C_2' \oldlambda^{n-1}$.

   It remains to handle $\oldlambda_0 \le \oldlambda \le \oldlambda_1$.  But
   this presents no difficulty; as $F(\oldlambda,0) > 0$ for all $\oldlambda
   > 0$, and $F$ is continuous, there exists $\delta$ so small that
   \begin{equation*}
     \inf_{\oldlambda_0 \le \oldlambda \le \oldlambda_1, 0 \le s \le \delta
     \oldlambda_1} \Re F(\oldlambda, s) > 0.
   \end{equation*}
   This completes the proof.
 \end{proof}

\index{steepest descent|(}
 \begin{proposition}\label{Fprop-III}
   There exists $\oldlambda_1 > 0$, $s_0 > 0$ and constants $C_1, C_2
   > 0$ such that
   \begin{equation}\label{Fprop-III-eqn}
     \frac{C_1}{\sqrt{\oldlambda}} \le \Re F(\oldlambda,s) \le
     \abs{F(\oldlambda,s)} \le \frac{C_2}{\sqrt{\oldlambda}}
   \end{equation}
   for all $\oldlambda > \oldlambda_1$, $0 < s < s_0$.
 \end{proposition}

 \begin{proof}
   Here the Gaussian approximation technique of Section
   \ref{steepest-descent-sec} is again applicable.  We will fix the
   contour in (\ref{Fdef}) as a circle of radius $\radius = s$,
   parametrize it, and examine the integrand directly.  Thus let $w =
   se^{i\polar}$ in (\ref{Fdef}) to obtain
 \begin{align}
   F(\oldlambda,s) = s^{n-1} \int_{-\pi}^\pi e^{-\oldlambda \phi(se^{i\polar},s)}
   a(i(\pi-se^{i\polar}))se^{i\polar}\,d\polar.
 \end{align}
 We shall apply Lemma \ref{abstract-lemma-multi}, with $m=1$, $\lambda=\polar$,
 $r=\pi$, $x=\sqrt{\oldlambda}$.  Let 
 \begin{align}
 g(\polar, s) &= \Re \phi(se^{i\polar},s) \\
 k(\sqrt{\oldlambda}, \polar, s) &= e^{-i \sqrt{\oldlambda}^2 \Im
   \phi(se^{i\polar},s)} s^n a(i(\pi-se^{i\polar}))e^{i\polar}
 \end{align}

 Since $\phi(s,s) = 0$ and $\circlevar=s$ is a critical point of
 $\phi(\circlevar,s)$, we have 
 \begin{align}
   \evalat{\frac{\partial^2}{\partial^2\polar} \phi(se^{i\polar},
   s)}{\polar=0} &= \frac{s}{4\pi} \phi''(s,s) (is)^2 = \frac{s^3
     \nu'(\pi-s)}{4\pi}
 \end{align}
which is bounded and positive for all small $s$ (recall $\nu(\pi-s)
\sim s^{-2}$).  Thus there exists $s_0, \epsilon$ small enough and
constants $b_1, b_2$ such that
\begin{equation}\label{g-bounds}
  b_1 \polar^2 \le g(\polar,s) \le b_2 \polar^2
\end{equation}
for $s < s_0$, $\abs{\polar} < \epsilon$.  Also, we have from (\ref{phi-series}) that
\begin{equation}
  \phi(se^{i\polar},s) = \frac{1}{2} - \frac{1}{2}\cos\polar -sU(se^{i\polar},s)
\end{equation}
so that by taking $s_0$ smaller if necessary, we can ensure
$g(\polar,s) > 0$ for all $s < s_0$ and $\epsilon \le \abs{\polar} \le
\pi$.  Thus (\ref{g-bounds}) holds for $s < s_0$ and all $\polar \in
   [-\pi,\pi]$, with possibly different constants $b_1, b_2$.

Boundedness of $k$ follows from the fact that $a$ has a pole of order
$n$ at $i\pi$, so $s^n a(i(\pi-se^{i\polar})) = V(se^{i\polar})$ is
bounded for small $s$.  Finally, since
$\evalat{\frac{\partial^2}{\partial^2\polar} \phi(se^{i\polar},
  s)}{\polar=0} > 0$ and $V(0)>0$, the argument used in the proof of Theorem
\ref{region-I-theorem} shows that the necessary lower bound on $k$
also holds.  Then an application of Lemma \ref{abstract-lemma-multi}
completes the proof.
 \end{proof}
\index{steepest descent|)}

 \begin{proof}[Proof of Theorem \ref{h-estimate}]
   Choose $\oldlambda_1, s_0$ so that Proposition \ref{Fprop-III} holds, and take
   $B_1$ large enough so that $\theta(x,z) \ge \pi-s$ when $\abs{z}
   \ge B_1 \abs{x}^2$.  Use this value of $\oldlambda_1$ and choose a
   $\delta$ such that Proposition \ref{Fprop-II} holds, and take $D_0$
   large enough that $s < \delta \oldlambda$ when $\abs{z} \ge dD0$
   (see (\ref{z-comparisons})).  So for such $(x,z)$, either
   (\ref{Fprop-II-eqn}) or (\ref{Fprop-III-eqn}) holds; which one depends on the value
   of $\oldlambda = \oldlambda (x,z)$.  We can combine them to get
   \begin{equation}\label{F-unified}
     C_1' \frac{\oldlambda^{n-1}}{1+\oldlambda^{n-\frac{1}{2}}} \le \Re
     F(\oldlambda,s) \le \abs{F(\oldlambda,s)} \le C_2'
     \frac{\oldlambda^{n-1}}{1+\oldlambda^{n-\frac{1}{2}}}.
   \end{equation}
   Inserting this into (\ref{hr-F}) and using (\ref{z-comparisons}),
   we have (in more compact notation)
   \begin{equation}
     h_r(x,z) \asymp \left(\frac{\oldlambda}{s}\right)^{n-1}
     \frac{1}{1+\oldlambda^{n-\frac{1}{2}}} \asymp
     \frac{\abs{z}^{n-1}}{1+(\abs{x}\sqrt{\abs{z}})^{n-\frac{1}{2}}}.
   \end{equation}
   By Lemma (\ref{hl-estimate}), $h_l$ is clearly negligible by
   comparison, so Theorem \ref{h-estimate} is proved.
\end{proof}

A similar argument will give us the estimates on $\grad p_1$ and $q_2$
which correspond to Theorems \ref{main-gradient-theorem} and
\ref{vertical-gradient-theorem}.

\begin{theorem}\label{gradient-thm-regionII}
  For $m$ odd, there exist constants $B_1, D_0, C$ such that
  \begin{equation}
    \abs{\grad p_1(x,z)} \asymp \frac{\abs{x}
      d_0(x,z)^{2n-m+1}}{1+(\abs{x}d_0(x,z))^{n+\frac{1}{2}}}e^{-\frac{1}{4}d_0(x,z)^2}
  \end{equation}
and
\begin{equation} \label{q2-upper-regionII}
  \abs{q_2(x,z)} \le C
  \frac{d_0(x,z)^{2n-m-1}}{1+(\abs{x}d_0(x,z))^{n-\frac{1}{2}}}e^{-\frac{1}{4}d_0(x,z)^2}
\end{equation}
  whenever $\abs{z}\ge B_1\abs{x}^2$ and $d_0(x,z) \ge D_0$.
\end{theorem}

\begin{proof}
  Applying (\ref{grad-q1-q2}) to (\ref{p-expansion}), we have
\begin{align*}
  \grad p_1(x,z) &= -\frac{1}{2}(2\pi)^{-m}(4\pi)^{-n}\abs{x} (q_1(x,z) \hat{x} + q_2(x,z)
  J_{\hat{z}}\hat{x}) \label{grad-q1-q2}\\
   \intertext{where}
   q_1(x,z) &= -\frac{2}{\abs{x}} \frac{\partial p_1(x,z)}{\partial
   \abs{x}} \\
   &= - \sum_{k=0}^{(m-1)/2} c_{m,k}  \abs{z}^{k-m+1} \int_{-\infty}^\infty
   e^{{i} \rho \abs{z} -\frac{\abs{x}^2}{4} \rho \coth
     \rho} \left(\frac{\rho}{\sinh\rho}\right)^{n+1} (-\cosh\rho) (-i\rho)^{k}
  \,d\rho \\
 q_2(x,z) &=\frac{\partial p_1(x,z)}{\partial
   \abs{z}} \\
 &= \sum_{k=0}^{(m-1)/2} \left[ c_{m,k} (k-m+1)\abs{z}^{k-m} \int_{-\infty}^\infty
   e^{{i} \rho \abs{z} -\frac{\abs{x}^2}{4} \rho \coth
     \rho} \left(\frac{\rho}{\sinh\rho}\right)^n (-i\rho)^{k}
   \,d\rho\right] \\
   &\quad  - \sum_{k=0}^{(m-1)/2} \left[c_{m,k} \abs{z}^{k-m+1} \int_{-\infty}^\infty
   e^{{i} \rho \abs{z} -\frac{\abs{x}^2}{4} \rho \coth
     \rho} \left(\frac{\rho}{\sinh\rho}\right)^n (-i\rho)^{k+1}
   \,d\rho \right].
\end{align*}

  Each integral can be estimated by Theorem \ref{h-estimate}.  For
  $q_1$, each integral is comparable to
  $e^{-\frac{1}{4}d_0(x,z)^2}\frac{\abs{z}^{n}}{1+(\abs{x}\sqrt{\abs{z}})^{n+\frac{1}{2}}}$,
  and the \mbox{$k=(m-1)/2$} term dominates, so
  \begin{equation}\label{q1-estimate-regionII}
    \abs{q_1(x,z)} \asymp \frac{\abs{z}^{n-(m-1)/2}}{1+(\abs{x}\sqrt{\abs{z}})^{n+\frac{1}{2}}}e^{-\frac{1}{4}d_0(x,z)^2}.
  \end{equation}
  The appearance of the extra minus sign in $q_1$ is to account for
  the fact that $\cosh(i\pi) = -1$, but Theorem \ref{h-estimate}
  requires that $a(\lambda)$ be positive near $\lambda=i\pi$.

  For $q_2$, each integral is comparable to
  $\frac{\abs{z}^{n-1}}{1+(\abs{x}\sqrt{\abs{z}})^{n-\frac{1}{2}}}e^{-\frac{1}{4}d_0(x,z)^2}$,
  and the \mbox{$k=(m-1)/2$} term of the second sum dominates, so
  \begin{equation}\label{q2-estimate-regionII}
    \abs{q_2(x,z)} \asymp
    \frac{\abs{z}^{n-1-(m-1)/2}}{1+(\abs{x}\sqrt{\abs{z}})^{n-\frac{1}{2}}}
    e^{-\frac{1}{4}d_0(x,z)^2}
  \end{equation}
  which in particular implies (\ref{q2-upper-regionII}).  To combine
  (\ref{q1-estimate-regionII}) and (\ref{q2-estimate-regionII}), note
  that for $\abs{x}^2 \abs{z}$ bounded we have
  \begin{equation}
    \abs{q_1(x,z)} \asymp \abs{z}^{n-(m-1)/2}e^{-\frac{1}{4}d_0(x,z)^2}; \quad
    \abs{q_2(x,z)} \asymp \abs{z}^{n-1-(m-1)/2}e^{-\frac{1}{4}d_0(x,z)^2}
  \end{equation}
  so that the $q_1$ term dominates, and
  \begin{equation}
    \abs{\grad p_1(x,z)} \asymp \abs{x} \abs{z}^{n-(m-1)/2}e^{-\frac{1}{4}d_0(x,z)^2}.
  \end{equation}
  For $\abs{x}^2 \abs{z}$ bounded away from $0$ we have
  \begin{equation}
    \begin{split}
    \abs{q_1(x,z)} &\asymp \abs{x}^{-n-\frac{1}{2}} \abs{z}^{\frac{n}{2}
      - \frac{m}{2} + \frac{1}{4}}e^{-\frac{1}{4}d_0(x,z)^2} 
\\  \abs{q_2(x,z)} &\asymp
    \abs{x}^{-n+\frac{1}{2}} \abs{z}^{\frac{n}{2} - \frac{m}{2} -
      \frac{1}{4}}e^{-\frac{1}{4}d_0(x,z)^2} \asymp
    \frac{\abs{x}}{\sqrt{\abs{z}}} q_1(x,z)
    \end{split}
  \end{equation}
  so that the $q_1$ term dominates again
  ($\frac{\abs{x}}{\sqrt{\abs{z}}}$ is bounded by assumption).  Thus
  \begin{equation}
    \abs{\grad p_1(x,z)} \asymp \abs{x}
    \frac{\abs{z}^{n-(m-1)/2}}{1+(\abs{x}\sqrt{\abs{z}})^{n+\frac{1}{2}}}e^{-\frac{1}{4}d_0(x,z)^2}
  \end{equation}
  which is equivalent to the desired estimate.
\end{proof}

\index{polar coordinates|)}

\section{Hadamard descent}\label{hadamard-sec}

\index{Hadamard descent|(}
In this section, we obtain estimates for $p_1(x,z)$ and $\abs{\grad
  p_1(x,z)}$ for $\abs{z} \ge B_1 \abs{x}^2$, $\abs{z} \ge D_0$, in
the case where the center dimension $m$ is even.  The methods of the
previous section are not directly applicable, but we can deduce an
estimate for even values of $m$ by integrating the corresponding estimate for
$m+1$.  As discussed in the remark at the end of Section
\ref{heat-kernel-statement-sec}, this is valid even though there may not exist
an $H$-type group of dimension $2n+m+1$ with center dimension $m+1$,
since the estimates we use are derived from the formula
(\ref{pt-formula-1}) and hold for all values of $n,m$.

We continue to assume that $\abs{z} \ge B_1 \abs{x}^2$ and $\abs{z}
\ge D_0$ for some sufficiently large $B_1, D_0$.  To emphasize the
dependence on the dimension, we write $p^{(n,m)}$ for the function
$p_1$ in (\ref{pt-formula-1}).

In order to estimate $p^{(n,m)}$ for $m$ even, we consider
$p^{(n,m+1)}$.  We can observe that
\begin{equation}\label{dimension-reduce}
  p^{(n,m)}(x,z) = \int_\R p^{(n,m+1)}(x, (z, z_{m+1}))\,dz_{m+1}
\end{equation}
since $\int_\R\int_\R e^{i \lambda_{m+1} z_{m+1}}
f(\lambda_{m+1})\,d\lambda_{m+1}\,dz_{m+1} = 2 \pi f(0)$.  Note that
$\abs{(\lambda,0)}_{\R^{m+1}} = \abs{\lambda}_{\R^m}$.  Now
$p^{(n,m+1)}$ can be estimated by means of Theorem
\ref{region-II-III-theorem}.  Using the fact that $\abs{(z,z_{m+1})}
\ge \abs{z}$, we have that for $m$ even, there exist constants $B_1,
D_0$ such that
  \begin{equation}\label{p-Q-comp}
    p^{(n,m)}(x,z) \asymp  Q^{(2n-m-2,n-\frac{1}{2})}(x,z)
  \end{equation}
  whenever $\abs{z} \ge B_1 \abs{x}^2$ and $\abs{z} \ge D_0$, where
  \begin{equation}\label{Qnm-def}
    Q^{(\alpha,\beta)}(x,z) := \int_\R \frac{d_0(x,(z,z_{m+1}))^{\alpha}}{1+(\abs{x}d_0(x,(z,z_{m+1})))^{\beta}}
    e^{-\frac{1}{4}d_0(x,(z,z_{m+1}))^2}\,dz_{m+1}.
  \end{equation}
  Thus it suffices to estimate the integrated bounds given by
  $Q^{(\alpha,\beta)}$.

  \begin{lemma}\label{Q-est}
    For $\abs{z} \ge B_1 \abs{x}^2$ and $\abs{z} \ge D_0$, we have
    \begin{equation}
      Q^{(\alpha,\beta)}(x,z) \asymp \frac{d_0(x,z)^{\alpha +
          1}}{1+(\abs{x}d_0(x,z))^\beta} e^{-\frac{1}{4}d_0(x,z)^2}.
    \end{equation}
  \end{lemma}

We will require two preliminary computations.  Since $d_0(x,z)$ depends
on $z$ only through $\abs{z}$, we will occasionally treat $d_0$ as a
function on $\R^{2n} \times [0,\infty)$.

\begin{lemma}\label{dist-z-deriv}
  There exist positive constants $c_1, c_2, B_1$ such that for all
  $x \in \R^{2n}, u \in \R$ with $u
  \ge B_1 \abs{x}^2$, we have $0 < c_1 \le \frac{\partial}{\partial u}
  d_0(x,u)^2 \le c_2 < \infty$.
\end{lemma}

\begin{proof}
  Let $\mu(\theta) = \frac{\theta^2}{\sin^2\theta}$, so that $d_0(x,u)^2
  = \abs{x}^2 \mu(\theta)$ with $\theta = \theta(x,z) =
  \nu^{-1}\left(\frac{2u}{\abs{x}^2}\right)$.  Then
  \begin{equation}
    \frac{\partial}{\partial u}
  d_0(x,u)^2 = 2 \frac{\mu'(\theta)}{\nu'(\theta)}.
  \end{equation}
  It is easily verified that $\mu'(\theta) > 0$, $\nu'(\theta) > 0$
  for all $\theta \in (0,\pi)$, and $\frac{\mu'(\theta)}{\nu'(\theta)}
  \to \pi > 0$ as $\theta \to \pi$.
\end{proof}

\begin{lemma}\label{exp-poly-est}
  For any $\alpha \in \R$, there exists $C_\alpha > 0$ such that for
  all $w_0 \ge 1$ we have
  \begin{equation}
    \int_{w_0}^\infty w^\alpha e^{-w}\,dw \le C_\alpha w_0^\alpha e^{-w_0}.
  \end{equation}
\end{lemma}

\begin{proof}
  For $\alpha \le 0$, $w^\alpha$ is decreasing for $w \ge 1$, so
  \begin{equation}
    \int_{w_0}^\infty w^\alpha e^{-w}\,dw \le w_0^\alpha
    \int_{w_0}^\infty e^{-w}\,dw = w_0^\alpha e^{-w_0}
  \end{equation}
  and this holds with $C_\alpha = 1$.  Now, for a nonnegative integer
  $n$, suppose the lemma holds for all $\alpha \le n$.  Then if $n <
  \alpha \le n+1$, we integrate by parts to obtain
  \begin{align*}
    \int_{w_0}^\infty w^\alpha e^{-w}\,dw = w_0^\alpha e^{-w_0} +
    \alpha \int_{w_0}^{\infty} w^{\alpha - 1} e^{-w}\,dw \le (1+\alpha
    C_{\alpha - 1}) w_0^\alpha e^{-w_0}
  \end{align*}
  so that the lemma also holds for all $\alpha \le n+1$.  By induction
  the proof is complete.
\end{proof}

\begin{proof}[Proof of Lemma \ref{Q-est}]
We make the change of variables $u = \abs{(z, z_{m+1})}$ so that
$z_{m+1} = \sqrt{u^2 - \abs{z}^2}$.  By our previous abuse of
notation, we can write $d_0(x,(z,z_{m+1}))=d_0(x,u)$.  Thus
\begin{align*}
  Q^{(\alpha,\beta)}(x,z) &= \int_{\abs{z}}^{\infty} 
  \frac{d_0(x,u)^{\alpha}}{1+(\abs{x}d_0(x,u))^{\beta}}
  e^{-\frac{1}{4}d_0(x,u)^2} \frac{u}{\sqrt{u^2-\abs{z}^2}}\,du \\
    &\asymp \int_{\abs{z}}^{\infty} \frac{1}{\sqrt{u-\abs{z}}}
  \frac{1}{\sqrt{u+\abs{z}}}
  \frac{d_0(x,u)^{\alpha+2}}{1+(\abs{x}d_0(x,u))^{\beta}} e^{-\frac{1}{4}d_0(x,u)^2} \,du.
\end{align*}
We used the fact that $u \asymp d_0(x,u)^2$ where $\abs{z} \ge B_1
\abs{x}^2$, by Corollary \ref{distance-estimate}.

Now, noting that $u \mapsto d_0(x,u)$ is an increasing function, and $w
\mapsto w^{\alpha+2}e^{-\frac{1}{4}w^2}$ is decreasing for large
enough $w$, the lower bound can be obtained by
\begin{align*}
  Q^{(\alpha,\beta)}(x,z) &\ge \int_{\abs{z}}^{\abs{z}+1} \frac{1}{\sqrt{u-\abs{z}}}
  \frac{1}{\sqrt{u+\abs{z}}}
  \frac{d_0(x,u)^{\alpha+2}}{1+(\abs{x}d_0(x,u))^{\beta}}
  e^{-\frac{1}{4}d_0(x,u)^2} \,du  \\
&\ge \left( \int_{\abs{z}}^{\abs{z}+1} \frac{1}{\sqrt{u-\abs{z}}}
  \,du  \right) \frac{1}{\sqrt{2\abs{z}+1}}
  \frac{d_0(x,\abs{z}+1)^{\alpha+2}}{1+(\abs{x}d_0(x,\abs{z}+1))^{\beta}}
  e^{-\frac{1}{4}d_0(x,\abs{z}+1)^2}  \\
  &= 2 \frac{1}{\sqrt{2\abs{z}+1}}
  \frac{d_0(x,\abs{z}+1)^{\alpha+2}}{1+(\abs{x}d_0(x,\abs{z}+1))^{\beta}}
  e^{-\frac{1}{4}d_0(x,\abs{z}+1)^2} \\
  &\ge C \frac{1}{\sqrt{2\abs{z}}}
  \frac{d_0(x,{z})^{\alpha+2}}{1+(\abs{x}d_0(x,{z}))^{\beta}}
  e^{-\frac{1}{4}d_0(x,{z})^2}
\end{align*}
where the last line follows because $u \mapsto d_0(x,u)^2$ is Lipschitz,
as shown by Lemma \ref{dist-z-deriv}, with a constant independent of $x$.

Since $\abs{z} \asymp d_0(x,z)^2$, we have that
\begin{equation}
  Q^{(\alpha,\beta)}(x,z)\ge C'
  \frac{d_0(x,{z})^{\alpha+1}}{1+(\abs{x}d_0(x,{z}))^{\beta}}
  e^{-\frac{1}{4}d_0(x,{z})^2}.
\end{equation}

For an upper bound, we have
\begin{align*}
  Q^{(\alpha,\beta)}(x,z) \le C \left[ \int_{\abs{z}}^{\abs{z}+1} \frac{1}{\sqrt{u-\abs{z}}}
  \frac{1}{\sqrt{u+\abs{z}}}
  \frac{d_0(x,u)^{\alpha+2}}{1+(\abs{x}d_0(x,u))^{\beta}}
  e^{-\frac{1}{4}d_0(x,u)^2} \,du + \int_{\abs{z}+1}^\infty \dots. \right]
\end{align*}
Now
\begin{align*}
  &\quad \int_{\abs{z}}^{\abs{z}+1} \frac{1}{\sqrt{u-\abs{z}}}
  \frac{1}{\sqrt{u+\abs{z}}}
  \frac{d_0(x,u)^{\alpha+2}}{1+(\abs{x}d_0(x,u))^{\beta}}
  e^{-\frac{1}{4}d_0(x,u)^2} \,du \\
&\le \left(\int_{\abs{z}}^{\abs{z}+1} \frac{1}{\sqrt{u-\abs{z}}}
   \,du\right) \frac{1}{\sqrt{2\abs{z}}}
  \frac{d_0(x,z)^{\alpha+2}}{1+(\abs{x}d_0(x,z))^{\beta}}
  e^{-\frac{1}{4}d_0(x,z)^2}  \\
   &= 2 \frac{1}{\sqrt{2\abs{z}}}
  \frac{d_0(x,z)^{\alpha+2}}{1+(\abs{x}d_0(x,z))^{\beta}}
  e^{-\frac{1}{4}d_0(x,z)^2} \\
  &\le C \frac{d_0(x,z)^{\alpha+1}}{1+(\abs{x}d_0(x,z))^{\beta}}
  e^{-\frac{1}{4}d_0(x,z)^2}.
\end{align*}
For the other term, we observe
\begin{align*}
  &\int_{\abs{z}+1}^{\infty} \frac{1}{\sqrt{u-\abs{z}}}
  \frac{1}{\sqrt{u+\abs{z}}}
  \frac{d_0(x,u)^{\alpha+2}}{1+(\abs{x}d_0(x,u))^{\beta}}
  e^{-\frac{1}{4}d_0(x,u)^2} \,du  \\
  &\le \int_{\abs{z}+1}^{\infty}   \frac{1}{\sqrt{u+\abs{z}}}
  \frac{d_0(x,u)^{\alpha+2}}{1+(\abs{x}d_0(x,u))^{\beta}}
  e^{-\frac{1}{4}d_0(x,u)^2} \,du \\
  &\le \int_{\abs{z}}^{\infty}   \frac{1}{\sqrt{2u}}
  \frac{d_0(x,u)^{\alpha+2}}{1+(\abs{x}d_0(x,u))^{\beta}}
  e^{-\frac{1}{4}d_0(x,u)^2} \,du \\
  &\le C \int_{\abs{z}}^{\infty}   \frac{d_0(x,u)^{\alpha+1}}{1+(\abs{x}d_0(x,u))^{\beta}}
  e^{-\frac{1}{4}d_0(x,u)^2} \,du.
\end{align*}
We now make the change of variables $w=\frac{1}{4}d_0(x,u)^2$.  By the
above lemma, $du/dw$ is bounded, so
\begin{align*}
  \int_{\abs{z}}^{\infty}   \frac{d_0(x,u)^{\alpha+1}}{1+(\abs{x}d_0(x,u))^{\beta}}
  e^{-\frac{1}{4}d_0(x,u)^2} \,du \le C \int_{\frac{1}{4}d_0(x,z)^2}^\infty
  \frac{(4w)^{(\alpha+1)/2}}{1+(2\abs{x}\sqrt{w})^\beta} e^{-w}\,dw.
\end{align*}
If $d_0(x,z) \le 1/\abs{x}$, we have
\begin{align*}
  \int_{\frac{1}{4}d_0(x,z)^2}^\infty
  \frac{(4w)^{(\alpha+1)/2}}{1+(2\abs{x}\sqrt{w})^\beta} e^{-w}\,dw
  &\le \int_{\frac{1}{4}d_0(x,z)^2}^\infty
  {(4w)^{(\alpha+1)/2}} e^{-w}\,dw \\
  &\le C d_0(x,z)^{\alpha+1} e^{-\frac{1}{4}d_0(x,z)^2} \\
  &\le 2C \frac{d_0(x,z)^{\alpha+1}}{1+(\abs{x}d_0(x,z))^\beta}e^{-\frac{1}{4}d_0(x,z)^2}
\end{align*}
where we have used Lemma \ref{exp-poly-est}.

On the other hand, when $d_0(x,z) \ge 1/\abs{x}$, we have
\begin{align*}
  \int_{\frac{1}{4}d_0(x,z)^2}^\infty
  \frac{(4w)^{(\alpha+1)/2}}{1+(2\abs{x}\sqrt{w})^\beta} e^{-w}\,dw
  &\le (2\abs{x})^{-\beta} \int_{\frac{1}{4}d_0(x,z)^2}^\infty
  {(4w)^{(\alpha+1-\beta)/2}} e^{-w}\,dw \\
  &\le C \abs{x}^{-\beta} d_0(x,z)^{\alpha+1-\beta} e^{-\frac{1}{4}d_0(x,z)^2} \\
  &\le 2C \frac{d_0(x,z)^{\alpha+1}}{1+(\abs{x}d_0(x,z))^\beta}e^{-\frac{1}{4}d_0(x,z)^2}.
\end{align*}

Combining all this, we have as desired that
\begin{equation}
  Q^{(\alpha,\beta)}(x,z) \asymp \frac{d_0(x,z)^{\alpha+1}}{1+(\abs{x}d_0(x,z))^\beta}e^{-\frac{1}{4}d_0(x,z)^2}.
\end{equation}
\end{proof}

\begin{corollary}
  Theorems  \ref{region-II-III-theorem} and
  \ref{gradient-thm-regionII} also hold for $m$ even.
\end{corollary}

\begin{proof}
The heat kernel estimate of Theorem \ref{region-II-III-theorem} is
immediate, given (\ref{p-Q-comp}) and Lemma \ref{Q-est}.
  
To obtain an estimate on $\grad p_1$, we define $q_1^{(n,m)} :=
-\frac{2}{\abs{x}} \frac{\partial}{\partial \abs{x}} p_1^{(n,m)}(x,z)$,
$q_2^{(n,m)} := \frac{\partial}{\partial \abs{z}} p_1^{(n,m)}(x,z)$,
as in (\ref{grad-q1-q2}).

For $q_1$, we simply differentiate (\ref{dimension-reduce}) to see
\begin{align*}
  q_1^{(n,m)}(x,z) &= \int_\R q_1^{(n,m+1)}(x,(z,z_{m+1}))\,dz_{m+1}
  \\
  &\asymp Q^{(2n-m,n+\frac{1}{2})}(x,z) && \text{by
    (\ref{q1-estimate-regionII})}\\
  &\asymp
  \frac{d_0(x,z)^{2n-m+1}}{1+(\abs{x}d_0(x,z))^{n+\frac{1}{2}}}e^{-\frac{1}{4}d_0(x,z)^2}
  && \text{by Lemma \ref{Q-est}.}
\end{align*}

For $q_2$, we again differentiate (\ref{dimension-reduce}).  Here we
obtain
\begin{align*}
  q_2^{(n,m)}(x,z) &= \int_\R q_2^{(n,m+1)}(x,(z,z_{m+1}))
  \frac{\abs{z}}{\abs{(z,z_{m+1})}}\,dz_{m+1} \\
  &\asymp \abs{z} Q^{(2n-m-4,n-\frac{1}{2})} && \text{by
    (\ref{q2-estimate-regionII})} \\
  &\asymp d_0(x,z)^2
  \frac{d_0(x,z)^{2n-m-3}}{1+(\abs{x}d_0(x,z))^{n-\frac{1}{2}}}e^{-\frac{1}{4}d_0(x,z)^2}
  \\
  &\asymp
  \frac{d_0(x,z)^{2n-m-1}}{1+(\abs{x}d_0(x,z))^{n-\frac{1}{2}}}e^{-\frac{1}{4}d_0(x,z)^2}.
\end{align*}

Repeating the computation from Theorem \ref{gradient-thm-regionII}, we
have the desired estimates on $\abs{\grad p_1}$ and $\abs{q_2}$.
\end{proof}
\index{Hadamard descent|)}

Chapter \ref{heat-chapter}, in large part, is adapted from material
awaiting publication as \heatcite.  The
dissertation author was the sole author of this paper.

\chapter{Gradient Estimates}\label{gradient-chapter}

\section{Statement of results}

\begin{notation}
\index{C@$\funcclass$}
Let $\funcclass$ be the class of $f \in C^1(G)$ for which there exist
constants $M \ge 0$, $a \ge 0$, and $\epsilon \in (0,1)$ such that
\begin{equation*}
  \abs{f(g)} + \abs{\grad f(g)} + \abs{\hat{\grad} f(g)} \le M e^{a d(0,g)^{2-\epsilon}}
\end{equation*}
for all $g \in G$.  By the heat kernel bounds in Theorem
\ref{main-theorem}, $P_t f$ as defined by (\ref{Pt-integral-symmetry}) makes
sense for all $f \in \funcclass$.
\end{notation}

The main theorem of this chapter is the following:
\begin{theorem}\label{main-grad-theorem}
  \index{gradient bounds!H-type group}
  \index{H-type group!gradient bounds}
  \index{H-type group!heat semigroup}
  There exists a finite constant $K$ such that for all $f \in
  \funcclass$,
  \begin{equation}\label{grad-ineq}
    \abs{\grad P_t f} \le K P_t(\abs{\grad f}).
  \end{equation}
\end{theorem}

This theorem can be interpreted as a quantitative statement about the
smoothing properties of $P_t$.  As we shall see in Section
\ref{gradient-consequences-sec}, it has a number of significant consequences.

\section{Previous work}\label{grad-previous-sec}

The inequality (\ref{grad-ineq}) arose in the work of Bakry in the
context of Riemannian manifolds \cite{bakry-riesz-notes-ii},
\cite{bakry-sobolev}.  In this case, (\ref{grad-ineq}) is strongly
related to the geometry of the manifold, and in particular to its
Ricci curvature.  \index{Ricci curvature}

\begin{theorem}[Bakry]\label{bakry-thm}
  \index{gradient bounds!Riemannian manifold}
  If $L$ is the Laplace-Beltrami operator on a Riemannian manifold
  $M$, then
  \begin{equation} \label{grad-riemannian}
    \abs{\grad P_t f} \le e^{kt} P_t(\abs{\grad f}) \quad
    \forall\, f \in C^\infty_c(M)
  \end{equation}
  if and only if
  \begin{equation}\label{ricci-bound}
    \Ric(v, v) \ge -k \abs{v}^2 \quad \forall\, v \in TM.
  \end{equation}
\end{theorem}

\fixnotme{Stuff about Riesz transforms?}

To contrast this theorem with the situation of hypoelliptic operators,
consider the Heisenberg group $\mathbb{H}_1$ as in Section
\ref{H1-sec}.  The operator $L = X^2 + Y^2$ cannot be the
Laplace-Beltrami operator of any Riemannian metric, since $L$ is not
elliptic.  One might try to approximate $L$ by Laplace-Beltrami
operators.  For example, if for $\epsilon > 0$ we impose a metric
$\inner{\cdot}{\cdot}_\epsilon$ on $\mathbb{H}_1$ such that
$\inner{X}{Y} = \inner{X}{Z} = \inner{Y}{Z} = 0$,
$\inner{X}{X}=\inner{Y}{Y}=1$, and $\inner{Z}{Z}=1/\epsilon$, the
corresponding Laplace-Beltrami operator is $L_\epsilon := X^2 + Y^2 +
\epsilon Z^2$, which approximates $L$ as $\epsilon \to 0$.  However,
the Ricci curvature tensor corresponding to
$\inner{\cdot}{\cdot}_\epsilon$ has $\Ric(X, X) = -\frac{1}{4
  \epsilon} \to -\infty$ as $\epsilon \to 0$, so the constant $k$ in
Theorem \ref{bakry-thm} blows up, and the gradient bound in
(\ref{grad-riemannian}) becomes useless.  Therefore, Theorem
\ref{bakry-thm} does not seem to be directly useful in the
hypoelliptic Lie group setting.

\index{gradient bounds!behavior of constant}
We remark another distinction: in (\ref{grad-riemannian}), the
``constant'' $e^{kt}$ on the right side of the inequality depends on
$t$ and in particular tends to $1$ as $t \to 0$.  ($e^{kt}$ is not
known, or even believed, to be the best constant.)  However, in
(\ref{grad-ineq}) in the case of H-type groups, the constant $K$ is
independent of $t$.  Indeed, it can be shown that the \emph{best}
constant is also independent of $t$, and is strictly greater than
$1$.  In particular, the best constant does not tend to $1$ as $t \to
0$.  This is another indication of the significant differences between
the hypoelliptic and Riemannian settings.

\index{gradient bounds!Lp type@$L^p$ type}
\index{Lie group!gradient bounds}
Progress in the hypoelliptic case was made by Driver and Melcher in
\cite{driver-melcher}, which obtained the following $L^p$-type
estimate in the case of the Heisenberg group $\mathbb{H}_1$:
\begin{equation}\label{intro-Lp-ineq}
  \abs{\grad P_t f}^p \le K_p P_t(\abs{\grad f}^p).
\end{equation}
Their argument proceeded probabilistically via methods of Malliavin
calculus
\index{Malliavin calculus}
 and did not depend on heat kernel estimates, but they also
showed that their argument could not produce (\ref{grad-ineq}), which
is the corresponding estimate with $p=1$.  \cite{tai-thesis} extended
(\ref{intro-Lp-ineq}) to the case of a general Lie group, at the cost
of replacing the constant $K_p$ with a function $K_p(t)$ which in
general was shown only to be finite for all $t$.  However, for
nilpotent Lie groups, the constant $K_p(t)$ was shown to be bounded
independent of $t$.

  \index{gradient bounds!H-type group}
  \index{H-type group!gradient bounds}
As mentioned in Section \ref{H1-sec}, 
the first extension of (\ref{grad-ineq}) itself to a hypoelliptic
setting was due to H.-Q. Li in \cite{li-jfa}.  Like the argument in
this dissertation, the proof relies on pointwise upper and lower
estimates for the heat kernel, and a pointwise upper estimate for its
gradient, shown in the case of Heisenberg-Weyl groups in
\cite{li-heatkernel}.  \cite{bbbc-jfa} contains two alternate (and
much simpler) proofs of (\ref{grad-ineq}) for the classical Heisenberg
group, also depending on the pointwise heat kernel estimates from
\cite{li-heatkernel}.  The proof of Theorem \ref{main-grad-theorem} in
the case of H-type groups, which occupies the following section,
follows rather closely the approach taken by the first proof in
\cite{bbbc-jfa}.

\section{Proof of gradient estimate}

Following an argument found in \cite{driver-melcher}, by
left-invariance of $P_t$ and $\grad$, we see that in order to
establish (\ref{grad-ineq}) it suffices to show that it holds at the
identity, i.e. to show
\begin{equation}\label{grad-ineq-identity}
\abs{(\grad P_t f)(0)} \le K P_t(\abs{\grad f})(0).
\end{equation}
It also suffices to assume $t=1$.  This can be seen by taking $t=1$ in
 (\ref{grad-ineq-identity}) and replacing $f$ by $f \circ \varphi_{s^{1/2}}$.

Therefore, in order to prove Theorem \ref{main-grad-theorem}, it will
suffice to show $\abs{(\grad P_1 f)(0)} \le K P_1(\abs{\grad f})(0)$.
We may replace $\grad$ by $\hat{\grad}$ on the left side, since $\grad
= \hat{\grad}$ at $0$.  Since $[X_i, \hat{X}_j] =
0$, we expect that $\hat{\grad}$ should commute with $P_t$, which we
now verify.

\begin{proposition}\label{commute}
  For $f \in \funcclass$, $\hat{\grad} P_t f(0)=(P_t \hat{\grad}f )(0)$.
\end{proposition}
\begin{proof}
  By (\ref{Xi-dds}) and (\ref{Pt-integral-symmetry}) we have
  \begin{align*}
    \hat{X}_i P_t f(0) &= \diffat{s}{0} P_t f(s
    e_i, 0) \\
    &= \diffat{s}{0} \int_G f((s e_i, 0)
    \groupop k) p_t(k)\,d\haar(k).
  \end{align*}
  We now differentiate under the integral sign, which can be justified
  because
  \begin{align*}
    \abs{\frac{d}{ds} f((s e_i, 0) \groupop k)} &=
    \abs{\diffat{\sigma}{0} f(((s+\sigma)e_i, 0)
    \groupop k)} \\
    &= \abs{\diffat{\sigma}{0} f((\sigma e_i, 0)
    \groupop (s e_i, 0) \groupop k)} \\
    &= \abs{\hat{X_i} f((s e_i, 0) \groupop k)} \\
    &\le M e^{a d(0, (s e_i, 0) \groupop k)^{2-\epsilon}}.
  \end{align*}
  But 
  \begin{align*}
    d(0, (s e_i, 0) \groupop k) &= d((s e_i, 0)^{-1}, k) = d((-s e_i,
    0), k) \\
    &\le d(0, (-s e_i, 0)) + d(0, k) = \abs{s} + d(0,k).
\end{align*}
  Thus for all
  $s \in [-1,1]$ we have
  \begin{equation*}
    \abs{\frac{d}{ds} f((s e_i, 0) \groupop k)} \le M e^{a (1 +
      d(0,k))^{2-\epsilon}} \le M' e^{a' d(0,k)^{2-\epsilon}}
  \end{equation*}
  for some $M', a'$, and therefore by the heat kernel bounds
  of Theorem \ref{main-theorem} we have
  \begin{equation*}
    \int_G \sup_{s \in [-1,1]} \abs{\frac{d}{ds} f((s e_i, 0) \groupop
      k)} p_t(k)\,d\haar(k) < \infty
  \end{equation*}
  which justifies differentiating under the integral sign.  Thus
  \begin{align*}
    \hat{X}_i P_t f(0) &= \int_G \diffat{s}{0}  f((s e_i, 0)
    \groupop k) p_t(k)\,d\haar(k) \\
    &= \int_G \hat{X}_i f(k)
    p_t(k)\,d\haar(k) \\
    &= P_t \hat{X}_i f(0).
  \end{align*}
  This completes the proof.
\end{proof}

Thus Theorem
\ref{main-grad-theorem} reduces to showing
\begin{equation}\label{grad-ineq-hat-semigroup}
  \abs{(P_1 \hat{\grad} f)(0)} \le K P_1(\abs{\grad f})(0)
\end{equation}
or in other words
\begin{equation}
  \abs{\int_G (\hat{\grad} f) p_1\,d\haar} \le K \int_G \abs{\grad f} p_1\,d\haar
\end{equation}
for which it suffices to show
\begin{equation}\label{integral-to-show}
 \abs{\int_G ((\grad-\hat{\grad}) f) p_1\,d\haar} \le K \int_G \abs{\grad f} p_1\,d\haar.
\end{equation}

A similar argument can be used to verify the following integration by
parts formula.
\begin{proposition}
  If $f \in \funcclass$, then
  \begin{equation}\label{byparts}
    \begin{split}
      \int_G (\grad f) p_1\,d\haar &= - \int_G (\grad p_1) f\,d\haar \\
      \int_G (\hat{\grad} f) p_1\,d\haar &= - \int_G (\hat{\grad} p_1) f\,d\haar.
    \end{split}
  \end{equation}
\end{proposition}
\begin{proof}
  As in Theorem \ref{Pt-semigroup}, take $\psi_n \in C^\infty_c(G)$ to
  be a sequence of compactly supported smooth functions such that
  $\psi_n \to 1$ and $\grad \psi_n \to 0$ boundedly. By item
  \ref{Xi-parts} of Proposition \ref{Xi-integration}, we have
  \begin{align*}
    \int_G (\grad p_1) (f \psi_n)\,d\haar &= -\int_G p_1 \grad(f
    \psi_n)\,d\haar  \\
    &= -\int_G p_1 f \grad \psi_n \,d\haar - \int_G p_1 \psi_n \grad f\,d\haar.
  \end{align*}
  Since $p_1 f, p_1 \grad f, (\grad p_1) f \in L^1(G)$ by definition
  of $\funcclass$ and the heat kernel bounds, two applications of the
  dominated convergence theorem complete the proof.
\end{proof}

\index{H-type group!geodesic coordinates}
We now introduce an alternate coordinate system on $G$, similar but
not exactly analogous to the so-called ``polar coordinate'' system
used in \cite{bbbc-jfa}.  As shown in Section \ref{subriemannian-sec},
there is a unique (up to reparametrization) shortest horizontal path
from the identity $0$ to each point $(x,z) \in G$ with $x,z$ nonzero;
it has as its projection onto $\R^{2n} \times 0$ an arc of a circle
lying in the plane spanned by $x$ and $J_z x$, with the origin as one
endpoint, and $x$ as the other.  The region in this plane bounded by
the arc and the straight line from $0$ to $x$ has area equal to
$\abs{z}$.  The projection onto $0 \times \R^m$ is a straight line
from $0$ to $z$.

Our new coordinate system will identify a point $(x,z)$ with the point
$u \in \R^{2n}$ which is the center of the arc, and a vector $\eta \in
\R^m$ which is parallel to $z$ and whose magnitude equals the angle
subtended by the arc.  The change of coordinates $(u, \eta) \mapsto
(x,z)$ will be denoted by
\begin{align*}
  \Phi &: \{(u,\eta) \in \R^{2n+m} : 0 < \abs{\eta} < 2\pi\}  \to
  \{(x,z) \in G : x \ne 0, z \ne 0\} 
\intertext{where}
  \Phi(u,\eta) &:= \left( \left(I - e^{J_{\eta}}\right) u,
  \frac{\abs{u}^2}{2}
  \left(1-\frac{\sin\abs{\eta}}{\abs{\eta}}\right)\eta\right) \\
  &= \left( (1-\cos\abs{\eta})u + \frac{\sin\abs{\eta}}{\abs{\eta}}
  J_\eta u,     \frac{\abs{u}^2}{2}
  \left(1-\frac{\sin\abs{\eta}}{\abs{\eta}}\right)\eta\right)
\end{align*}
by Proposition \ref{Jz-props}, items \ref{Jz-skew} and \ref{Jz-square}.
$\Phi$ has the property that for each $(u, \eta)$, the path $s \mapsto
\Phi(u,s\eta)$ traces the shortest horizontal path between any two of
its points, and has constant speed $\abs{u}\abs{\eta}$.  In
particular,
  \begin{equation}\label{dPhi}
    d(0, \Phi(u,\eta)) = \abs{u}\abs{\eta}.
    \end{equation}
Also, for any $f \in C^1(G)$,
\begin{equation}\label{geodesic-grad}
  \abs{\frac{d}{ds} f(\Phi(u, s\eta))} \le \abs{u} \abs{\eta}
  \abs{\grad f(\Phi(u, s\eta))}.
\end{equation}

  Note that if $(x,z) = \Phi(u,\eta)$, we have 
  \begin{align*}
    \abs{x}^2 &= \abs{u}^2 (2-2\cos\abs{\eta}) \\ 
    \abs{z} &= \frac{\abs{u}^2}{2} (\abs{\eta}-\sin\abs{\eta}).
\end{align*} 

To compare this with the ``polar coordinates'' $(u,s)$ used in
\cite{bbbc-jfa}, take $u=u$ and $s = \abs{u} \eta$.

In $(u, \eta)$ coordinates, the heat kernel estimate
(\ref{main-eqn-t}) reads
  \begin{align}
    p_1(\Phi(u,\eta)) &\asymp
    \frac{1+(\abs{u}\abs{\eta})^{2n-m-1}}{1+(\abs{u}^2 \abs{\eta}
      \sqrt{2-2\cos\abs{\eta}})^{n-\frac{1}{2}}}
    e^{-\frac{1}{4}(\abs{u}\abs{\eta})^2} \label{p1-u-eta-estimate-cos} \\
    &\asymp
    \frac{1+(\abs{u}\abs{\eta})^{2n-m-1}}{1+\left(\abs{u}^2 \abs{\eta}^2
      (2\pi-\abs{\eta}) \right)^{n-\frac{1}{2}}}
    e^{-\frac{1}{4}(\abs{u}\abs{\eta})^2} \label{p1-u-eta-estimate}
  \end{align}
  since $1-\cos \theta \asymp \theta^2(2\pi-\theta)^2$ for $\theta \in [0,2\pi]$.
We will often abuse notation and write $p_1(u,\eta)$ for
$p_1(\Phi(u,\eta))$, when no confusion will result.

We now begin the proof of Theorem \ref{main-grad-theorem}, which
occupies the rest of this section.

\index{Jacobian determinant}
We begin by computing the Jacobian determinant of the change of
coordinates $\Phi$, so that we can use $(u,\eta)$ coordinates in
explicit computations.

\begin{lemma}
  Let $A(u,\eta)$ denote the Jacobian determinant of $\Phi$, so
  that $d\haar = A(u, \eta)\,du\,d\eta$.  Then
  \begin{equation}
    A(u,\eta) = \abs{u}^{2m} \left(\frac{1}{2}-\frac{\sin\abs{\eta}}{2\abs{\eta}}\right)^{m-1}
      (2-2\cos\abs{\eta})^{n-1} \left(2 - 2\cos\abs{\eta} - \abs{\eta}\sin\abs{\eta}\right).
  \end{equation}
\end{lemma}
Note that $A(u,\eta)$ depends on $u, \eta$ only through their
absolute values $\abs{u}, \abs{\eta}$.  By an abuse of notation we may
occasionally use $A$ with $u$ or $\eta$ replaced by scalars, so that
$A(r, \rho)$ means $A(r \hat{u}, \rho \hat{\eta})$ for arbitrary unit
vectors $\hat{u}, \hat{\eta}$. 

For the Heisenberg group with $n=m=1$, this reduces to
\begin{equation*}
      A(u,\eta) = \abs{u}^{2} \left(2 - 2\cos\abs{\eta} - \abs{\eta}\sin\abs{\eta}\right).
\end{equation*}
The analogous expression appearing in \cite{bbbc-jfa} is slightly
incorrect.  However, it does have the same asymptotics as the correct
expression (see Corollary \ref{A-estimates}), which is sufficient for
the rest of the argument in \cite{bbbc-jfa}, so that its overall
correctness is not affected.

\begin{proof}
  Fix $u, \eta$.  Form an orthonormal basis for
  $T_{(u,\eta)}\Phi^{-1}(G) \cong \R^{2n+m}$ as follows.  Let
  $\hat{u}$ be a unit vector in the direction of $(u,0)$, $\hat{v}$ a
  unit vector in the direction of $(J_\eta u,0)$.  For $i=1, \dots,
  n-1$ let $\hat{w}_i, \hat{y}_i \in \R^{2n} \times 0$ be unit vectors
  such that $\hat{w}_i$ is orthogonal to $\hat{u}, \hat{v}, \hat{w}_j,
  \hat{y}_j, 1 \le j < i$, and let $\hat{y}_i$ be in the direction of
  $J_\eta \hat{w}_i$ so that $\hat{y}_i$ is orthogonal to $\hat{u},
  \hat{v}, \hat{w}_j, \hat{y}_j, 1 \le j < i$ as well as to
  $\hat{w}_i$.  (To see this, note that if $\inner{x}{y}=0$ and
  $\inner{x}{J_z y}=0$, then $\inner{J_z x}{y} = 0$ and $\inner{J_z
    x}{J_z y} = -\abs{z}^2\inner{x}{y}=0$.)  Let
  $\hat{\eta}$ be a unit vector in the direction of $(0,\eta)$, and
  let $\hat{\zeta}_k, k=1,\dots,m-1$ be orthonormal vectors in $0
  \times \R^{m}$ which are orthogonal to $\hat{\eta}$.  Then
  $\{\hat{u}, \hat{v}, \hat{w}_i, \hat{y}_i, \hat{\eta},
  \hat{\zeta}_k\}$ form an orthonormal basis for $\R^{2n+m}$.  Note
  $J_\eta \hat{u} = \abs{\eta} \hat{v}$, $J_\eta \hat{v} = -\abs{\eta}
  \hat{u}$,$J_\eta \hat{w}_i = \abs{\eta} \hat{y}_i$, $J_\eta
  \hat{y}_i = -\abs{\eta} \hat{w}_i$.  Then
  \begin{align*}
    \partial_{\hat{u}} \Phi(u,\eta) &= (1-\cos\abs{\eta})\hat{u}
    + \sin\abs{\eta} \hat{v} + \abs{u}
    \left(\abs{\eta}-\sin\abs{\eta}\right)\hat{\eta}\\
    \partial_{\hat{v}} \Phi(u,\eta) &= (1-\cos\abs{\eta})\hat{v}
    - \sin\abs{\eta} \hat{u} \\
    \partial_{\hat{w}_i} \Phi(u,\eta) &= (1-\cos\abs{\eta})\hat{w}_i
    + \sin\abs{\eta} \hat{y}_i \\
    \partial_{\hat{y}_i} \Phi(u,\eta) &= (1-\cos\abs{\eta})\hat{y}_i
    - \sin\abs{\eta} \hat{w}_i \\
    \partial_{\hat{\eta}} \Phi(u,\eta) &= \abs{u}(\sin\abs{\eta}) \hat{u} +
    \abs{u}(\cos\abs{\eta}) \hat{v} + \frac{\abs{u}^2}{2}
    \left(1-\cos\abs{\eta}\right)\hat{\eta} \\
    \partial_{\hat{\zeta}_k} \Phi(u,\eta) &=
    \frac{\sin{\abs{\eta}}}{\abs{\eta}} J_{\hat{\zeta}_k} u + \frac{\abs{u}^2}{2}
    \left(1-\frac{\sin\abs{\eta}}{\abs{\eta}}\right)\hat{\zeta}_k.
  \end{align*}
  In this basis, the Jacobian matrix has the form
  \begin{align}
    J &= 
    \begin{pmatrix}
      1 - \cos\abs{\eta} & -\sin\abs{\eta} & 0 & \abs{u}\sin\abs{\eta}
      & 0 \\
      \sin\abs{\eta} & 1-\cos\abs{\eta} & 0 & \abs{u}\cos\abs{\eta} &
      0 \\
      0 & 0 & B & 0 & * \\
      \abs{u}(\abs{\eta}-\sin\abs{\eta}) & 0 & 0 &
      \frac{\abs{u}^2}{2}(1-\cos\abs{\eta}) & 0 \\
      0 & 0 & 0 & 0 & D
    \end{pmatrix}_{(2n+m) \times (2n+m)}
    \intertext{where}
    B &:= \begin{pmatrix}
            1 - \cos\abs{\eta} & -\sin\abs{\eta} &  & & \\
            \sin\abs{\eta} & 1-\cos\abs{\eta} & & & \\
            & & \ddots & & \\
            & & & 1 - \cos\abs{\eta} & -\sin\abs{\eta} \\
            & & & \sin\abs{\eta} & 1-\cos\abs{\eta}
    \end{pmatrix}_{2(n-1) \times 2(n-1)}
    \intertext{ is a block-diagonal matrix of $2\times 2$ blocks, and}
    D &:=     \begin{pmatrix}
      \frac{\abs{u}^2}{2}
      \left(1-\frac{\sin\abs{\eta}}{\abs{\eta}}\right) & & \\
      & \ddots & \\
      & & \frac{\abs{u}^2}{2}
      \left(1-\frac{\sin\abs{\eta}}{\abs{\eta}}\right)
    \end{pmatrix}_{(m-1)\times (m-1)}
  \end{align}
  is diagonal.  Note $\abs{B} = (2-2\cos\abs{\eta})^{n-1}$ and $\abs{D} = \left(\frac{\abs{u}^2}{2}
      \left(1-\frac{\sin\abs{\eta}}{\abs{\eta}}\right)\right)^{m-1}$.

  So factoring out $\abs{D}$ and expanding about the $\hat{\eta}$ row,
  we have
  \begin{align*}
    \abs{J} &= \abs{D} \left(\abs{u}(\abs{\eta}-\sin\abs{\eta}) 
    \begin{vmatrix}
      -\sin\abs{\eta} & 0 & \abs{u}\sin\abs{\eta} \\
      1-\cos\abs{\eta} & 0 & \abs{u}\cos\abs{\eta} \\
      0 & B & 0       
    \end{vmatrix}
    \right. \\
    &\quad \left. + \frac{\abs{u}^2}{2}(1-\cos\abs{\eta}) 
    \begin{vmatrix}
      1 - \cos\abs{\eta} & -\sin\abs{\eta} & 0 \\
      \sin\abs{\eta} & 1-\cos\abs{\eta} & 0  \\
      0 & 0 & B 
    \end{vmatrix}
    \right) \\
    &= \left(\frac{\abs{u}^2}{2}
      \left(1-\frac{\sin\abs{\eta}}{\abs{\eta}}\right)\right)^{m-1} 
      \\ &\quad \times
      \left(\abs{u}(\abs{\eta}-\sin\abs{\eta})
      (-\abs{u}\sin\abs{\eta}) (2-2\cos\abs{\eta})^{n-1} + 
      \frac{\abs{u}^2}{2}(1-\cos\abs{\eta}) (2-2\cos\abs{\eta})^{n}
    \right)  \\
    &= \left(\frac{\abs{u}^2}{2}
      \left(1-\frac{\sin\abs{\eta}}{\abs{\eta}}\right)\right)^{m-1}
      \abs{u}^2 (2-2\cos\abs{\eta})^{n-1} \left((\abs{\eta}-\sin\abs{\eta})
      (-\sin\abs{\eta})  + 
      (1-\cos\abs{\eta})^2\right) \\
    &= \abs{u}^{2m} \left(\frac{1}{2}-\frac{\sin\abs{\eta}}{2\abs{\eta}}\right)^{m-1}
      (2-2\cos\abs{\eta})^{n-1} \left(2 - 2\cos\abs{\eta} - \abs{\eta}\sin\abs{\eta}\right).
  \end{align*}
\end{proof}

\begin{corollary}\label{A-estimates}
  \begin{equation}\label{A-estimates-eqn}
    A(u, \eta) \asymp \abs{u}^{2m} \abs{\eta}^{2(m+n)} (2\pi-\abs{\eta})^{2n-1}.
  \end{equation}
\end{corollary}
\begin{proof}
  The asymptotic equivalence near $\abs{\eta}=0$ and $\abs{\eta}=2\pi$
  follows from a routine Taylor series computation.

  It then suffices to show that $A(u,\eta) > 0$ for all $0 <
  \abs{\eta} < 2\pi$.  We have $\frac{1}{2} -
  \frac{\sin\abs{\eta}}{2\abs{\eta}} > 0$ for all $\abs{\eta} > 0$,
  since $x > \sin x$ for all $x > 0$.  We also have $2-2\cos\abs{\eta}
  > 0$ for all $0 < \abs{\eta} < 2\pi$.

  Finally, to show $f(\abs{\eta}) := 2 - 2\cos\abs{\eta} -
  \abs{\eta}\sin\abs{\eta} > 0$, let $\theta = \frac{1}{2}\abs{\eta}$.
  Using double-angle identities, we have $f(2\theta) = 4 \sin\theta
  (\sin\theta - \theta\cos\theta)$.  For $0 < \theta < \pi$ we have
  $\sin\theta > 0$ so it suffices to show $g(\theta) :=
  \sin\theta-\theta\cos\theta > 0$.  But we have $g(0)=0$ and
  $g'(\theta) = \theta\sin\theta > 0$ for $0 < \theta < \pi$.
\end{proof}

The heat kernel estimates will be used to prove a technical lemma
regarding integrating the heat kernel along a geodesic.  The proof
requires the following simple fact from calculus, a close
relative of Lemma \ref{exp-poly-est}.

\begin{lemma}\label{gaussian-calculus}
  For any $q \in \R$, $a_0 > 0$ there exists a constant $C = C_{q,a_0}$ such that for any
  $a \ge a_0$ we have
  \begin{equation}
    \int_{t=1}^{t=\infty} t^q e^{-(at)^2}\,dt \le C \frac{1}{a^2} e^{-a^2}.
  \end{equation}
\end{lemma}

\begin{proof}
  Apply Lemma \ref{exp-poly-est}, taking $w_0=a^2$, $\alpha =
  (q-1)/2$, and making the change of variables $w = (a t)^2$.
\end{proof}

Let $B := \{ g : d(0,g) \le 1\}$ be the Carnot-Carath\'eodory unit
ball.

\begin{lemma}\label{geodesic-integral-bound}
  For each $q \in \R$ there exists a constant $C_q$ such that for
  all $u, \eta$ with $\Phi(u, \eta) \in B^C$, i.e. $\abs{u}\abs{\eta}
  \ge 1$,  we have
  \begin{align}
    \int_{t=1}^{t=\frac{2\pi}{\abs{\eta} }} p_1(u,t\eta) A(u,t\eta)
    t^{q}\,dt &\le  \frac{C_q}{(\abs{u} \abs{\eta})^2} p_1(u,\eta)
    A(u,\eta) \label{geodesic-integral-bound-stronger} \\
    &\le C_q  p_1(u,\eta) A(u,\eta). \label{geodesic-integral-bound-weaker}
  \end{align}
\end{lemma}

Note that (\ref{geodesic-integral-bound-weaker}) follows immediately
from the stronger statement (\ref{geodesic-integral-bound-stronger}),
since by assumption $\abs{u}\abs{\eta} \ge 1$.  In fact, we shall only
use (\ref{geodesic-integral-bound-weaker}) in the sequel.

\fixnotme{Connect to Cheeger-type isoperimetric inequality from \cite{bbbc-jfa}.}

\begin{proof}
  Assume throughout that $\abs{u}\abs{\eta} \ge 1$ and $0 < \abs{\eta}
  < 2\pi$.

  The proof involves the fact that a geodesic passes through (up to) three
  regions of $G$ in which the estimates for $p_1$ and $A$ simplify in
  different ways.  We define these regions, which partition
  $B^C$, as follows.  See Figure \ref{regions-fig}.
  \begin{figure}
    \centering
    \includegraphics[height=3in]{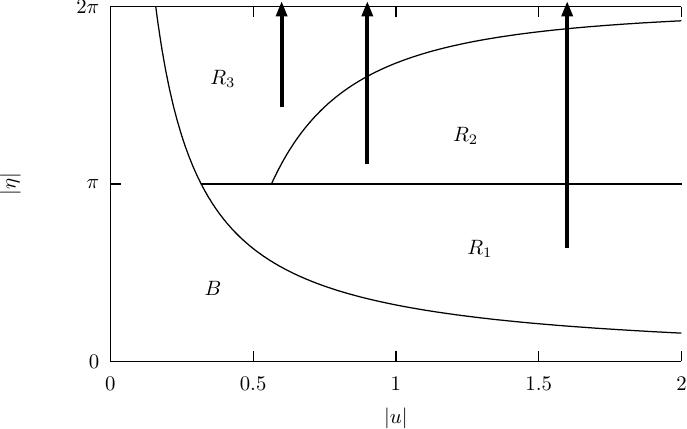}
    \caption[The regions $R_1, R_2, R_3$]{The regions $R_1, R_2, R_3$, seen in the
      $\abs{u}$-$\abs{\eta}$ plane.  The dark lines indicate examples of
      the geodesic paths of integration used in
      (\ref{geodesic-integral-bound-stronger}).  }\label{regions-fig}
  \end{figure}
  \begin{enumerate}
    \item Region $R_1$ is the set of $\Phi(u,\eta)$ such that
       $0 <
      \abs{\eta} \le \pi$.  (This corresponds to having $\abs{x}^2
      \lesssim \abs{z}$.) In this region we have $\abs{u} \ge
      \frac{1}{\pi}$ and $\pi \le 2\pi - \abs{\eta} < 2\pi$.
      Therefore (\ref{p1-u-eta-estimate}) becomes
      \begin{align*}
        p_1(u, \eta) &\asympon{R_1} (\abs{u}\abs{\eta})^{-m}
        e^{-\frac{1}{4}(\abs{u}\abs{\eta})^2}
        \intertext{and Corollary \ref{A-estimates} yields}
        A(u, \eta) &\asympon{R_1} \abs{u}^{2m} \abs{\eta}^{2(n+m)}
        \intertext{so that}
        p_1(u, \eta) A(u, \eta) &\asympon{R_1} \abs{u}^m
        \abs{\eta}^{2n+m} e^{-\frac{1}{4}(\abs{u}\abs{\eta})^2} =: F_1(u,\eta).
      \end{align*}

    \item Region $R_2$ is the set of $\Phi(u,\eta)$ such that
       $\pi <
      \abs{\eta} \le 2\pi-\frac{1}{\abs{u}^2}$.  (This corresponds to having
      $\abs{x}^2 \gtrsim \abs{z}$ and $\abs{x}^2 \abs{z} \gtrsim 1$.)
      In this region, we have $\abs{u}^2 \abs{\eta}^2
      (2\pi-\abs{\eta}) \ge \pi^2$, so that
      \begin{align*}
        p_1(u, \eta) &\asympon{R_2} \abs{u}^{-m}
        (2\pi-\abs{\eta})^{-n+\frac{1}{2}}
        e^{-\frac{1}{4}(\abs{u}\abs{\eta})^2} \\
        A(u, \eta) &\asympon{R_2} \abs{u}^{2m} (2\pi-\abs{\eta})^{2n-1}
        \\
        p_1(u, \eta) A(u, \eta) &\asympon{R_2} \abs{u}^m
        (2\pi-\abs{\eta})^{n-\frac{1}{2}}
        e^{-\frac{1}{4}(\abs{u}\abs{\eta})^2} =: F_2(u,\eta)
        \\ &\asympon{R_2}  \abs{u}^m \abs{\eta}^{2n+m}
        (2\pi-\abs{\eta})^{n-\frac{1}{2}}
        e^{-\frac{1}{4}(\abs{u}\abs{\eta})^2} =: \tilde{F}_2(u, \eta).
      \end{align*}
      We shall use the estimates $F_2, \tilde{F}_2$ at different
      times.  Although $F_2 \asympon{R_2} \tilde{F}_2$ (since
      $\abs{\eta} \asympon{R_2} 1$), they are not
      equivalent on $R_1$.

    \item Region $R_3$ is the set of $\Phi(u,\eta)$ such that
       $\abs{\eta} > \max\left(\pi,
      2\pi-\frac{1}{\abs{u}^2}\right)$.  (This corresponds to having
      $\abs{x}^2 \gtrsim \abs{z}$ and $\abs{x}^2 \abs{z} \lesssim 1$.)
      In this region, we have $\abs{u}^2 \abs{\eta}^2
      (2\pi-\abs{\eta}) < (2\pi)^2$, so that
      \begin{align*}
        p_1(u, \eta) &\asympon{R_3} \abs{u}^{2n-m-1}
        e^{-\frac{1}{4}(\abs{u}\abs{\eta})^2} \\
        A(u, \eta) &\asympon{R_3} \abs{u}^{2m} (2\pi-\abs{\eta})^{2n-1}
        \\
        p_1(u, \eta) A(u, \eta) &\asympon{R_3} \abs{u}^{2n+m-1}
        (2\pi-\abs{\eta})^{2n-1} e^{-\frac{1}{4}(\abs{u}\abs{\eta})^2}
        =: F_3(u,\eta).
      \end{align*}
  \end{enumerate}

  We observe that a geodesic starting from the origin (given by $t
  \mapsto \Phi(u, t\eta)$ for some fixed $u,\eta$) passes through
  these regions in order, except that it skips Region 2 if $\abs{u} < \pi^{-1/2}$.

  We now estimate the desired integral along a portion of a geodesic
  lying in a single region.
  \begin{claim}\label{single-region}
    Let $q \in \R$.  Suppose that $F : G \to \R$ is given by
    \begin{equation*}
      F(u,\eta) =  \abs{u}^{\alpha}
      \abs{\eta}^{\beta} (2\pi-\abs{\eta})^{\gamma}
      e^{-\frac{1}{4}(\abs{u}\abs{\eta})^2}
    \end{equation*}
    for some nonnegative powers $\alpha, \beta, \gamma$, and that
    there is some region $R \subset G$ such that $F \asympon{R} p_1
    A$.  Then there is a constant $C$ depending on
    $q$, $F$, $R$ such that for all $u, \eta, \tau_0, \tau_1, \tau_2$
    satisfying
    \begin{itemize}
    \item $1 \le \tau_0 \le \tau_1 \le \tau_2 \le
    \frac{2\pi}{\abs{\eta}}$; and
    \item $\Phi(u, t\eta) \in R$ for all $t \in
          [\tau_1,\tau_2]$
    \end{itemize}
    we have
    \begin{equation}
      \int_{t=\tau_1}^{t=\tau_2} p_1(u,t\eta) A(u, t\eta) t^q\,dt \le
      C \frac{\tau_0^{q-1}}{(\abs{u}\abs{\eta})^2} F(u,\tau_0 \eta).
    \end{equation}
  \end{claim}
  
\begin{proof}[Proof of Claim \ref{single-region}] We have
  \begin{align*}
    \int_{t=\tau_1}^{t=\tau_2} p_1(u,t\eta) A(u, t\eta) t^q\,dt &\le C
    \int_{t=\tau_1}^{t=\tau_2} F(u, t\eta) t^q\,dt \\
    &\le C \int_{t=\tau_0}^{t=\tau_2} F(u, t\eta) t^q\,dt \\
    &= C \abs{u}^{\alpha} \abs{\eta}^{\beta} \int_{t=\tau_0}^{t=\tau_2}
    t^{q+\beta_i} (2\pi-t\abs{\eta})^{\gamma}
    e^{-\frac{1}{4}(t\abs{u}\abs{\eta})^2}\,dt \\
    &\le C \abs{u}^{\alpha} \abs{\eta}^{\beta} (2\pi-\tau_0 \abs{\eta})^{\gamma}\int_{t=\tau_0}^{t=\tau_2}
    t^{q+\beta} 
    e^{-\frac{1}{4}(t\abs{u}\abs{\eta})^2}\,dt
\intertext{since $t \ge \tau_0$.  We now make the change of
  variables $t = t' \tau_0$:}
    &\le C \abs{u}^{\alpha} \abs{\eta}^{\beta}
(2\pi-\tau_0 \abs{\eta})^{\gamma} \tau_0^{q+\beta+1} \int_{t'=1}^{t'=\infty}
    t'^{q+\beta} 
    e^{-\frac{1}{4}(t' \tau_0 \abs{u}\abs{\eta})^2}\,dt' \\
    &\le C' \abs{u}^{\alpha} \abs{\eta}^{\beta}
    (2\pi-\tau_0 \abs{\eta})^{\gamma} \tau_0^{q+\beta+1} \frac{1}{(\tau_0
      \abs{u}\abs{\eta})^2} e^{-\frac{1}{4}(\tau_0 \abs{u}\abs{\eta})^2} \\
    &= C' \frac{\tau_0^{q-1}}{(\abs{u}\abs{\eta})^2} F_i(u, \tau_0 \eta)
  \end{align*}
  where in the second-to-last line we applied Lemma
  \ref{gaussian-calculus} with $a = \frac{1}{2} \tau_0
  \abs{u}\abs{\eta}$, \mbox{$a_0 = \frac{1}{2}$}.
\end{proof}

  Now for fixed $u, \eta$, let
  \begin{align*}
    t_2 &:= \max\left(1, \frac{\pi}{\abs{\eta}}\right) \\
    t_3 &:= \max\left(t_2, \frac{1}{\abs{\eta}}\left(2\pi-\frac{1}{\abs{u}^2}\right)\right)
  \end{align*}
  so that
  \begin{align*}
    \Phi(u, t\eta) &\in R_1 \text{ for } 1 < t \le t_2 \\
    \Phi(u, t\eta) &\in R_2 \text{ for } t_2 < t < t_3 \\
    \Phi(u, t\eta) &\in R_3 \text{ for } t_3 \le t <
    \frac{2\pi}{\abs{\eta}}.
  \end{align*}

  We divide the remainder of the proof into cases, depending on the
  region where $\Phi(u,\eta)$ resides.

  \begin{case}
  Suppose that $\Phi(u,\eta) \in R_1$.  We have
  \begin{equation*}
    \int_{t=1}^{t=\frac{2\pi}{\abs{\eta}}} p_1(u,t\eta) A(u,t\eta)
    t^{q}\,dt = \int_{t=1}^{t=t_2} + \int_{t=t_2}^{t=t_3} +
    \int_{t=t_3}^{t=\frac{2\pi}{\abs{\eta}}}.
  \end{equation*}

  For the first integral, where $\Phi(u, t\eta) \in R_1$, we have by
  Claim \ref{single-region} (taking $\tau_0 = \tau_1 = 1$, $\tau_2 =
  t_2$, $R = R_1$, $F = F_1$) that
  \begin{align*}
    \int_{t=1}^{t=t_2} p_1(u,t\eta) A(u,t\eta) t^{q}\,dt \le 
    \frac{C}{(\abs{u}\abs{\eta})^2} F_1(u,\eta) \le \frac{C'}{(\abs{u}\abs{\eta})^2} p_1(u,\eta) A(u,\eta)
  \end{align*}
  since $F_1 \asympon{R_1} p_1 A$.  

  For the second integral, where $\Phi(u, t\eta) \in R_2$, we take
  $\tau_0 = 1$, $\tau_1 = t_2$, $\tau_2 = t_3$, $R = R_2$, $F =
  \tilde{F}_2$ in Claim \ref{single-region} to obtain
  \begin{align*}
    \int_{t=t_2}^{t=t_3} p_1(u,t\eta) A(u,t\eta) t^{q}\,dt 
    &\le  \frac{C}{(\abs{u}\abs{\eta})^2} \tilde{F}_2(u,\eta).
  \end{align*}
  However, for $\Phi(u, \eta) \in R_1$ we have
  \begin{align*}
    \frac{\tilde{F}_2(u, \eta)}{F_1(u,\eta)} =
    (2\pi-\abs{\eta})^{n-\frac{1}{2}} \le (2\pi)^{n-\frac{1}{2}}.
  \end{align*}
  Thus
  \begin{align*}
    \int_{t=t_2}^{t=t_3} p_1(u,t\eta) A(u,t\eta) t^{q}\,dt &\le
    \frac{C'}{(\abs{u}\abs{\eta})^2} {F}_1(u,\eta) \\
    &\le \frac{C''}{(\abs{u}\abs{\eta})^2} p_1(u, \eta) A(u, \eta).
  \end{align*}

  The third integral is more subtle.  We apply Claim
  \ref{single-region} with $\tau_0 = \tau_1 = t_3$, $\tau_3 =
  \frac{2\pi}{\abs{\eta}}$, $R = R_3$, $F = {F}_3$:
  \begin{align*}
    \int_{t=t_3}^{t=\frac{2\pi}{\abs{\eta}}} p_1(u,t\eta) A(u,t\eta)
    t^{q}\,dt &\le C \frac{t_3^{q-1}}{(\abs{u}\abs{\eta})^2} {F}_3(u,
    t_3 \eta).
  \end{align*}
  Then 
  \begin{align}\label{F3-ratio}
    \frac{t_3^{q-1} {F}_3(u, t_3 \eta)}{F_1(u, \eta)} &=
    t_3^{q-1} \abs{u}^{2n-1} \abs{\eta}^{-2n-m}
    (2\pi-t_3 \abs{\eta})^{2n-1} e^{-\frac{1}{4} (\abs{u} \abs{\eta})^2
      (t_3^2 - 1)}.
  \end{align}
  We must show that this ratio is bounded.  Fix some $\epsilon > 0$. If $\abs{u} \ge
  (\pi-\epsilon)^{-1/2} > \pi^{-1/2}$, we have
  $2\pi-\frac{1}{\abs{u}^2} > \pi + \epsilon$ and thus $t_3 =
  \frac{1}{\abs{\eta}}\left(2\pi-\frac{1}{\abs{u}^2}\right)$.  Then
  \begin{align*}
     \abs{\eta}^2  (t_3^2 - 1) &= \left(2\pi-\frac{1}{\abs{u}^2}\right)^2 -
     \abs{\eta}^2 \\
     &\ge (\pi + \epsilon)^2 - \pi^2 = 2 \pi \epsilon + \epsilon^2.
  \end{align*}
  So in this case (\ref{F3-ratio}) becomes
  \begin{align*}
    \frac{t_3^{q-1} {F}_3(u, t_3 \eta)}{F_1(u, \eta)} &=
    \left(\frac{1}{\abs{\eta}}\left(2\pi-\frac{1}{\abs{u}^2}\right)\right)^{q-1}
    \abs{u}^{2n-1} \abs{\eta}^{-2n-m}
    \left(\frac{1}{\abs{u}^2}\right)^{2n-1} e^{-\frac{1}{4} (\abs{u} \abs{\eta})^2
      (t_3^2 - 1)} \\
    &= \left(2\pi-\frac{1}{\abs{u}^2}\right)^{q-1} \abs{u}^{-2n+1}
    \abs{\eta}^{-2n-m-q+1}  e^{-\frac{1}{4} (\abs{u} \abs{\eta})^2
      (t_3^2 - 1)} \\
    &\le (2\pi)^{q-1} \abs{u}^{m+q} e^{-\frac{1}{4}(2 \pi \epsilon + \epsilon^2)\abs{u}^2}
  \end{align*}
  since $\abs{\eta} \le \frac{1}{\abs{u}}$.  This is certainly bounded
  by some constant.  On the other hand, if $\abs{u} \le
  (\pi-\epsilon)^{-1/2}$, then $\abs{\eta} \ge (\pi-\epsilon)^{1/2}$
  and $1 \le t_3 \le \left(\frac{\pi + \epsilon}{\pi -
    \epsilon}\right)^{1/2}$, so that the right side of (\ref{F3-ratio})
  is clearly bounded.

  Thus we have
  \begin{align*}
    \int_{t=t_3}^{t=\frac{2\pi}{\abs{\eta}}} p_1(u,t\eta) A(u,t\eta) t^{q}\,dt &\le
    \frac{C'}{(\abs{u}\abs{\eta})^2} {F}_1(u,\eta) \\
    &\le \frac{C''}{(\abs{u}\abs{\eta})^2} p_1(u, \eta) A(u, \eta).
  \end{align*}
  This completes the proof of this case.
\end{case}

\begin{case}
  Suppose that $\Phi(u,\eta) \in R_2$.  We have
  \begin{equation*}
    \int_{t=1}^{t=\frac{2\pi}{\abs{\eta}}} p_1(u,t\eta) A(u,t\eta)
    t^{q}\,dt = \int_{t=1}^{t=t_3} + \int_{t=t_3}^{t=\frac{2\pi}{\abs{\eta}}}.
  \end{equation*}
  Note that in this region we have $1 \le t_3 \le 2$.
  Again by Claim \ref{single-region}, with $\tau_0 = \tau_1 = 1$ and
  $\tau_2 = t_3$, we have
  \begin{align*}
    \int_{t=1}^{t=t_3} p_1(u,t\eta) A(u,t\eta) t^{q}\,dt 
    \le  \frac{C}{(\abs{u}\abs{\eta})^2} F_2(u,\eta) 
    \le  \frac{C'}{(\abs{u}\abs{\eta})^2} p_1(u,\eta) A(u,\eta).
  \end{align*}
  For the second integral, we apply Claim \ref{single-region} with
  $\tau_0 = 1$, $\tau_1 = t_3$, $\tau_2 = \frac{2\pi}{\abs{\eta}}$ to get
  \begin{align*}
    \int_{t=t_3}^{t=\frac{2\pi}{\abs{\eta}}} p_1(u,t\eta) A(u,t\eta)
    t^{q}\,dt &\le  \frac{C}{(\abs{u}\abs{\eta})^2} F_3(u,
    \eta).
  \end{align*}
  But $\abs{\eta} \ge 2\pi - \frac{1}{\abs{u}^2}$ on $R_3$, so we have
  \begin{align*}
    \frac{F_3(u, \eta)}{F_2(u, \eta)} &= \abs{u}^{2n-1}
      (2\pi-\abs{\eta})^{n-\frac{1}{2}} \\
      &\le \abs{u}^{2n-1}
        \left(\frac{1}{\abs{u}^2}\right)^{n-\frac{1}{2}} = 1.      
  \end{align*}
  Thus
  \begin{align*}
    \int_{t=t_3}^{t=\frac{2\pi}{\abs{\eta}}} p_1(u,t\eta) A(u,t\eta) t^{q}\,dt &\le
    \frac{C'}{(\abs{u}\abs{\eta})^2} {F}_2(u,\eta) \\
    &\le \frac{C''}{(\abs{u}\abs{\eta})^2} p_1(u, \eta) A(u, \eta).
  \end{align*}
\end{case}

\begin{case}
  Suppose $\Phi(u,\eta) \in R_3$; we apply Claim \ref{single-region}
  with $\tau_0 = \tau_1 = 1$, $\tau_2 = \frac{2\pi}{\abs{\eta}}$ to
  get
  \begin{equation*}
    \int_{t=1}^{t=\frac{2\pi}{\abs{\eta}}} p_1(u,t\eta) A(u,t\eta) t^{q}\,dt \le C
    \frac{1}{(\abs{u}\abs{\eta})^2} F_3(u,\eta) \le C' p_1(u,\eta) A(u,\eta).
  \end{equation*}
\end{case}

The three cases together complete the proof of Lemma \ref{geodesic-integral-bound}.
\end{proof}

\begin{notation}
   For $f \in C^1(G)$, let $m_f := \frac{\int_B
     f\,d\haar}{\int_B d\haar}$, where $B$ is the Carnot-Carath\'eodory unit ball.
\end{notation}

To continue to follow the line of \cite{bbbc-jfa}, we need the
following Poincar\'e inequality.  This theorem can be found in
\cite{jerison}, and is a special case of a more general theorem
appearing in \cite{maheux-saloffcoste}.

\index{Poincar\'e inequality}
\begin{theorem}\label{poincare}
  There exists a constant $C$ such that for any $f \in C^\infty(G)$,
  \begin{equation}\label{poincare-eq}
    \int_B \abs{f - m_f}\,d\haar \le C \int_B \abs{\grad f}\,d\haar.
  \end{equation}
\end{theorem}

\begin{corollary}\label{poincare-p_1}
    There exists a constant $C$ such that for any $f \in C^\infty(G)$,
  \begin{equation}\label{poincare-p_1-eq}
    \int_B \abs{f - m_f} p_1 \,d\haar \le C \int_B \abs{\grad f}p_1 \,d\haar.
  \end{equation}
\end{corollary}

\begin{proof}
  $p_1$ is bounded and bounded below away from $0$ on $B$.
\end{proof}

\begin{lemma}[akin to Lemma 5.2 of \cite{bbbc-jfa}]\label{bbbc-52}
  There exists a constant $C$ such that for all $f \in \funcclass$,
  \begin{equation}\label{bbbc-52-eq}
    \int_{B^C} \abs{f-m_f} p_1\,d\haar \le C \int_G \abs{\grad f} p_1\,d\haar.
  \end{equation}
\end{lemma}

\begin{proof}
  Changing to $(u,\eta)$ coordinates, we wish to show
  \begin{equation}
    \int_{\abs{u} \ge \frac{1}{2\pi}} \int_{\frac{1}{\abs{u}} \le
      \abs{\eta} < 2\pi} \abs{f(\Phi(u,\eta))-m_f}p_1(\Phi(u,\eta))
    A(u,\eta)\,d\eta\,du \le C \int_{G} \abs{\grad f} p_1 \,d\haar.
  \end{equation}
  By an abuse of notation we shall write $f(u, \eta)$ for
  $f(\Phi(u,\eta))$, $p_1(u,\eta)$ for $p_1(\Phi(u,\eta))$, $\grad
  f(u,\eta)$ for $(\grad f)(\Phi(u,\eta))$, et cetera.

  Let $g(u,\eta) := f\left(u,
  \min\left(\abs{\eta},\frac{1}{\abs{u}}\right)
  \frac{\eta}{\abs{\eta}}\right)$.  Then $g=f$ on $B$ (in particular
  $m_g = m_f$), $g$ is bounded, the function $s \mapsto g(u, s\eta)$
  is absolutely continuous, and $\frac{d}{ds}g(u,s\eta) = 0$ for $s >
  \frac{1}{\abs{u}\abs{\eta}}$.

  Now $\abs{f - m_f} \le \abs{f-g} + \abs{g - m_f}$.  We first
  observe that for $\abs{u}\abs{\eta} \ge 1$ we have
  \begin{align*}
    \abs{f(u,\eta) - g(u,\eta)} &=
    \abs{\int_{s=\frac{1}{\abs{u}\abs{\eta}}}^{s=1} \left(\frac{d}{ds}
      f(u,s\eta)-\cancel{\frac{d}{ds}g(u,s\eta)}\right)\,ds} \\
    &\le \int_{s=\frac{1}{\abs{u}\abs{\eta}}}^{s=1} \abs{\grad
      f(u,s\eta)} \abs{u} \abs{\eta}\,ds
  \end{align*} 
  by (\ref{geodesic-grad}).  Thus
  \begin{align*}
    \int_{B^C} \abs{f-g}p_1\,d\haar &= \int_{\abs{u} \ge \frac{1}{2\pi}} \int_{\abs{\eta} \ge
      \frac{1}{\abs{u}}} \abs{f(u,\eta)-g(u,\eta)}p_1(u,\eta)
    A(u,\eta)\,d\eta\,du
\\
\intertext{where the limits of integration come from the conditions
  $\abs{u} \abs{\eta} \ge 1$, $\abs{\eta} < 2\pi$;}
 &\le \int_{\abs{u} \ge \frac{1}{2\pi}} \int_{\abs{\eta} \ge
      \frac{1}{\abs{u}}} \int_{s=\frac{1}{\abs{u}\abs{\eta}}}^{s=1} \abs{\grad
      f(u,s\eta)} \abs{u} \abs{\eta}p_1(u,\eta)
    A(u,\eta)\,ds\,d\eta\,du \\
 &= \int_{\abs{u} \ge \frac{1}{2\pi}} \int_{s=0}^{s=1}
    \int_{\frac{1}{s\abs{u}} \le \abs{\eta} \le 2\pi} \abs{\grad
      f(u,s\eta)} \abs{u} \abs{\eta}p_1(u,\eta)
    A(u,\eta)\,d\eta\,ds\,du
\intertext{by Tonelli's theorem.  We now make the change of variables $\eta' =
  s\eta$ to obtain}
&= \int_{\abs{u} \ge \frac{1}{2\pi}} \int_{s=0}^{s=1}
    \int_{\frac{1}{\abs{u}} \le \abs{\eta'} \le 2\pi s} \abs{\grad
      f(u,\eta')} \abs{u} \frac{1}{s}\abs{\eta'} \\
    &\quad \times p_1\left(u,\frac{1}{s}\eta'\right)
    A\left(u,\frac{1}{s}\eta'\right) \frac{1}{s^m}\,d\eta'\,ds\,du \\
    &= \int_{\abs{u} \ge \frac{1}{2\pi}} \int_{\frac{1}{\abs{u}} \le \abs{\eta'} \le 2\pi} \abs{\grad
      f(u,\eta')} \abs{u} \abs{\eta'} \\
    &\quad \times \left(\int_{s=\frac{\abs{\eta'}}{2\pi}}^{s=1}
     p_1\left(u,\frac{1}{s}\eta'\right)
    A\left(u,\frac{1}{s}\eta'\right) \frac{1}{s^{m+1}}\,ds\right)\,d\eta'\,du.
\intertext{Make the further change of variables $t=\frac{1}{s}$ to
  get}
    &= \int_{\abs{u} \ge \frac{1}{2\pi}} \int_{\frac{1}{\abs{u}} \le \abs{\eta'} \le 2\pi} \abs{\grad
      f(u,\eta')} \abs{u} \abs{\eta'} \\
&\quad \times \left(\int_{t=1}^{t=\frac{2\pi}{\abs{\eta'}}}
     p_1(u,t\eta')
    A(u,t\eta') t^{m-1}\,dt\right)\,d\eta'\,du.
    \intertext{Applying Lemma \ref{geodesic-integral-bound} to the
      bracketed term gives}
    &\le C \int_{\abs{u} \ge \frac{1}{2\pi}} \int_{\frac{1}{\abs{u}} \le
      \abs{\eta'} \le 2\pi} \frac{1}{\abs{u} \abs{\eta'}} \abs{\grad
      f(u,\eta')} p_1(u,\eta') A(u,\eta') \,d\eta' \,du  \\
    &\le C' \int_{B^C} \abs{\grad f} p_1\,d\haar
  \end{align*}
  converting back from geodesic coordinates and using the fact that
  $\abs{u}\abs{\eta'} \ge 1$.
  
  To complete the proof, we must show that $\int_{B^C} \abs{g - m_f}
  p_1\,d\haar \le \int_G \abs{\grad f}p_1\,d\haar$.  Note that
  for $\Phi(u,\eta) \in B^C$, i.e. $\abs{u} \abs{\eta} \ge 1$, we
  have $g(u,\eta) = f\left(u, \frac{1}{\abs{u}
    \abs{\eta}}\eta\right)$, so
  \begin{align}
    \int_{B^C} \abs{g - m_f} p_1\,d\haar &= \int_{\abs{u} \ge \frac{1}{2\pi}}
    \int_{\frac{1}{\abs{u}} \le \abs{\eta} \le 2\pi } \abs{f\left(u, \frac{1}{\abs{u}
    \abs{\eta}}\eta\right) - m_f} p_1(u,\eta) A(u,\eta) \,d\eta\,du.
    \intertext{Change the $\eta$ integral to polar coordinates by
      writing $\eta = \rho \hat{\eta}$, where $\rho \ge 0$ and
      $\abs{\hat{\eta}} = 1$.  Note that $p_1(u,\eta),
      A(u,\eta)$ depend on $\eta$ only through $\rho$ and not $\hat{\eta}$.}
&= C \int_{\abs{u} \ge \frac{1}{2\pi}}
    \int_{\hat{\eta} \in S^{m-1}} \abs{f\left(u, \frac{1}{\abs{u}}
      \hat{\eta}\right) - m_f}
    \\ &\quad \times \int_{\rho=\frac{1}{\abs{u}}}^{\rho=2\pi} p_1(u,\rho) A(u,\rho)
    \rho^{m-1}\,d\rho \,d\hat{\eta}\,du. \label{foogle}
  \end{align}
  Now, for any $s \in [0,1]$ we have
  \begin{equation}\label{gmf1}
    \abs{f\left(u, \frac{1}{\abs{u}}\hat{\eta}\right) - m_f} \le
    \abs{f\left(u, \frac{1}{\abs{u}}\hat{\eta}\right) - f\left(u,
      \frac{s}{\abs{u}}\hat{\eta}\right)} + \abs{ f\left(u,
      \frac{s}{\abs{u}}\hat{\eta}\right) - m_f}.
  \end{equation}
  Let 
\begin{equation}
D(u) := \int_{s=0}^{s=1} \frac{s^{m-1}}{\abs{u}^m} A\left(u,
  \frac{s}{\abs{u}}\right)\,ds.
\end{equation}
By multiplying both sides of (\ref{gmf1}) by
  $\frac{1}{D(u)}\frac{s^{m-1}}{\abs{u}^m} A\left(u,
  \frac{s}{\abs{u}}\right)$ and integrating we obtain
  \begin{equation}\label{xyzzy}
    \begin{split}
    \abs{f\left(u, \frac{1}{\abs{u}}\hat{\eta}\right) - m_f} &\le \frac{1}{D(u)} \int_{s=0}^{s=1} \left(\abs{f\left(u, \frac{1}{\abs{u}}\hat{\eta}\right) - f\left(u,
      \frac{s}{\abs{u}}\hat{\eta}\right)} + \abs{ f\left(u,
      \frac{s}{\abs{u}}\hat{\eta}\right) - m_f}\right) \\
    &\quad \times \frac{s^{m-1}}{\abs{u}^m} A\left(u,
  \frac{s}{\abs{u}}\right)\,ds.
  \end{split}
  \end{equation}
  Let
  \begin{equation}\label{R-def}
    R(u) := \frac{1}{D(u)} \int_{\rho=\frac{1}{\abs{u}}}^{\rho=2\pi}
      p_1(u,\rho) A(u,\rho) \rho^{m-1}\,d\rho.
  \end{equation}
  Then substituting (\ref{xyzzy}) into (\ref{foogle}) and using
  (\ref{R-def}) we have
  \begin{equation}
    \int_{B^C} \abs{g-m_f}p_1\,d\haar \le I_1 + I_2
  \end{equation}
  where
  \begin{align}
     I_1 &:= \int_{\abs{u} \ge \frac{1}{2\pi}}
    \int_{\hat{\eta} \in S^{m-1}} 
\int_{s=0}^{s=1} \abs{f\left(u, \frac{1}{\abs{u}}\hat{\eta}\right) - f\left(u, \frac{s}{\abs{u}}\hat{\eta}\right)} \frac{s^{m-1}}{\abs{u}^m} A\left(u,
\frac{s}{\abs{u}}\right)\,ds \,
   R(u) 
   \,d\hat{\eta}\,du \label{I1-def} \\
     I_2 &:= \int_{\abs{u} \ge \frac{1}{2\pi}}
    \int_{\hat{\eta} \in S^{m-1}} 
\int_{s=0}^{s=1} \abs{ f\left(u, \frac{s}{\abs{u}}\hat{\eta}\right) -
  m_f} \frac{s^{m-1}}{\abs{u}^m} A\left(u,
\frac{s}{\abs{u}}\right)\,ds\, R(u) \,d\hat{\eta}\,du. \label{I2-def}
  \end{align}
  We now show that $I_1$, $I_2$ can each be bounded by a constant
  times $\int_G \abs{\grad f}p_1\,d\haar$, using the
  following claim.
  \begin{claim}\label{R-claim}
    There exists a constant $C$ such that for all $\abs{u} \ge
    \frac{1}{2\pi}$ we have
    \begin{equation}
      R(u) \le C
      \left(2\pi-\frac{1}{\abs{u}}\right)^{2n-1} \le (2\pi)^{2n-1}C.
    \end{equation}
  \end{claim}

  \begin{proof}[Proof of Claim]
    First, by Corollary \ref{A-estimates} we have
    \begin{align*}
      D(u) &:= \int_{s=0}^{s=1} \frac{s^{m-1}}{\abs{u}^m} A\left(u,
  \frac{s}{\abs{u}}\right)\,ds \\
  &\ge C \int_{s=0}^{s=1} \frac{s^{m-1}}{\abs{u}^m} \abs{u}^{2m}
  \left(\frac{s}{\abs{u}}\right)^{2(m+n)} \left(2\pi -
  \frac{s}{\abs{u}}\right)^{2n-1}\,ds \\
  &= C \abs{u}^{-2n-m} \int_{s=0}^{s=1} s^{3m+2n-1} \left(2\pi -
  \frac{s}{\abs{u}}\right)^{2n-1}\,ds \\
  &\ge C \abs{u}^{-2n-m}  \int_{s=0}^{s=1} s^{3m+2n-1}
  \left(2\pi(1-s)\right)^{2n-1}\,ds && \text{since $u \ge
    \frac{1}{2\pi}$} \\
  &= C' \abs{u}^{-2n-m}
    \end{align*}
    since the $s$ integral is a positive constant independent of $u$.

    On the other hand, making the change of variables $\rho =
    \frac{t}{\abs{u}}$ shows
    \begin{align*}
      \int_{\rho=\frac{1}{\abs{u}}}^{\rho=2\pi}
      p_1(u,\rho) A(u,\rho) \rho^{m-1}\,d\rho &= \abs{u}^{-m} \int_{t=1}^{t=2\pi\abs{u}}
      p_1\left(u,\frac{t}{\abs{u}}\right) A\left(u,\frac{t}{\abs{u}}\right) t^{m-1}\,dt \\
      &\le C \abs{u}^{-m} p_1\left(u, \frac{1}{\abs{u}}\right) A\left(u, \frac{1}{\abs{u}}\right)
    \end{align*}
    by taking $\abs{\eta} = \frac{1}{\abs{u}}$ in Lemma
    \ref{geodesic-integral-bound}.  Now $p_1\left(u, \frac{1}{\abs{u}}\right)$ is the heat
    kernel evaluated at a point on the unit sphere of $G$, so this is
    bounded by a constant independent of $u$.  Thus by Corollary
    \ref{A-estimates} we have
    \begin{align*}
       \int_{\rho=\frac{1}{\abs{u}}}^{\rho=2\pi}
      p_1(u,\rho) A(u,\rho) \rho^{m-1}\,d\rho &\le C \abs{u}^{-m}
      \abs{u}^{2m} \left(\frac{1}{\abs{u}}\right)^{2(m+n)} \left(2\pi-\frac{1}{\abs{u}}\right)^{2n-1} \\
      &\le C \left(2\pi-\frac{1}{\abs{u}}\right)^{2n-1} \abs{u}^{-2n-m}.
    \end{align*}
    Combining this with the estimate on $D(u)$ proves the claim.
  \end{proof}

  To estimate $I_1$ (see (\ref{I1-def})), we observe that
  \begin{align*}
    \abs{f\left(u, \frac{1}{\abs{u}}\hat{\eta}\right) - f\left(u,
      \frac{s}{\abs{u}}\hat{\eta}\right)} &= \abs{\int_{t=s}^{t=1} \frac{d}{dt}
      f\left(u, \frac{t}{\abs{u}}\hat{\eta}\right)\,dt} \\
    &\le \int_{t=s}^{t=1} \abs{\frac{d}{dt}
      f\left(u, \frac{t}{\abs{u}}\hat{\eta}\right)}\,dt \\
    &\le \int_{t=s}^{t=1} \abs{\grad f\left(u,
        \frac{t}{\abs{u}}\hat{\eta}\right)}\,dt
  \end{align*}
  by (\ref{geodesic-grad}).  Thus
  \begin{align}
    I_1 &\le \int_{\abs{u} \ge \frac{1}{2\pi}} \int_{\hat{\eta} \in
      S^{m-1}} \int_{s=0}^{s=1} \int_{t=s}^{t=1} \abs{\grad f\left(u,
      \frac{t}{\abs{u}}\hat{\eta}\right)} \frac{s^{m-1}}{\abs{u}^m} A\left(u,
    \frac{s}{\abs{u}}\right)\,dt\,ds \, R(u) \,d\hat{\eta}\,du \\
    &= \int_{\abs{u} \ge \frac{1}{2\pi}} \int_{\hat{\eta} \in
      S^{m-1}} \int_{t=0}^{t=1} \abs{\grad f\left(u,
      \frac{t}{\abs{u}}\hat{\eta}\right)}
    \frac{1}{\abs{u}^m}\left(R(u)\int_{s=0}^{s=t}  s^{m-1} A\left(u, 
    \frac{s}{\abs{u}}\right)\,ds\right)\,dt \, d\hat{\eta}\,du.
  \end{align}
  Now by Claim \ref{R-claim} and Corollary \ref{A-estimates}, we have
  for all $t \in [0,1]$: 
  \begin{align*}
    R(u) \int_{s=0}^{s=t}  s^{m-1} A\left(u,
    \frac{s}{\abs{u}}\right)\,ds &\le C 
      \left(2\pi-\frac{1}{\abs{u}}\right)^{2n-1} \\ &\quad\times \int_{s=0}^{s=t}
      s^{m-1} \abs{u}^{2m} \left(\frac{s}{\abs{u}}\right)^{2(m+n)}
      \left(2\pi-\frac{s}{\abs{u}}\right)^{2n-1}\,ds \\
      &\le C 
      \left(2\pi-\frac{t}{\abs{u}}\right)^{2n-1} (2\pi)^{2n-1}
      \abs{u}^{-2n} \int_{s=0}^{s=t}
      s^{3m+2n-1}\,ds \\
      &= C' \left(2\pi-\frac{t}{\abs{u}}\right)^{2n-1} \abs{u}^{-2n}
      t^{3m+2n} \\
      &= C' \left(2\pi-\frac{t}{\abs{u}}\right)^{2n-1} \abs{u}^{2m}
      \left(\frac{t}{\abs{u}}\right)^{2(m+n)} t^m \\
      &\le C'' A\left(u, \frac{t}{\abs{u}}\right) t^m \\
      &\le C'' A\left(u, \frac{t}{\abs{u}}\right) t^{m-1}.
  \end{align*}
  Thus
  \begin{align}
    I_1 &\le C \int_{\abs{u} \ge \frac{1}{2\pi}} \int_{\hat{\eta} \in
      S^{m-1}} \int_{t=0}^{t=1} \abs{\grad f\left(u,
      \frac{t}{\abs{u}}\hat{\eta}\right)} A\left(u, \frac{t}{\abs{u}}\right)
    \frac{t^{m-1}}{\abs{u}^m}\,dt \, d\hat{\eta}\,du. \\
    \intertext{Make the change of variables $r=\frac{t}{\abs{u}}$:}
    &= C \int_{\abs{u} \ge \frac{1}{2\pi}} \int_{\hat{\eta} \in
      S^{m-1}} \int_{r=0}^{r=\frac{1}{\abs{u}}} \abs{\grad f\left(u,
      r\hat{\eta}\right)} A\left(u, r\right) r^{m-1} \,dr \,
    d\hat{\eta}\,du \\
    &\le C \int_{u \in \R^{2n}} \int_{\hat{\eta} \in
      S^{m-1}} \int_{r=0}^{r=\frac{1}{\abs{u}}} \abs{\grad f\left(u,
      r\hat{\eta}\right)} A\left(u, r\right) r^{m-1} \,dr \,
    d\hat{\eta}\,du \\
    &= C \int_{B}  \abs{\grad f}\,d\haar \label{int-B-gradf} \\
    &\le \frac{C}{\inf_B p_1} \int_B \abs{\grad f} p_1\,d\haar \\
    &\le C' \int_G  \abs{\grad f} p_1\,d\haar. \label{int-G-gradf-p_1}
  \end{align}
  where we have used the fact that $p_1$ is bounded away from $0$ on
  $B$.

  For $I_2$ (see (\ref{I2-def})), we have by Claim \ref{R-claim} that
  \begin{align}
    I_2 &\le C \int_{\abs{u} \ge \frac{1}{2\pi}}
    \int_{\hat{\eta} \in S^{m-1}} 
\int_{s=0}^{s=1} \abs{ f\left(u, \frac{s}{\abs{u}}
        \hat{\eta}\right) - m_f} \frac{s^{m-1}}{\abs{u}^m} A\left(u,
  \frac{s}{\abs{u}}\right)\,ds
    \,d\hat{\eta}\,du.
\intertext{Make the change of variables $r=\frac{s}{\abs{u}}$:}
&= C \int_{\abs{u} \ge \frac{1}{2\pi}}
    \int_{\hat{\eta} \in S^{m-1}} 
\int_{r=0}^{r=\frac{1}{\abs{u}}} \abs{ f\left(u, r\hat{\eta}\right) - m_f} r^{m-1} A\left(u,
  r\right)\,dr
    \,d\hat{\eta}\,du \\
    &\le C \int_{u \in \R^{2n}} \int_{\hat{\eta} \in S^{m-1}} 
    \int_{r=0}^{r=\frac{1}{\abs{u}}} 
    \abs{ f\left(u, r\hat{\eta}\right) - m_f} r^{m-1}
    A\left(u, r\right)\,dr\,d\hat{\eta} du \\
    &= C \int_B \abs{f-m_f}\,d\haar \\
    &\le C \int_B \abs{\grad f}\,d\haar
  \end{align}
  by Theorem \ref{poincare}.  The inequalities
  (\ref{int-B-gradf}--\ref{int-G-gradf-p_1}) now show that $I_2
  \le C' \int_G \abs{\grad f}p_1\,d\haar$, as desired.
\end{proof}

\begin{corollary}\label{cheeger-combined}
  There exists a constant $C$ such that for all $f \in \funcclass$,
  \begin{equation}
    \int_G \abs{f - m_f} p_1 \,d\haar \le C \int_G \abs{\grad f} p_1 \,d\haar.
  \end{equation}
\end{corollary}
\begin{proof}
  Add (\ref{poincare-p_1-eq}) and (\ref{bbbc-52-eq}).
\end{proof}

We can now prove some cases of the desired gradient inequality
(\ref{integral-to-show}).

\begin{notation}
  Let $\cylinder(R) = \{(x,z) : \abs{x} \le R\}$ denote the ``cylinder
  about the $z$ axis'' of radius $R$.
\end{notation}

\begin{lemma}\label{grad-ineq-cyl}
  For fixed $R > 0$, (\ref{integral-to-show}) holds, with a constant
  $C=C(R)$ depending on $R$, for all $f \in \funcclass$ which are
  supported on $\cylinder(R)$ and satisfy $m_f = 0$.
\end{lemma}
\begin{proof}
  \begin{align*}
     \abs{\int_G ((\grad-\hat{\grad})f) p_1}\,d\haar &= \abs{\int_G f
       (\grad-\hat{\grad})p_1}\,d\haar &&\text{by integration by parts (\ref{byparts})} \\
     &\le \int_G \abs{f}
     \abs{(\grad-\hat{\grad})p_1}\,d\haar \\
     &= \int_G \abs{f} \abs{x} \abs{\grad_z p_1} \,d\haar &&
     \text{by (\ref{abs-grad-difference})}\\
       &\le C R \int_G \abs{f} p_1 \,d\haar && \text{by
         (\ref{vertical-gradient-crude}).}
     \intertext{(Note that $\abs{x} \le R$ on the support
       of $f$.)}
       &\le C' R \int_G \abs{\grad f} p_1 \,d\haar && \text{by Corollary \ref{cheeger-combined}}.
  \end{align*}
\end{proof}

\begin{notation}
  If $T : G \to M_{2n \times 2n}$ is a matrix-valued function on $G$,
  with $k \ell$th entry $a_{k \ell}$, let $\nabla \cdot T : G \to
  \R^{2n}$ be defined as
  \begin{equation}
    \nabla \cdot T(g) := \sum_{k,\ell=1}^{2n} X_\ell a_{k \ell}(g) e_k.
  \end{equation}
  Note that for $f : G \to \R$ we have the product formula
  \begin{equation}\label{div-product}
    \nabla \cdot (f T) = T \grad f + f \nabla \cdot T.
  \end{equation}
  \end{notation}

\begin{lemma}\label{grad-ineq-offcyl}
  For fixed $R > 1$, (\ref{integral-to-show}) holds, with a constant
  $C=C(R)$ depending on $R$, for all $f \in \funcclass$ which are
  supported on the complement of $\cylinder(R)$.
\end{lemma}

\begin{proof}
  Applying (\ref{gradient-formula}) we have
  \begin{equation*}
    \grad p_1(x,z) = \grad_x p_1(x,z) + \frac{1}{2}J_{\grad_z p_1(x,z)}x.
  \end{equation*}
  Now $p_1$ is a ``radial'' function (that is, $p_1(x,z)$ depends
  only on $\abs{x}$ and $\abs{z}$).  Thus we have that $\grad_x
  p_1(x,z)$ is a scalar multiple of $x$, and also that $\grad_z
  p_1(x,z)$ is a scalar multiple of $z$, so that $J_{\grad_z
    p_1(x,z)}x$ is a scalar multiple of $J_z x$.  

  For nonzero $x \in \R^{2n}$, let $T(x) \in M_{2n \times 2n}$ be
  orthogonal projection onto the $m$-dimensional subspace of $\R^{2n}$
  spanned by the orthogonal vectors $J_{u_1}x, \dots, J_{u_m}x$.
  (Recall $\inner{J_{u_i}x}{J_{u_j} x} = -\inner{u_i}{u_j} \norm{x}^2
  = -\delta_{ij} \norm{x}^2$.)  Thus for any $z \in \R^m$, $T(x) J_z x
  = J_z x$, and $T(x) x = 0$; in particular,
  \begin{equation}\label{Tgradp}
    T(x) \grad p_1(x,z) =
  \frac{1}{2}J_{\grad_z p_1(x,z)}x = \frac{1}{2}
  (\grad-\hat{\grad})p_1(x,z).
  \end{equation}

  Explicitly, we have
  \begin{equation*}
    T(x) = \frac{1}{\abs{x}^2} \sum_{j=1}^m J_{u_j} x (J_{u_j}x)^{T}.
  \end{equation*}
   Note that $\abs{T(x)} = 1$ (in operator norm) for all $x \ne 0$,
   and a routine computation verifies that $\abs{\nabla \cdot T(x)} =
   \abs{\nabla_x \cdot T(x)} \le \frac{C}{\abs{x}}$.  Indeed, the
   $k\ell$th entry of $T(x)$ is
\begin{equation*}
a_{k\ell}(x) = \frac{1}{\abs{x}^2} \sum_{j=1}^{m}
\inner{J_{u_j}x}{e_k}\inner{J_{u_j}x}{e_\ell}
\end{equation*}
so that $\abs{X_k a_{k\ell}(x)} =
\abs{\frac{\partial}{\partial x^k} a_{k\ell}(x)} \le
\frac{3m}{\abs{x}}$; thus $\abs{\nabla \cdot T(x)} \le
\frac{3m(2n)^2}{\abs{x}}$.

  Since $p_1$ decays rapidly at infinity, we have the integration by
  parts formula
  \begin{equation}\label{div-parts}
    0 = \int_G \nabla \cdot (f p_1T)\,d\haar = \int_G (f p_1 \nabla \cdot T +
    f T \grad p_1 + p_1 T \grad f) \,d\haar.
  \end{equation}
  Thus
  \begin{align*}
    \abs{\int_G ((\grad - \hat{\grad})f) p_1\,d\haar} &= \abs{\int_G f
      (\grad - \hat{\grad})p_1\,d\haar} \\
    &= 2\abs{\int_G f T \grad p_1\,d\haar} \\
    &= 2\abs{\int_G f p_1 (\nabla \cdot T +  T \grad f)\,d\haar} \\
    &\le 2\int_G \abs{f} \abs{\nabla \cdot T} p_1\,d\haar + 2 \int_G
    \abs{T} \abs{\grad f} p_1\,d\haar\\
    &\le \frac{2C}{R} \int_G \abs{f} p_1\,d\haar + 2 \int_G
    \abs{\grad f} p_1\,d\haar
  \end{align*}
  since on the support of $f$, we have $ \abs{\nabla \cdot T} \le
  \frac{C}{\abs{x}} \le \frac{C}{R}$, and $\abs{T} = 1$.  The second integral is the
  desired right side of (\ref{integral-to-show}).  The first integral
  is bounded by the same by Corollary \ref{cheeger-combined}, where we
  note that $m_f = 0$ because $f$ vanishes on $D(R) \supset B$.
\end{proof}

We can now complete the proof of Theorem \ref{main-grad-theorem}.

\begin{proof}[Proof of Theorem \ref{main-grad-theorem}]
  We prove (\ref{integral-to-show}) for general $f \in \funcclass$.
  By replacing $f$ by $f - m_f \in \funcclass$, we can assume $m_f = 0$.

  Let $\psi \in C^\infty(G)$ be a smooth function such that $\psi
  \equiv 1$ on $D(1)$ and 
  $\psi$ is supported in $D(2)$.  Then $f = \psi f + (1-\psi)f$.

  $\psi f$ is supported on $D(2)$, so Lemma \ref{grad-ineq-cyl}
  applies to $\psi f$.  (Note that $m_{\psi f} = 0$ since $\psi \equiv
  1$ on $D(1) \supset B$.)  We have
  \begin{align*}
    \abs{\int_G (\grad - \hat{\grad})(\psi f) p_1\,d\haar} &\le C \int_G
    \abs{\grad(\psi f)} p_1\,d\haar \\
   &\le C \int_G \abs{\grad \psi} \abs{f} p_1 \,d\haar + \int_G \abs{\psi}
    \abs{\grad f} p_1\,d\haar \\
    &\le C \sup_G \abs{\grad \psi} \int_G \abs{f} p_1 \,d\haar + C \sup_G
    \abs{\psi}\int_G \abs{\grad f} p_1\,d\haar.
  \end{align*}
  The second integral is the right side of (\ref{integral-to-show}),
  and the first is bounded by the same by Corollary
  \ref{cheeger-combined}.

  Precisely the same argument applies to $(1-\psi)f$, which is
  supported on the complement of $D(1)$, by using Lemma
  \ref{grad-ineq-offcyl} instead of Lemma \ref{grad-ineq-cyl}.
\end{proof}

\section{The optimal constant $K$}

\index{gradient bounds!behavior of constant|(}
We observed previously that the constant $K$ in (\ref{grad-ineq}) can
be taken to be independent of $t$.  We now show that the
\emph{optimal} constant is also independent of $t > 0$, and is
discontinuous at $t=0$.  This distinguishes the current situation from
the elliptic case, in which the constant is continuous at $t=0$; see,
for instance \cite[Proposition 2.3]{bakry-sobolev}.  This
fact was initially noted for the Heisenberg group in
\cite{driver-melcher}, and the proof here is similar to the one found
there.

\newcommand{\Kopt}{K_{\mathrm{opt}}}

\begin{proposition}
  For $t \ge 0$, let
  \begin{equation}\label{Kopt}
    \Kopt(t) := \sup \left\{ 
    \frac{\abs{(\grad P_t f)(g)}}{P_t (\abs{\grad f})(g)} : f \in \funcclass, g \in G , P_t (\abs{\grad
        f})(g) \ne 0 \right\}.
  \end{equation}
  Then $\Kopt(0) = 1$, and for all $t > 0$, $\Kopt(t) \equiv \Kopt >
  1$ is independent of $t$, so that $\Kopt(t)$ is discontinuous at
  $t=0$.  In particular, \mbox{$\Kopt \ge \sqrt{\frac{3n+5}{3n+1}}$}.
\end{proposition}

\begin{proof}
  It is obvious that $\Kopt(0) = 1$.
  
  As before, by the left invariance of $P_t$ and $\grad$, it suffices
  to take $g=0$ on the right side of (\ref{Kopt}).  To show
  independence of $t > 0$, fix $t, s > 0$.  If $f \in \funcclass$,
  then $\tilde{f} := f \circ \varphi_{s^{1/2}}^{-1} \in \funcclass$
  and $f = \tilde{f} \circ \varphi_{s^{1/2}}$.  Then
  \begin{align*}
    \frac{\abs{(\grad P_t f)(0)}}{P_t (\abs{\grad f})(0)} &= 
    \frac{\abs{(\grad P_t (\tilde{f} \circ
        \varphi_{s^{1/2}}))(0)}}{P_t \left(\abs{\grad (\tilde{f} \circ
        \varphi_{s^{1/2}})}\right)(0)} \\
    &=     \frac{\abs{(\grad (P_{st} \tilde{f}) \circ
        \varphi_{s^{1/2}})(0)}}{P_t \left(s^{1/2} \abs{\grad \tilde{f}} \circ
        \varphi_{s^{1/2}}\right)(0)} \\
    &=     \frac{s^{1/2}\abs{(\grad P_{st}
        \tilde{f})(\varphi_{s^{1/2}}(0))}}
         {s^{1/2} P_{st} \left(\abs{\grad
             \tilde{f}}\right)(\varphi_{s^{1/2}}(0))}
         \le \Kopt(st).
  \end{align*}
  Taking the supremum over $f$ shows that $\Kopt(t) \le \Kopt(st)$.
  $s$ was arbitrary, so $\Kopt(t)$ is constant for $t > 0$.  

  In order to bound the constant, we explicitly compute a related ratio for
  a particular choice of function $f$.  The function used is an obvious
  generalization of the example used in \cite{driver-melcher} for the
  Heisenberg group.

  Fix a unit vector $u_1$ in the center of $G$, i.e. $u_1 \in 0 \times
  \R^m \subset \R^{2n+m}$.  We note that the operator $L$ and the norm
  of the gradient $\abs{\grad f}^2 = \frac{1}{2}(L(f^2) - 2 f L f)$
  are independent of the orthonormal basis $\{e_i\}$ chosen to define
  the vector fields $\{X_i\}$, so without loss of generality we
  suppose that $J_{u_1} e_1 = e_2$.  Then take
  \begin{align*}
    f(x,z) &:= \inner{x}{e_1} + \inner{z}{u_1}\inner{x}{e_2} = x_1 +
    z_1 x_2 \\
    k(t) &:= \frac{\abs{(\grad P_t f)(0)}}{P_t(\abs{\grad f})(0)}.
  \end{align*}
  Note that $k(t) \le \Kopt$ for all $t$.  By the Cauchy-Schwarz
  inequality,
  \begin{equation*}
    k(t)^2 \ge k_2(t) := \frac{\abs{(\grad P_t f)(0)}^2}{P_t\left(\abs{\grad f}^2\right)(0)}.
  \end{equation*}
  Since $f$ is a polynomial, we can compute $P_t f$ by the formula
  $P_t f = f + \frac{t}{1!} L f + \frac{t^2}{2!} L^2 f + \cdots$ since
  the sum terminates after a finite number of terms (specifically,
  two).  The same is true of $\abs{\grad f}^2$, which is also a
  polynomial (three terms are needed).  The formulas
  (\ref{Xi-formula}) are helpful in carrying out this tedious but
  straightforward computation.
  \fixnotme{Consider including the computation here.}
  We find
  \begin{equation*}
    k_2(t) = \frac{(1+t)^2}{1-2t+(3n+2)t^2}
  \end{equation*}
  which, by differentiation, is maximized at $t_{\mathrm{max}} =
  \frac{2}{3n+3}$, with $k_2(t_{\mathrm{max}}) = \frac{3n+5}{3n+1}$.
  Since $\Kopt \ge k(t_{\mathrm{max}}) \ge
  \sqrt{k_2(t_{\mathrm{max}})} = \sqrt{\frac{3n+5}{3n+1}}$, this is
  the desired bound.
\end{proof}
\index{gradient bounds!behavior of constant|)}

\section{Consequences}\label{gradient-consequences-sec}

Section 6 of \cite{bbbc-jfa} gives several important consequences of
the gradient inequality (\ref{grad-ineq}).  The proofs given
there are generic (see their Remark 6.6); with Theorem
\ref{main-grad-theorem} in hand, they go through without change in the
case of H-type groups.  These consequences include:
\begin{itemize}
\item Local Gross-Poincar\'e inequalities, or $\varphi$-Sobolev
  inequalities;
  \index{Gross-Poincar\'e inequalities}
  \index{Sobolev inequalities}
\item Cheeger type inequalities; and
  \index{Cheeger type inequalities}
\item Bobkov type isoperimetric inequalities.
  \index{Bobkov inequalities}
  \index{isoperimetric inequalities}
\end{itemize}
We refer the reader to \cite{bbbc-jfa} for the statements and proofs of
these theorems, and many references as well.

\fixnotme{Say a little about log Sobolev inequalities, and about
  isoperimetric Cheeger inequalities.}

Chapter \ref{gradient-chapter}, in large part, is adapted from
material awaiting publication as \gradcite.  The
dissertation author was the sole author of this paper.

\newcommand{\foo}[1]{\thepage}
\index{recursion|foo}

\nohyphens{\printindex}

\bibliographystyle{plain}
\bibliography{allpapers}

\def\cprime{$'$} \def\cprime{$'$} \def\cprime{$'$} \def\cprime{$'$}
\begin{thebibliography}{1}

\bibitem{eldredge-thesis}
Nathaniel Eldredge.
\newblock {\em Hypoelliptic heat kernel inequalities on {H}-type groups}.
\newblock PhD thesis, {U}niversity of {C}alifornia, {S}an {D}iego, 2009.
\newblock arXiv:1406.1840.

\bibitem{eldredge-precise-estimates}
Nathaniel Eldredge.
\newblock Precise estimates for the subelliptic heat kernel on {$H$}-type
  groups.
\newblock {\em J. Math. Pures Appl. (9)}, 92(1):52--85, 2009.
\newblock arXiv:0810.3218.

\bibitem{eldredge-gradient}
Nathaniel Eldredge.
\newblock Gradient estimates for the subelliptic heat kernel on {$H$}-type
  groups.
\newblock {\em J. Funct. Anal.}, 258(2):504--533, 2010.
\newblock arXiv:0904.1781.

\bibitem{li-gradient-h-type}
Jun-Qi Hu and Hong-Quan Li.
\newblock Gradient estimates for the heat semigroup on {H}-type groups.
\newblock {\em Potential Anal.}, 33(4):355--386, 2010.

\bibitem{koranyi85}
Adam Kor{\'a}nyi.
\newblock Geometric properties of {H}eisenberg-type groups.
\newblock {\em Adv. in Math.}, 56(1):28--38, 1985.

\bibitem{li-heat-h-type}
Hong-Quan Li.
\newblock Estimations optimales du noyau de la chaleur sur les groupes de type
  {H}eisenberg.
\newblock {\em J. Reine Angew. Math.}, 646:195--233, 2010.

\bibitem{rigot}
S{\'e}verine Rigot.
\newblock Counter example to the {B}esicovitch covering property for some
  {C}arnot groups equipped with their {C}arnot-{C}arath\'eodory metric.
\newblock {\em Math. Z.}, 248(4):827--848, 2004.

\bibitem{tan-yang}
Kang-Hai Tan and Xiao-Ping Yang.
\newblock Characterisation of the sub-{R}iemannian isometry groups of
  {$H$}-type groups.
\newblock {\em Bull. Austral. Math. Soc.}, 70(1):87--100, 2004.

\end{thebibliography}


\def\cprime{$'$} \def\cprime{$'$} \def\cprime{$'$} \def\cprime{$'$}
\begin{thebibliography}{10}

\bibitem{abgr-intrinsic}
Andrei Agrachev, Ugo Boscain, Jean-Paul Gauthier, and Francesco Rossi.
\newblock The intrinsic hypoelliptic {L}aplacian and its heat kernel on
  unimodular {L}ie groups.
\newblock {\em J. Funct. Anal.}, 256(8):2621--2655, 2009.

\bibitem{atiyah-bott-shapiro}
M.~F. Atiyah, R.~Bott, and A.~Shapiro.
\newblock Clifford modules.
\newblock {\em Topology}, 3(suppl. 1):3--38, 1964.

\bibitem{bakry-riesz-notes-ii}
D.~Bakry.
\newblock Transformations de {R}iesz pour les semi-groupes sym\'etriques. {II}.
  \'{E}tude sous la condition {$\Gamma\sb 2\geq 0$}.
\newblock In {\em S\'eminaire de probabilit\'es, {XIX}, 1983/84}, volume 1123
  of {\em Lecture Notes in Math.}, pages 145--174. Springer, Berlin, 1985.

\bibitem{bakry-sobolev}
D.~Bakry.
\newblock On {S}obolev and logarithmic {S}obolev inequalities for {M}arkov
  semigroups.
\newblock In {\em New trends in stochastic analysis ({C}haringworth, 1994)},
  pages 43--75. World Sci. Publ., River Edge, NJ, 1997.

\bibitem{bbbc-jfa}
Dominique Bakry, Fabrice Baudoin, Michel Bonnefont, and Djalil Chafa{\"{\i}}.
\newblock On gradient bounds for the heat kernel on the {H}eisenberg group.
\newblock {\em J. Funct. Anal.}, 255(8):1905--1938, 2008.

\bibitem{beals-fundamental-solutions}
Richard Beals.
\newblock A note on fundamental solutions.
\newblock {\em Comm. Partial Differential Equations}, 24(1-2):369--376, 1999.

\bibitem{bgg}
Richard Beals, Bernard Gaveau, and Peter~C. Greiner.
\newblock Hamilton-{J}acobi theory and the heat kernel on {H}eisenberg groups.
\newblock {\em J. Math. Pures Appl. (9)}, 79(7):633--689, 2000.

\bibitem{blu-book}
A.~Bonfiglioli, E.~Lanconelli, and F.~Uguzzoni.
\newblock {\em Stratified {L}ie groups and potential theory for their
  sub-{L}aplacians}.
\newblock Springer Monographs in Mathematics. Springer, Berlin, 2007.

\bibitem{calin-book}
Ovidiu Calin, Der-Chen Chang, and Peter Greiner.
\newblock {\em Geometric analysis on the {H}eisenberg group and its
  generalizations}, volume~40 of {\em AMS/IP Studies in Advanced Mathematics}.
\newblock American Mathematical Society, Providence, RI, 2007.

\bibitem{calin-H-type}
Ovidiu Calin, Der-Chen Chang, and Irina Markina.
\newblock Geometric analysis on {$H$}-type groups related to division algebras.
\newblock Preprint, available online
  {\mbox{http://math.cts.nthu.edu.tw/Mathematics/preprints/prep2007-1-004.pdf}},
  2007.

\bibitem{cygan}
Jacek Cygan.
\newblock Heat kernels for class {$2$} nilpotent groups.
\newblock {\em Studia Math.}, 64(3):227--238, 1979.

\bibitem{davies-pang}
E.~B. Davies and M.~M.~H. Pang.
\newblock Sharp heat kernel bounds for some {L}aplace operators.
\newblock {\em Quart. J. Math. Oxford Ser. (2)}, 40(159):281--290, 1989.

\bibitem{driver-melcher}
Bruce~K. Driver and Tai Melcher.
\newblock Hypoelliptic heat kernel inequalities on the {H}eisenberg group.
\newblock {\em J. Funct. Anal.}, 221(2):340--365, 2005.

\bibitem{eckmann}
Beno Eckmann.
\newblock Gruppentheoretischer {B}eweis des {S}atzes von {H}urwitz-{R}adon
  \"uber die {K}omposition quadratischer {F}ormen.
\newblock {\em Comment. Math. Helv.}, 15:358--366, 1943.

\bibitem{evans-pde-book}
Lawrence~C. Evans.
\newblock {\em Partial differential equations}, volume~19 of {\em Graduate
  Studies in Mathematics}.
\newblock American Mathematical Society, Providence, RI, 1998.

\bibitem{garofalo-segala}
Nicola Garofalo and Fausto Seg{\`a}la.
\newblock Estimates of the fundamental solution and {W}iener's criterion for
  the heat equation on the {H}eisenberg group.
\newblock {\em Indiana Univ. Math. J.}, 39(4):1155--1196, 1990.

\bibitem{gaveau77}
Bernard Gaveau.
\newblock Principe de moindre action, propagation de la chaleur et estim\'ees
  sous elliptiques sur certains groupes nilpotents.
\newblock {\em Acta Math.}, 139(1-2):95--153, 1977.

\bibitem{hormander67}
Lars H{\"o}rmander.
\newblock Hypoelliptic second order differential equations.
\newblock {\em Acta Math.}, 119:147--171, 1967.

\bibitem{hueber-muller}
H.~Hueber and D.~M{\"u}ller.
\newblock Asymptotics for some {G}reen kernels on the {H}eisenberg group and
  the {M}artin boundary.
\newblock {\em Math. Ann.}, 283(1):97--119, 1989.

\bibitem{hulanicki}
A.~Hulanicki.
\newblock The distribution of energy in the {B}rownian motion in the {G}aussian
  field and analytic-hypoellipticity of certain subelliptic operators on the
  {H}eisenberg group.
\newblock {\em Studia Math.}, 56(2):165--173, 1976.

\bibitem{hunt}
G.~A. Hunt.
\newblock Semi-groups of measures on {L}ie groups.
\newblock {\em Transactions of the American Mathematical Society},
  81(2):264--293, March 1956.

\bibitem{jerison}
David Jerison.
\newblock The {P}oincar\'e inequality for vector fields satisfying
  {H}\"ormander's condition.
\newblock {\em Duke Math. J.}, 53(2):503--523, 1986.

\bibitem{jerison-sanchez}
David~S. Jerison and Antonio S{\'a}nchez-Calle.
\newblock Estimates for the heat kernel for a sum of squares of vector fields.
\newblock {\em Indiana Univ. Math. J.}, 35(4):835--854, 1986.

\bibitem{kaplan80}
Aroldo Kaplan.
\newblock Fundamental solutions for a class of hypoelliptic {PDE} generated by
  composition of quadratic forms.
\newblock {\em Trans. Amer. Math. Soc.}, 258(1):147--153, 1980.

\bibitem{klingler}
Andrew Klingler.
\newblock New derivation of the {H}eisenberg kernel.
\newblock {\em Comm. Partial Differential Equations}, 22(11-12):2051--2060,
  1997.

\bibitem{kusuoka-stroock-III}
S.~Kusuoka and D.~Stroock.
\newblock Applications of the {M}alliavin calculus. {III}.
\newblock {\em J. Fac. Sci. Univ. Tokyo Sect. IA Math.}, 34(2):391--442, 1987.

\bibitem{levy50}
Paul L{\'e}vy.
\newblock Wiener's random function, and other {L}aplacian random functions.
\newblock In {\em Proceedings of the Second Berkeley Symposium on Mathematical
  Statistics and Probability, 1950}, pages 171--187, Berkeley and Los Angeles,
  1951. University of California Press.

\bibitem{li-jfa}
Hong-Quan Li.
\newblock Estimation optimale du gradient du semi-groupe de la chaleur sur le
  groupe de {H}eisenberg.
\newblock {\em J. Funct. Anal.}, 236(2):369--394, 2006.

\bibitem{li-heatkernel}
Hong-Quan Li.
\newblock Estimations asymptotiques du noyau de la chaleur sur les groupes de
  {H}eisenberg.
\newblock {\em C. R. Math. Acad. Sci. Paris}, 344(8):497--502, 2007.

\bibitem{luan-zhu}
Jingwen Luan and Fuliu Zhu.
\newblock The heat kernel on the {C}ayley {H}eisenberg group.
\newblock {\em Acta Math. Sci. Ser. B Engl. Ed.}, 25(4):687--702, 2005.

\bibitem{lust-piquard}
Fran{\c{c}}oise Lust-Piquard.
\newblock A simple-minded computation of heat kernels on {H}eisenberg groups.
\newblock {\em Colloq. Math.}, 97(2):233--249, 2003.

\bibitem{magnus}
Wilhelm Magnus, Fritz Oberhettinger, and Raj~Pal Soni.
\newblock {\em Formulas and theorems for the special functions of mathematical
  physics}.
\newblock Third enlarged edition. Die Grundlehren der mathematischen
  Wissenschaften, Band 52. Springer-Verlag New York, Inc., New York, 1966.

\bibitem{maheux-saloffcoste}
P.~Maheux and L.~Saloff-Coste.
\newblock Analyse sur les boules d'un op\'erateur sous-elliptique.
\newblock {\em Math. Ann.}, 303(4):713--740, 1995.

\bibitem{tai-thesis}
Tai~A. Melcher.
\newblock {\em Hypoelliptic heat kernel inequalities on {L}ie groups}.
\newblock PhD thesis, {U}niversity of {C}alifornia, {S}an {D}iego, 2004.

\bibitem{montgomery}
Richard Montgomery.
\newblock {\em A tour of subriemannian geometries, their geodesics and
  applications}, volume~91 of {\em Mathematical Surveys and Monographs}.
\newblock American Mathematical Society, Providence, RI, 2002.

\bibitem{dido}
Henry Purcell.
\newblock {\em Dido and Aeneas}.
\newblock 1689.

\bibitem{randall}
Jennifer Randall.
\newblock The heat kernel for generalized {H}eisenberg groups.
\newblock {\em J. Geom. Anal.}, 6(2):287--316, 1996.

\bibitem{reed-simon-vol1}
Michael Reed and Barry Simon.
\newblock {\em Methods of modern mathematical physics. {I}}.
\newblock Academic Press Inc. [Harcourt Brace Jovanovich Publishers], New York,
  second edition, 1980.

\bibitem{rogawski}
Jon Rogawski.
\newblock {\em Single variable calculus: early transcendentals}.
\newblock W. H. Freeman, New York, 2008.

\bibitem{rothschild-stein}
Linda~Preiss Rothschild and E.~M. Stein.
\newblock Hypoelliptic differential operators and nilpotent groups.
\newblock {\em Acta Math.}, 137(3-4):247--320, 1976.

\bibitem{simon-field-theory}
Barry Simon.
\newblock {\em The {$P(\phi )\sb{2}$} {E}uclidean (quantum) field theory}.
\newblock Princeton University Press, Princeton, N.J., 1974.
\newblock Princeton Series in Physics.

\bibitem{simon-functional-integration}
Barry Simon.
\newblock {\em Functional integration and quantum physics}.
\newblock AMS Chelsea Publishing, Providence, RI, second edition, 2005.

\bibitem{stroock-pde-probabilists}
Daniel~W. Stroock.
\newblock {\em Partial differential equations for probabilists}, volume 112 of
  {\em Cambridge Studies in Advanced Mathematics}.
\newblock Cambridge University Press, Cambridge, 2008.

\bibitem{taylor}
Thomas Taylor.
\newblock A parametrix for step-two hypoelliptic diffusion equations.
\newblock {\em Trans. Amer. Math. Soc.}, 296(1):191--215, 1986.

\bibitem{varopoulos-II}
N.~Th. Varopoulos.
\newblock Small time {G}aussian estimates of heat diffusion kernels. {II}.
  {T}he theory of large deviations.
\newblock {\em J. Funct. Anal.}, 93(1):1--33, 1990.

\bibitem{purplebook}
N.~Th. Varopoulos, L.~Saloff-Coste, and T.~Coulhon.
\newblock {\em Analysis and geometry on groups}, volume 100 of {\em Cambridge
  Tracts in Mathematics}.
\newblock Cambridge University Press, Cambridge, 1992.

\bibitem{aeneid}
Virgil.
\newblock {\em Aeneid}.
\newblock 19 B.C.E.
\newblock Translated by John Dryden. Available online,
  \mbox{http://classics.mit.edu/Virgil/aeneid.html}.

\end{thebibliography}

{\tiny   \texttt{\RCSId}}

\end{document}


\title{Addendum to Dissertation \emph{Hypoelliptic Heat Kernel Inequalities on H-Type Groups}}
\author{Nathaniel Eldredge}
\maketitle

\begin{abstract}
  This note lists some addenda to the material in my dissertation
  \cite{eldredge-thesis}, and by extension the papers
  \cite{eldredge-precise-estimates} and \cite{eldredge-gradient} on
  which it is based, which have arisen since their publication.
\end{abstract}

\section{Related Work}

First, a note that the articles \cite{eldredge-precise-estimates} and
\cite{eldredge-gradient} have now been published.  Citations appear in
the References at the end of this Addendum.

I have learned that the formulas for shortest horizontal paths in H-type groups
\cite[Theorem 3.2.4]{eldredge-thesis}, \cite[Theorem
  3.11]{eldredge-precise-estimates} and the corresponding distance
function \cite[Theorem 3.2.7]{eldredge-thesis}, \cite[Theorem
  3.5]{eldredge-precise-estimates}, have previously appeared in the
literature.  Specifically, \cite{koranyi85} gave a formula for
subriemannian geodesics in H-type groups, which \cite[Section
  5]{rigot} and \cite{tan-yang} used to determine the shortest paths
and distance formulas.

Precise heat kernel estimates for H-type groups as in \cite[Theorems
  4.1.1 and 4.1.3]{eldredge-thesis}, \cite[Theorems 4.2 and
  4.4]{eldredge-precise-estimates} have also been obtained
  independently in \cite{li-heat-h-type}.  Gradient estimates as in
  \cite[Theorem 5.1.2]{eldredge-thesis}, \cite[Theorem
  2.4]{eldredge-gradient} have also been obtained independently in
  \cite{li-gradient-h-type}.

\section{Additional Notes}

In \cite[Section 4.5]{eldredge-thesis} and \cite[Section
  7]{eldredge-precise-estimates}, a Hadamard descent technique is used
  to estimate the heat kernel $p^{(n,m)}$ where $m$ is even, where
  $2n+m$ is the dimension of the H-type group $G$ with center of
  dimension $m$ (we shall say it has dimensions $(2n,m)$).  It relies
  upon the fact that the case of odd $m$ is previously handled, and
  uses an estimate on $p^{(n,m+1)}$.  We remarked that it is not
  necessary that there exist an H-type group with the corresponding
  dimensions, since our analysis is entirely in terms of the explicit
  heat kernel formula \cite[(2.4.10)]{eldredge-thesis},
  \cite[(4.5)]{eldredge-precise-estimates}, valid for all values of
  $m,n$.  Indeed, such an H-type group may not exist: here is an
  example.  We know from the Hurwitz-Radon-Eckmann theorem
  \cite[Theorem 2.2.6]{eldredge-thesis}, \cite[Theorem
  2.3]{eldredge-precise-estimates} that an H-type group with
  dimensions $(2n,m)$ exists iff $m < \rho(2n)$, where $\rho(k) =
  8p+2^q$ if $k = a 2^{4p+q}$ with $a$ odd and $0 \le q \le 3$.  Take
  $n=8$.  Then $\rho(2n) = 9$, since $2n=16=1 \cdot 2^{4+0}$.
  Therefore, if $m=8$, there is an H-type group with dimensions
  $(2n,m)=(16,8)$, but none with dimensions $(2n,m+1)=(16,9)$.  More
  generally, this situation will occur iff $n = 2^{4p-1}$ and $m=8p$
  for some integer $p \ge 1$.

\section{Errata}

\begin{itemize}
\item In Example 2.1.10, the assertion that ``if the $a_i$ are not all
  equal, then there is no inner product on $\mathfrak{g} = \Lie G$
  under which it is an H-type Lie algebra'' is incorrect.  It is not
  an H-type Lie algebra under the inner product which makes $\{X_1,
  \dots, X_{2n}, Z\}$ orthonormal, but I made the error of forgetting
  that $J_Z$ depends implicitly on the choice of inner product.  If we
  take an inner product $\inner{\cdot}{\cdot}$ such that all $\{X_i\}$
  are orthogonal to each other and to $Z$, $\inner{X_{2i-1}}{X_{2i-1}}
  = \inner{X_{2i}}{X_{2i}} = \abs{a_i}$, and $\inner{Z}{Z}=1$, then
  $\mathfrak{g}$ will be H-type.  Indeed, we will have $J_Z X_{2i-1} =
  \sgn(a_i) X_{2i}$ and $J_Z X_{2i} = -\sgn(a_i) X_{2i-1}$, so that
  clearly $J_Z^2 = -I$.  (Note of course that all $a_i$ must be
  nonzero; otherwise $X_{2i}$ is in the center and $J_{X_{2i}} = 0$.)


  For a better example of a step $2$ stratified Lie algebra which is
  not H-type under any inner product, consider the $4$-dimensional Lie
  algebra $\mathfrak{g}$ spanned by $\{X_1, X_2, X_3, Z\}$, with
  bracket defined by
  \begin{equation*}
    [X_1, X_2] = [X_2, X_3] = [X_3, X_1] = Z, \quad [X_i, Z] = 0.
  \end{equation*}
  The center $\mathfrak{z}$ of $\mathfrak{g}$ has dimension 1 (it is
  spanned by $Z$), so under any inner product
  $\dim(\mathfrak{z}^\perp) = 3$.  But we argued in Proposition 2.1.7
  (item 10) that in any H-type Lie algebra, $\dim(\mathfrak{z}^\perp)$
  must be even.
  
\item In the discussion following Theorem 3.1.4, there is a reference
to ``the energy $\frac{1}{2}\int_0^1 \norm{\dot{\gamma}(t)}\,dt$''.
This should be ``the energy $\frac{1}{2}\int_0^1
\norm{\dot{\gamma}(t)}^2\,dt$''.

\item Equation (3.1.2) is missing a factor of $\frac{1}{2}$ and should read
\begin{equation*}
H(p_q) = \frac{1}{2} \sum_i p_q(v_i)^2
\end{equation*}

\item In Corollary 4.1.2, equation (4.1.2) is missing minus signs on some of the
  powers of $t$.  It should read
  \begin{equation*}
  p_t(x,z) \asymp t^{-m-n} \frac{1 +
    (t^{-1/2}d_0(x,z))^{2n-m-1}}{1+(t^{-1} \abs{x}d_0(x,z))^{n-\frac{1}{2}}}
    e^{-\frac{1}{4t}d_0(x,z)^2}
  \end{equation*}

\item In (4.3.11) and (4.3.39), $p_t$ should be $p_1$.
\end{itemize}

\bibliographystyle{plain}
\bibliography{allpapers}